\documentclass[a4paper,11pt]{amsart}

\usepackage{amssymb}
\usepackage{mathrsfs}
\usepackage{cite}
\usepackage{url}
\usepackage[all]{xy}

\newtheorem{theorem}{Theorem}[section]
\newtheorem{lemma}[theorem]{Lemma}
\newtheorem{corollary}[theorem]{Corollary}
\newtheorem{proposition}[theorem]{Proposition}
\newtheorem{conjecture}[theorem]{Conjecture}

\theoremstyle{definition}
\newtheorem{definition}[theorem]{Definition}
\newtheorem{remark}[theorem]{Remark}

\numberwithin{equation}{section}

\makeatletter
\@namedef{subjclassname@2020}{%
  \textup{2020} Mathematics Subject Classification}
\makeatother

\DeclareMathOperator{\RE}{Re}
\DeclareMathOperator{\IM}{Im}
\DeclareMathOperator{\tr}{tr}

\begin{document}

\title[Riemann zeta-function and random matrices]
{Connection between the Riemann zeta-function and random matrices via hyperfunctions}

\author[M. Mine]{Masahiro Mine}
\address{Global Education Center, Waseda University, 1-6-1 Nishiwaseda, Shinjuku-ku, Tokyo 169-8050, Japan}
\email{m-mine@aoni.waseda.jp}

\date{}

\begin{abstract}
Bohr pioneered the study of the statistical behavior of the Riemann zeta-function. 
A classical result by Bohr and Jessen revealed that the values of the Riemann zeta-function to the right of the critical line behave like a random variable. 
We now propose to extend Bohr's theory to the stage of hyperfunctions. 
In this paper, we introduce two random hyperfunctions: one is associated with the values of the Riemann zeta-function on the critical line, and the other is associated with the characteristic polynomial of a random matrix from the circular unitary ensemble. 
We then derive a relationship between these random hyperfunctions which is consistent with the Keating--Snaith conjecture on the moments of the Riemann zeta-function. 
\end{abstract}



\maketitle

\section{Introduction}\label{sec:1}
The theory of the statistical behavior of the Riemann zeta-function $\zeta(s)$ was established by Bohr and his colleagues in the early 20th century. 
For example, Bohr showed in \cite{Bohr1916} that the set of the values of $\log{\zeta}(s)$ on the line $\RE(s)=\sigma$ is dense in the complex plane $\mathbb{C}$ if $\sigma$ is fixed with $1/2<\sigma \leq1$. 
It should be mentioned here how to choose the branch of the logarithm, but we will postpone the details for now. 
Then, according to the celebrated result by Bohr and Jessen \cite{BohrJessen1930, BohrJessen1932}, there exists a Borel probability measure $\mu_{\sigma}$ on $\mathbb{C}$ for any fixed $\sigma>1/2$ such that
\begin{align}\label{eq:032114543}
\lim_{T \to\infty}
\frac{1}{T}
\int_{0}^{T}
F \big(
\log{\zeta}(\sigma+it)
\big)
\,dt
=
\int_{\mathbb{C}}
F
\,d \mu_{\sigma}
\end{align}
for all bounded continuous functions $F$ on $\mathbb{C}$. 
The limit measure $\mu_{\sigma}$ can be explicitly constructed; we will focus on it later in this section. 
The statistical behavior of $\zeta(s)$ on the critical line $\RE(s)=1/2$ has also been investigated, but many issues remain unresolved. 
Indeed, it is still an open problem whether the set of the values of $\log{\zeta}(s)$ on $\RE(s)=1/2$ is dense in $\mathbb{C}$. 
As for limit formula \eqref{eq:032114543}, we obtain some analogous result with a normalization factor. 
Denote by $\gamma$ the standard complex Gaussian measure on $\mathbb{C}$. 
Then we have
\begin{align*}
\lim_{T \to \infty}
\frac{1}{T}
\int_{0}^{T}
F \bigg(
\frac{\log{\zeta}(1/2+it)}{\sqrt{\log\log{T}}}
\bigg)
\,dt
=
\int_{\mathbb{C}}
F
\,d \gamma
\end{align*}
for all bounded continuous functions $F$ on $\mathbb{C}$. 
This is known as the Selberg central limit theorem, and a complete proof can be found in \cite{Joyner1986} for example. 
Without any normalization factor, a conjecture of the limit formula has been proposed when $F$ is an exponential function. 
Throughout this paper, we write
\begin{align*}
\vec{\kappa} \circ \lambda
=
\kappa_{1} \lambda
+ 
\kappa_{2} \overline{\lambda}
\end{align*}
for $\vec{\kappa}=(\kappa_{1},\kappa_{2}) \in \mathbb{C}^{2}$ and $\lambda \in \mathbb{C}$. 
Therefore, for example, we have 
\begin{align*}
|\zeta(1/2+it)|^{2k}
=
\exp \big(
(k,k) \circ \log{\zeta}(1/2+it)
\big)
\end{align*} 
for $k \in \mathbb{C}$. 
The following conjecture is essentially due to Keating and Snaith \cite{KeatingSnaith2000a}. 

\begin{conjecture}\label{conj:1.1}
For any $\vec{\kappa}=(\kappa_{1},\kappa_{2}) \in \mathbb{C}^{2}$ with $\RE(\kappa_{1}+\kappa_{2})>-1$, 
\begin{align}\label{eq:03211910}
\lim_{T \to\infty}
\frac{1}{T(\log{T})^{\kappa_{1} \kappa_{2}}}
\int_{0}^{T} 
\exp \big(
\vec{\kappa} \circ \log{\zeta}(1/2+it)
\big)
\,dt
=
a(\vec{\kappa})\, g(\vec{\kappa}),
\end{align}
where $a(\vec{\kappa})$ and $g(\vec{\kappa})$ are represented as
\begin{align*}
a(\vec{\kappa})
&
= 
\prod_{p}
\bigg\{
\big(1-p^{-1}\big)^{\kappa_{1} \kappa_{2}}
\sum_{m=0}^{\infty}
\frac{\Gamma(\kappa_{1}+m)}{m!\, \Gamma(\kappa_{1})}
\frac{\Gamma(\kappa_{2}+m)}{m!\, \Gamma(\kappa_{2})}
p^{-m}
\bigg\}, 
\\
g(\vec{\kappa})
&
=
\frac{G(1+\kappa_{1})\, G(1+\kappa_{2})}{G(1+\kappa_{1}+\kappa_{2})}
\end{align*}
using the Barnes $G$-function $G(z)$. 
\end{conjecture}

It is known that \eqref{eq:03211910} holds for $\vec{\kappa}=(0,0)$ obviously, for $\vec{\kappa}=(1,1)$ by Hardy and Littlewood \cite{HardyLittlewood1916}, and for $\vec{\kappa}=(2,2)$ by Ingham \cite{Ingham1927}. 
Furthermore, the cases for $\vec{\kappa}=(1,2), (2,1)$ are also discussed in \cite{CaoTanigawaZhai2021, Pliego2025}. 
These appear to be all that is known about the truth of Conjecture \ref{conj:1.1}. 
Then, we recall that the constant $g(\vec{\kappa})$ has a background derived from random matrix theory. 
For a unitary matrix $U \in \mathcal{U}(N)$, we denote its characteristic polynomial by
\begin{align*}
Z(s;U)
=
\det(I-s\, U)
=
\prod_{n=1}^{N}
\big(
1-s\, e^{i \theta_{n}}
\big)
\end{align*}
for $s \in \mathbb{C}$, where $e^{i \theta_{1}},\ldots,e^{i \theta_{N}}$ are the eigenvalues of $U$. 
Let $s=e^{-i \theta}$ be any fixed point on the unit circle. 
By the method of Keating and Snaith \cite{KeatingSnaith2000a}, we obtain
\begin{align*}
\lim_{N \to \infty}
\frac{1}{N^{\kappa_{1} \kappa_{2}}}
\int_{\mathcal{U}(N)}
\exp \big(
\vec{\kappa} \circ \log{Z}(e^{-i \theta};U)
\big)
\,dU
=
g(\vec{\kappa})
\end{align*}
for any $\vec{\kappa}=(\kappa_{1},\kappa_{2}) \in \mathbb{C}^{2}$ with $\RE(\kappa_{1}+\kappa_{2})>-1$, where $dU$ denotes the normalized Haar measure of the compact group $\mathcal{U}(N)$ so that the total measure is equal to $1$. 
This implies that \eqref{eq:03211910} is equivalent to 
\begin{align}\label{eq:05240210}
\lim_{T \to \infty}
\frac{
\displaystyle{
\frac{1}{T}
\int_{0}^{T} 
\exp \big(
\vec{\kappa} \circ \log{\zeta}(1/2+it)
\big)
\,dt
}}{
\displaystyle{
\int_{\mathcal{U}(N_{T})}
\exp \big(
\vec{\kappa} \circ \log{Z}(e^{-i \theta};U)
\big)
\,dU
}}
=
a(\vec{\kappa}), 
\end{align}
where $N_{T}$ is a positive integer satisfying $N_{T}/\log{T} \to 1$ as $T \to \infty$. 
In this paper, we prove a similar limit formula that may provide evidence to support the truth of Conjecture \ref{conj:1.1}. 
Before proceeding to the statement of the main result, we describe the strict definition of the logarithms of $\zeta(s)$ and $Z(s;U)$. 
Let $\log$ denote the usual branch of the logarithm defined on $\mathbb{C} \setminus (-\infty,0]$ so that
\begin{align*}
-\log(1-z)
= 
\sum_{k=1}^{\infty} \frac{1}{k}\, z^{k}
\end{align*}
holds for $|z| \leq1$ with $z \neq 1$. 
Denote by $\mathcal{G}$ the open set in the complex plane obtained from the half-plane $\RE(s)>0$ 
by leaving out the segment $(0,1]$ and all segments $\{ x+i \IM(\rho) \mid 0<x \leq \RE(\rho) \}$, where $\rho$ denote the nontrivial zeros of $\zeta(s)$. 
Then $\mathcal{G}$ is simply connected, and $\zeta(s)$ is nonvanishing and holomorphic on $\mathcal{G}$. 
Hence we can construct a branch of the logarithm of $\zeta(s)$ on the domain $\mathcal{G}$ by letting
\begin{align*}
\log{\zeta}(s)
=
\log{\frac{\pi^{2}}{6}}
+
\int_{2}^{s}
\frac{\zeta'}{\zeta}(z)
\,dz
\end{align*}
for each $s \in \mathcal{G}$.  
We know that it is a holomorphic function on $\mathcal{G}$ and satisfies
\begin{align}\label{eq:05081554}
\log{\zeta}(s)
=
- \sum_{p} \log(1-p^{-s})
=
\sum_{p} \sum_{k=1}^{\infty} \frac{1}{k}\, p^{-ks} 
\end{align}
for $s \in \mathbb{C}$ with $\RE(s)>1$. 
It should be noted that $\log{\zeta}(\sigma+it)$ has a discontinuity at $t=0$ if $\sigma$ is fixed with $0<\sigma<1$. 
To avoid this, we shall modify the branch cut on $\mathbb{R}$. 
Denote by $\mathcal{G}^{*}$ the open set in $\mathbb{C}$ obtained from the half-plane $\RE(s)>0$ by leaving out the half-line $[1,\infty)$ and all segments $\{ x+i \IM(\rho) \mid 0<x \leq \RE(\rho) \}$, where $\rho$ denote the nontrivial zeros of $\zeta(s)$. 
Then $\mathcal{G}^{*}$ is also simply connected, and $\zeta(s)$ is nonvanishing and holomorphic on $\mathcal{G}^{*}$. 
We construct another branch of the logarithm of $\zeta(s)$ on the domain $\mathcal{G}^{*}$ by letting
\begin{align*}
\log^{*}{\zeta}(s)
=
\log{\zeta}(2+i)
+
\int_{2+i}^{s}
\frac{\zeta'}{\zeta}(z)
\,dz 
\end{align*}
for each $s \in \mathcal{G}^{*}$, where $\log{\zeta}(2+i)$ is defined by \eqref{eq:05081554}. 
We also know that $\log^{*}{\zeta}(s)$ is a holomorphic function on $\mathcal{G}^{*}$. 
In particular, it is holomorphic on the vertical strip $1/2<\RE(s)<1$ if we assume the Riemann Hypothesis (RH), that is, all nontrivial zeros of $\zeta(s)$ lie on $\RE(s)=1/2$. 
There is a simple relation between these branches of the logarithm such that
\begin{align}\label{eq:05152038}
\log^{*}{\zeta}(s)
=
\begin{cases}
\log{\zeta}(s)
&
\text{for $s \in \mathcal{G}^{*}$ with $\IM(s)>0$},
\\
\log{\zeta}(s)-2\pi i 
&
\text{for $s \in \mathcal{G}^{*}$ with $\IM(s)<0$},  
\end{cases}
\end{align}
and $\log^{*}{\zeta}(\sigma)=\lim_{\epsilon \downarrow 0} \log{\zeta}(\sigma+i \epsilon)$ for $0<\sigma<1$. 
Then, the logarithm of $Z(s;U)$ is defined as 
\begin{align}\label{eq:05091441}
\log{Z}(s;U)
=
\sum_{n=1}^{N}
\log \big(
1-s\, e^{i \theta_{n}}
\big)
=
\sum_{k=1}^{\infty} \frac{\tr(U^{k})}{k} s^{k} 
\end{align}
for $|s| \leq 1$ with $s \neq e^{-i \theta_{n}}$ for all $1 \leq n \leq N$. 
Since $|\tr(U^k)| \leq N$ for $U \in \mathcal{U}(N)$, this power series converges absolutely for $|s|<1$. 
Therefore $\log{Z}(s;U)$ is a holomorphic function on $|s|<1$. 
To present the statement of the result, we lastly define 
\begin{align}\label{eq:05170402}
f_{L}(z)
=
\sqrt{\frac{L}{2\pi}}\exp \Big(-\frac{Lz^2}{2}\Big)
\end{align}
for $z \in \mathbb{C}$ and $L>0$, which is an entire function rapidly decreasing as $|\RE(z)| \to \infty$ on the horizontal strip $a<\IM(z)<b$ for any fixed $a,b \in \mathbb{R}$.  
We will later see that $f_{L}(x) \log^{*}{\zeta}(1/2+ix+it)$ and $f_{L}(x) \log{Z}(e^{-ix};U)$ are integrable over $\mathbb{R}$. 

\begin{theorem}\label{thm:1.2}
Assume RH. 
Let $N_{T}$ be any positive integer such that $N_{T} \to \infty$ as $T \to \infty$. 
For any $\vec{\kappa}=(\kappa_{1},\kappa_{2}) \in \mathbb{C}^{2}$, 
\begin{align}\label{eq:03222030}
\lim_{L \to \infty}
\lim_{T \to \infty}
\frac{
\displaystyle{
\frac{1}{T}
\int_{0}^{T} 
\exp \Big(
\vec{\kappa} \circ
\int_{-\infty}^{\infty} 
f_{L}(x) \log^{*}{\zeta}(1/2+ix+it)
\,dx
\Big)
\,dt
}}{
\displaystyle{
\int_{\mathcal{U}(N_{T})}
\exp \Big(
\vec{\kappa} \circ
\int_{-\infty}^{\infty} 
f_{L}(x) \log{Z}(e^{-ix};U)
\,dx
\Big)
\,dU
}}
=
a(\vec{\kappa}), 
\end{align}
where $a(\vec{\kappa})$ is the same as in Conjecture \ref{conj:1.1}. 
Furthermore, the followings hold when $N_{T}/\log{T} \to 1$ as $T \to \infty$. 
\begin{itemize}
\item[$(\mathrm{i})$]
For any $\vec{\kappa}=(\kappa_{1},\kappa_{2}) \in \mathbb{C}^{2}$ with $\RE(\kappa_{1}+\kappa_{2}) \geq0$, limit formula  \eqref{eq:03211910} is valid if and only if the order of the limits in \eqref{eq:03222030} is commutative. 
\item[$(\mathrm{ii})$]
Assume further that all nontrivial zeros of $\zeta(s)$ are simple. 
Then, for any $\vec{\kappa}=(\kappa_{1},\kappa_{2}) \in \mathbb{C}^{2}$ with $-1<\RE(\kappa_{1}+\kappa_{2})<0$, limit formula \eqref{eq:03211910} is valid if and only if the order of the limits in \eqref{eq:03222030} is commutative. 
\end{itemize}

\end{theorem}

Recall that a modern method for the proof of limit formula \eqref{eq:032114543} is randomizing the values of $\zeta(s)$ for $\RE(s)>1/2$. 
Let $(\mathsf{X}(p))_{\text{$p$ prime}}$ be a sequence of independent random variables uniformly distributed on the unit circle. 
Furthermore, we extend this to all positive integers $n$ by complete multiplicativity. 
By the Kronecker–Weyl equidistribution theorem, $\mathsf{X}(n)$ is known to be a good model of $X_{\tau}(n):=n^{-i \tau}$ as $\tau$ varies over $\mathbb{R}$. 
Recalling that $\zeta(s+i \tau)$ is represented as
\begin{align*}
\zeta(s+i \tau)
= 
\sum_{n=1}^{\infty}
X_{\tau}(n) n^{-s}
=
\prod_{p}
\big(1- X_{\tau}(p) p^{-s}\big)^{-1} 
\end{align*}
for $\RE(s)>1$, we may expect that it behaves like the random analytic function 
\begin{align}\label{eq:03221850}
\zeta^{\mathsf{rand}}(s)
= 
\sum_{n=1}^{\infty}
\mathsf{X}(n) n^{-s}
=
\prod_{p}
\big(1-\mathsf{X}(p)p^{-s}\big)^{-1}. 
\end{align} 
In fact, it is known that the limit measure $\mu_{\sigma}$ in \eqref{eq:032114543} is obtained as the distribution of the random variable $\log{\zeta}^{\mathsf{rand}}(\sigma)$ for any fixed $\sigma>1/2$, and hence, 
\begin{align}\label{eq:05051441}
\lim_{T \to\infty}
\frac{1}{T}
\int_{0}^{T}
F \big(
\log{\zeta}(\sigma+it)
\big)
\,dt
=
\mathbb{E} \big[
F \big(
\log{\zeta}^{\mathsf{rand}}(\sigma)
\big)
\big]
\end{align}
for all bounded continuous functions $F$ on $\mathbb{C}$, where $\mathbb{E}[X]$ stands for the expectation of a random variable $X$. 
Using the same idea, we can further study the statistical behavior of $\zeta(s+i \tau)$ for $\tau \in \mathbb{R}$ in the space of holomorphic functions on the strip $1/2<\RE(s)<1$. 
This is related to the universality property of $\zeta(s)$, which means that any function with some suitable conditions is approximated by $\zeta(s+i \tau)$ for some $\tau \in \mathbb{R}$. 
See Bagchi \cite{Bagchi1981} for details. 

Then, we outline the strategy for the proof of Theorem \ref{thm:1.2}. 
We aim to randomize the values of $\zeta(s)$ on the critical line, but the series and product of \eqref{eq:03221850} are indeed convergent only for $\RE(s)>1/2$ almost surely. 
Moreover, we know that the critical line $\RE(s)=1/2$ is the natural boundary of the random analytic function $\zeta^{\mathsf{rand}}(s)$ almost surely; see \cite[p.~90]{Bagchi1981}. 
As it is difficult to assign meaning to $\zeta^{\mathsf{rand}}(1/2+ix)$ as random function values, we present it as boundary values of a random analytic function in the sense of hyperfunctions introduced by Sato \cite{Sato1959, Sato1960}. 
More precisely, we consider the randomization of $\log^{*}{\zeta}(s)$ in the space of Fourier hyperfunctions studied by Kawai \cite{Kawai1970}. 
As an analogue of \eqref{eq:05051441}, we will show that
\begin{align*}
\lim_{T \to \infty}
\frac{1}{T}
\int_{0}^{T} 
F \big(
\log^{*}{\zeta}(1/2+0+ix+it)
\big)
\,dt
=
\mathbb{E} \big[
F \big(
\log{\zeta}^{\mathsf{rand}}(1/2+0+ix)
\big)
\big]
\end{align*}
for all bounded continuous functions $F$ on the space of Fourier hyperfunctions, where the meanings of $\log^{*}{\zeta}(1/2+0+ix+it)$ and $\log{\zeta}^{\mathsf{rand}}(1/2+0+ix)$ are explained in Section \ref{sec:2.2}. 
Furthermore, we will show that
\begin{align*}
\lim_{N \to \infty}
\int_{\mathcal{U}(N)}
F \big(
\log{Z}(e^{-ix-0};U)
\big)
\,dU
=
\mathbb{E} \big[
F \big(
\log{Z}^{\mathsf{Gauss}}(e^{-ix-0})
\big)
\big], 
\end{align*}
where $\log{Z}^{\mathsf{Gauss}}(e^{-ix-0})$ is a certain random hyperfunction defined in Section \ref{sec:2.3}. 
Then limit formula \eqref{eq:03222030} follows by relaxing the condition of the test function $F$ and showing a relation between $\log{\zeta}^{\mathsf{rand}}(1/2+0+ix)$ and $\log{Z}^{\mathsf{Gauss}}(e^{-ix-0})$ such that 
\begin{align}\label{eq:03251633}
\lim_{L \to \infty}
\frac{
\displaystyle{
\mathbb{E} \Big[
\exp \Big(
\vec{\kappa} \circ
\big\langle
-f_{L}
\mid
\log{\zeta}^{\mathsf{rand}}(1/2+0+ix)
\big\rangle
\Big)
\Big]
}}{
\displaystyle{
\mathbb{E} \Big[
\exp \Big(
\vec{\kappa} \circ
\big\langle
-f_{L}
\mid
\log{Z}^{\mathsf{Gauss}}(e^{-ix-0})
\big\rangle
\Big)
\Big]
}}
=
a(\vec{\kappa}), 
\end{align}
where $\langle f \mid g \rangle$ is a certain paring of a function $f$ and a hyperfunction $g$. 
Finally, the connection to \eqref{eq:03211910} follows by applying the fact that $f_{L}(x)$ approaches to the Dirac delta on $\mathbb{R}$ as $L \to\infty$. 

\begin{remark}\label{rem:1.3}
Saksman and Webb \cite{SaksmanWebb2020} used another concept of generalized functions for studying the statistical behavior of $\zeta(s)$ on the critical line. 
They considered the randomization of $\zeta(s)$ in the $L^{2}$-based Sobolev space $W^{-\alpha,2}(\mathbb{R})$ for $\alpha>1/2$, which consists of all Schwartz distributions $g \in \mathcal{S}'(\mathbb{R})$ such that the Fourier transform $\widehat{g}(\xi)$ is locally $L^{2}$-integrable and satisfies
\begin{align*}
\int_{\mathbb{R}} 
(1+\xi^{2})^{-\alpha} |\widehat{g}(\xi)|^{2}
\,d \xi
<
\infty. 
\end{align*}
Then $\zeta^{\mathsf{rand}}(1/2+ix)$ is presented as a random Schwartz distribution, and several results on the statistical behavior of $\zeta(1/2+ix+it)$ in the space $W^{-\alpha,2}(\mathbb{R})$ were obtained in \cite{SaksmanWebb2020}. 
In particular, it was shown that the same Gaussian multiplicative chaos distribution appears in both mesoscopic behaviors of $\zeta(1/2+i \delta_{T} x+it)$ and $Z(e^{-i\delta_{N} x};U)$ with some $\delta_{T},\delta_{N}>0$ such that $\delta_{T},\delta_{N} \to 0$ as $T,N \to \infty$. 
This might be related to Theorem \ref{thm:1.2} in some way, because we have
\begin{align*}
\int_{-\infty}^{\infty} 
f_{L}(x) \log^{*}{\zeta}(1/2+ix+it)
\,dx
&
=
\int_{-\infty}^{\infty} 
f_{1}(x) \log^{*}{\zeta}\Big(1/2+\frac{ix}{\sqrt{L}}+it \Big)
\,dx, 
\\
\int_{-\infty}^{\infty} 
f_{L}(x) \log{Z}(e^{-ix};U)
\,dx
&
=
\int_{-\infty}^{\infty} 
f_{1}(x) \log{Z}\big(e^{-\frac{ix}{\sqrt{L}}};U\big)
\,dx 
\end{align*}
by changing the variables. 
The key novelty of Theorem \ref{thm:1.2} is that it yields the same constant $a(\vec{\kappa})$ as in Conjecture \ref{conj:1.1}. 
\end{remark}

\begin{remark}\label{rem:1.4}
Keating and Snaith conjectured in \cite{KeatingSnaith2000b} that the distributions of the values of $L$-functions in several families at the central point $s=1/2$ are also related to random matrices distributed on the groups $\mathrm{SO}(2N)$ and $\mathrm{USp}(2N)$. 
It may be reasonable to extend the results of this paper to such cases, but we do not discuss those issues and leave them for future consideration. 
\end{remark}

\subsection*{Organization of the paper}
\begin{itemize}
\item
Section \ref{sec:2} is devoted to preparing the notion of random hyperfunctions used in this paper. 
After reviewing the basics of Fourier hyperfunctions of one variable, we then introduce certain random hyperfunctions associated with the Riemann zeta-function and the characteristic polynomial of a random matrix from the circular unitary ensemble. 
\item
Section \ref{sec:3} contains the proof for a limit theorem which extends the result of Bohr and Jessen \cite{BohrJessen1930, BohrJessen1932} to the space of Fourier hyperfunctions. 
We prove this in terms of the theory of the convergence of random elements according to the method refined by Kowalski \cite{Kowalski2021}. 
Furthermore, we investigate a random matrix analogue of this limit theorem. 
\item
Section \ref{sec:4} completes the proof of Theorem \ref{thm:1.2}. 
The main ingredient is limit formula \eqref{eq:03251633}, which is proved by the local analysis at each prime $p$ and an application of the prime number theorem. 
Finally, we conclude this paper with discussions on the approximation of the Dirac delta on $\mathbb{R}$. 
\end{itemize}

\section{Certain random hyperfunctions}\label{sec:2}

\subsection{Basics on Fourier hyperfunctions}\label{sec:2.1}
The notion of Fourier hyperfunctions was introduce by Sato in his Japanese article \cite[p.23--25]{Sato1958c}. 
Kawai \cite{Kawai1970} developed the general theory of Fourier hyperfunctions and applied it to the study of partial differential equations. 
Here, we review the basics on Fourier hyperfunctions of one variable according to Sato \cite{Sato1958c} since we do not deal with multivariate cases. 
We will provide full proofs of the necessary results in Appendix for convenience. 

Let $D_{\delta}$ be the strip $|\IM(z)|<\delta$ on the complex plane for $\delta>0$. 
We define $\mathfrak{P}_{\delta}$ as the linear space of holomorphic functions $f$ on $D_{\delta}$ such that
\begin{align}\label{eq:04091459}
\sup_{|\IM(z)| \leq \delta_{1}}
|f(z)| e^{-\epsilon |z|}
<
\infty
\end{align}
for any $\epsilon>0$ and $0<\delta_{1}<\delta$. 
We also define $\widetilde{\mathfrak{P}}_{\delta}$ as the linear space of holomorphic functions $\varphi$ on $D_{\delta} \setminus \mathbb{R}$ such that
\begin{align}\label{eq:04091500}
\sup_{\delta_{1} \leq |\IM(z)| \leq \delta_{2}}
|\varphi(z)| e^{-\epsilon |z|}
<
\infty
\end{align}
for any $\epsilon>0$ and $0<\delta_{1}<\delta_{2}<\delta$. 
Note that the linear map $\mathfrak{P}_{\delta} \to \widetilde{\mathfrak{P}}_{\delta}$ defined by the restriction $f \mapsto f \mid_{D_{\delta}\setminus \mathbb{R}}$ is injective by the identity theorem. 
We regard $\mathfrak{P}_{\delta}$ as a subspace of $\widetilde{\mathfrak{P}}_{\delta}$ and define $\mathfrak{Q}_{\delta}=\widetilde{\mathfrak{P}}_{\delta}/\mathfrak{P}_{\delta}$. 
Then we define $\mathfrak{Q}$ as the inductive limit
\begin{align*}
\mathfrak{Q}
=
\varinjlim
\mathfrak{Q}_{\delta} 
\end{align*}
with respect to the linear map $\rho_{\delta,\delta'}: \mathfrak{Q}_{\delta} \to \mathfrak{Q}_{\delta'}$  induced from the restriction maps $\mathfrak{P}_{\delta} \to \mathfrak{P}_{\delta'}$ and $\widetilde{\mathfrak{P}}_{\delta} \to \widetilde{\mathfrak{P}}_{\delta'}$ for $0<\delta'<\delta$. 
We refer to the elements of the space $\mathfrak{Q}$ as \textit{Fourier hyperfunctions}. 

\begin{lemma}[$=$ Lemma \ref{lem:A.1}]\label{lem:2.1}
The linear map $\rho_{\delta,\delta'}: \mathfrak{Q}_{\delta} \to \mathfrak{Q}_{\delta'}$ is bijective for any $0<\delta'<\delta$. 
\end{lemma}

Hence $\mathfrak{Q}$ is identified with $\mathfrak{Q}_{\delta}$ for any $\delta>0$, and any Fourier hyperfunction $g$ is determined from a function $\varphi \in \widetilde{\mathfrak{P}}_{\delta}$ with some $\delta>0$. 
We write 
\begin{align}\label{eq:05052304}
g
=
[\varphi] 
\end{align}
to indicate that $g$ is determined from $\varphi$. 
Then we see that both $g=[\varphi_{1}]$ and $g=[\varphi_{2}]$ hold for $\varphi_{1} \in \widetilde{\mathfrak{P}}_{\delta}$ and $\varphi_{2} \in \widetilde{\mathfrak{P}}_{\delta'}$ if and only if there exist $0<\delta'' \leq \min(\delta,\delta')$ and $f \in \mathfrak{P}_{\delta''}$ such that
\begin{align*}
\varphi_{1}(z)
=
\varphi_{2}(z)+f(z)
\end{align*}
for any $z \in D_{\delta''} \setminus \mathbb{R}$. 
Let $\varepsilon(z)$ and $\overline{\varepsilon}(z)$ denote the functions
\begin{align*}
\varepsilon(z)
=
\begin{cases}
1
&
\text{for $0<\IM(z)<\delta$}, 
\\
0
&
\text{for $-\delta<\IM(z)<0$}, 
\end{cases}
\quad\text{and}\quad
\overline{\varepsilon}(z)
=
\begin{cases}
0
&
\text{for $0<\IM(z)<\delta$}, 
\\
1
&
\text{for $-\delta<\IM(z)<0$}, 
\end{cases}
\end{align*}
which obviously belong to the space $\widetilde{\mathfrak{P}}_{\delta}$ for any $\delta>0$. 
For a function $\varphi \in \widetilde{\mathfrak{P}}_{\delta}$, we define its boundary values $\varphi(x+i0), \varphi(x-i0) \in \mathfrak{Q}$ as
\begin{align}\label{eq:05071345}
\varphi(x+i0)
= 
[\varepsilon \varphi]
\quad\text{and}\quad
\varphi(x-i0)
= 
-[\overline{\varepsilon} \varphi]. 
\end{align}
Using these notations, we can rewrite \eqref{eq:05052304} as
\begin{align*}
g(x)
=
\varphi(x+i0)-\varphi(x-i0), 
\end{align*}
which helps us interpret the notion of hyperfunctions as the difference of boundary values of analytic functions. 
Lastly, the map $\mathfrak{P}_{\delta} \to \mathfrak{Q}$ defined by $f \mapsto f(x+i0)=f(x-i0)$ is injective by definition, and we regard $\mathfrak{P}_{\delta}$ as a subspace of $\mathfrak{Q}$. 
We say that a Fourier hyperfunction $g$ is regular if it belongs to $\mathfrak{P}_{\delta}$ with some $\delta>0$. 

Then, we investigate the duality property of the linear space $\mathfrak{Q}$. 
For every integer $m \geq1$, we define $\mathfrak{B}_{m}^{*}$ as the linear space of continuous functions $f$ on the strip $|\IM(z)| \leq 2^{-m}$ which are holomorphic in $|\IM(z)|<2^{-m}$ and satisfy
\begin{align*}
\sup_{|\IM(z)| \leq 2^{-m}}
|f(z)| e^{|z|/m}
<
\infty. 
\end{align*}
Then $\mathfrak{B}_{m}^{*}$ is a Banach space with the norm $\| f \|_{m}=\sup_{|\IM(z)| \leq 2^{-m}} |f(z)| e^{|z|/m}$ for every $m \geq1$. 
We define $\mathfrak{P}^{*}$ as the inductive limit
\begin{align}\label{eq:05062302}
\mathfrak{P}^{*}
=
\varinjlim
\mathfrak{B}_{m}^{*}
\end{align}
with respect to the restriction map $\mathfrak{B}_{m}^{*} \to \mathfrak{B}_{m+1}^{*}$. 
Since this restriction is injective by the identity theorem, we can regard $\mathfrak{B}_{m}^{*}$ as a subspace of $\mathfrak{B}_{m+1}^{*}$ for all $m \geq1$. 
Then $\mathfrak{P}^{*}$ is identified with the union of all $\mathfrak{B}_{m}^{*}$, and thus, $f$ belongs to $\mathfrak{P}^{*}$ if and only if there exists $\delta>0$ such that $f$ is holomorphic on the strip $D_{\delta}$ and satisfies 
\begin{align}\label{eq:05062315}
\sup_{|\IM(z)| \leq \delta_{1}}
|f(z)| e^{\epsilon |z|}
<
\infty
\end{align}
for any $0<\delta_{1}<\delta$ with some $\epsilon=\epsilon(\delta_{1})>0$. 
The linear space $\mathfrak{P}^{*}$ is equipped with the inductive limit topology according to \eqref{eq:05062302}. 
Let $f \in \mathfrak{P}^{*}$ and $g \in \mathfrak{Q}$. 
Then we can take $\delta,\delta'>0$ such that $f$ is holomorphic on the strip $D_{\delta}$ and satisfies \eqref{eq:05062315}, and that $g=[\varphi]$ with $\varphi \in \widetilde{\mathfrak{P}}_{\delta'}$. 
From the above, we shall define the pairing 
\begin{align*}
\langle f \mid g \rangle
=
\int_{-\infty+ic}^{\infty+ic}
f(z) \varphi(z)
\,dz
-
\int_{-\infty-ic}^{\infty-ic}
f(z) \varphi(z)
\,dz, 
\end{align*}
where $c$ is any constant with $0<c<\min(\delta_{1},\delta_{2})$. 
The integrals of the right-hand side are convergent and independent of $c$ due to \eqref{eq:04091500} and \eqref{eq:05062315}. 
Furthermore, the difference of them does not depend on the choice of the function $\varphi$ since 
\begin{align*}
\int_{-\infty+ic}^{\infty+ic}
f(z) \psi(z)
\,dz
=
\int_{-\infty-ic}^{\infty-ic}
f(z) \psi(z)
\,dz
\end{align*}
for $\psi \in \mathfrak{P}_{\delta}$ due to \eqref{eq:04091459} and \eqref{eq:05062315}, where $c$ is any constant with $0<c<\delta$. 
As a result, the pairing $\langle f \mid g \rangle$ is well-defined for any $f \in \mathfrak{P}^{*}$ and $g \in \mathfrak{Q}$. 
Furthermore, the map $T_{g}: \mathfrak{P}^{*} \to \mathbb{C}$ defined for any fixed $g \in \mathfrak{Q}$ by letting $T_{g}(f)= \langle f \mid g \rangle$ is a continuous linear functional on $\mathfrak{P}^{*}$. 

\begin{proposition}[$=$ Proposition \ref{prop:A.2}]\label{prop:2.2}
The linear map $j: \mathfrak{Q} \to (\mathfrak{P}^{*})'$ defined by $g \mapsto T_{g}$ is bijective. 
\end{proposition}

Recall that the dual space $(\mathfrak{B}_{m}^{*})'$ is also a Banach space for every $m \geq1$ with the norm $\| T \|_{m}'=\sup_{\| f \|_{m} \leq 1} |T(f)|$. 
By \eqref{eq:05062302} and Proposition \ref{prop:2.2}, we obtain 
\begin{align}\label{eq:04092050}
\mathfrak{Q}
\simeq
\varprojlim
(\mathfrak{B}_{m}^{*})'
\end{align}
as linear spaces, where the canonical map $\phi_{m}:\mathfrak{Q} \to (\mathfrak{B}_{m}^{*})'$ is presented by letting $\phi_{m}(g)= T_{g} \mid_{\mathfrak{B}_{m}^{*}}$ for every $m \geq1$. 
We endow $\mathfrak{Q}$ with the projective limit topology according to \eqref{eq:04092050}, which is a locally convex topology induced by the seminorms
\begin{align*}
\|g\|_{m}'
:=
\| \phi_{m}(g) \|_{m}'
=
\sup_{\| f \|_{m} \leq 1}
\big|
\langle f \mid g \rangle
\big|. 
\end{align*}
Then the space $\mathfrak{Q}$ is metrizable by the metric defined as
\begin{align}\label{eq:04170251}
d(g_{1},g_{2})
=
\sum_{m=1}^{\infty}
2^{-m}
\frac{\| g_{1}-g_{2} \|_{m}'}{1+\| g_{1}-g_{2} \|_{m}'}
\end{align}
for $g_{1},g_{2} \in \mathfrak{Q}$. 
In particular, $g_{n}$ converges to $g$ as $n \to \infty$ in $\mathfrak{Q}$ if and only if 
\begin{align*}
\| g_{n}-g \|_{m}'
=
\sup_{\| f \|_{m} \leq 1}
\big|
\langle f \mid g_{n}-g \rangle
\big|
\to 
0
\end{align*}
as $n \to \infty$ for all $m \geq1$. 

\begin{remark}\label{rem:2.3}
The linear spaces $\mathfrak{P}_{\delta}$ and $\widetilde{\mathfrak{P}}_{\delta}$ are also equipped with projective limit  topologies according to
\begin{align*}
\mathfrak{P}_{\delta}
=
\varprojlim
\mathfrak{B}_{\delta,m}
\quad\text{and}\quad
\widetilde{\mathfrak{P}}_{\delta}
=
\varprojlim
\widetilde{\mathfrak{B}}_{\delta,m}, 
\end{align*}
where $\mathfrak{B}_{\delta,m}$ and $\widetilde{\mathfrak{B}}_{\delta,m}$ are suitable Banach spaces. 
Then the isomorphism $\mathfrak{Q} \simeq (\mathfrak{P}^{*})'$ can be shown as  topological linear spaces by endowing $\mathfrak{Q}$ with the quotient topology with respect to the natural surjection $\widetilde{\mathfrak{P}}_{\delta} \to \mathfrak{Q}$; see \cite[Section 3.2]{Kawai1970}. 
In other words, the above projective limit topology according to \eqref{eq:04092050} is the same as this intentional topology of $\mathfrak{Q}$. 
\end{remark}

\subsection{Random hyperfunctions for the Riemann zeta-function}\label{sec:2.2}
Let $(\Omega,\mathcal{F},\mathbb{P})$ be a probability space. 
Denote by $\mathcal{B}(\mathfrak{Q})$ the Borel $\sigma$-algebra of $\mathfrak{Q}$ with respect to the topology described in Section \ref{sec:2.1}. 
We refer to a measurable map 
\begin{align*}
X: (\Omega,\mathcal{F}) \to (\mathfrak{O},\mathcal{B}(\mathfrak{Q}))
\end{align*}
as a \textit{random (Fourier) hyperfunction}. 
The purpose of this section is to introduce certain random hyperfunctions related to the Riemann zeta-function $\zeta(s)$. 
To begin with, we recall that the number of zeros $\rho=\beta+i \gamma$ of $\zeta(s)$ such that $|y-T| \leq 1$, say $m(T)$, satisfies
\begin{align}\label{eq:05081444}
m(T)
\ll
\log(|T|+2)
\end{align}
with an absolute implied constant. 
Furthermore, we obtain
\begin{align}\label{eq:05081445}
\frac{\zeta'}{\zeta}(z)
=
\sum_{|\gamma-y| \leq1}
\frac{1}{z-\rho}
+
\frac{1}{1-z}
+
O \big(\log(|y|+2)\big)
\end{align}
uniformly for $z=x+iy \in \mathbb{C}$ with $1/2 \leq x \leq 2$. 
See \cite[Proposition 5.7 (1),(2)]{IwaniecKowalski2004} for example. 
By the definition of $\log^{*}{\zeta}(s)$, we obtain
\begin{align}\label{eq:05081506}
\log^{*}{\zeta}(s)
=
\log{\zeta}(2+it)
+
\int_{2}^{\sigma}
\frac{\zeta'}{\zeta}(x+it)
\,dx
+
2\pi i\, \chi_{-}(t) 
\end{align}
for $s=\sigma+it \in \mathcal{G}^{*}$ with $\sigma \geq1/2$ and $t \neq0$, where we put $\chi_{-}(t)=0$ for $t>0$ and $\chi_{-}(t)=-1$ for $t<0$. 
Therefore, applying \eqref{eq:05081444} and \eqref{eq:05081445}, we derive
\begin{align}\label{eq:06011532}
\log^{*}{\zeta}(s)
=
\sum_{|\gamma-t| \leq1}
\log(s-\rho)
+
\log(1-s)
+
O \big(\log(|t|+2)\big)
\end{align}
for $s=\sigma+it \in \mathcal{G}^{*}$ with $\sigma \geq1/2$ and $t \neq0$, where $\log$ denotes the usual branch of the logarithm defined on $\mathbb{C} \setminus (-\infty,0]$, and the implied constant is absolute. 
Note that this remains true for $s=\sigma$ with $1/2 \leq \sigma<1$ by the continuity. 

\begin{lemma}\label{lem:2.4}
Assume RH. 
Let $\sigma_{1}$ and $\sigma_{2}$ be fixed with $1/2<\sigma_{1}<\sigma_{2}<1$. 
Then 
\begin{align*}
\sup_{\sigma_{1} \leq \sigma \leq \sigma_{2}}
\big|\log^{*}{\zeta}(\sigma+it)\big|
=
O \big(\log(|t|+2)\big)
\end{align*}
for all $t \in \mathbb{R}$ with the implied constant depending only on $\sigma_{1}$ and $\sigma_{2}$. 
\end{lemma}

\begin{proof}
Since the result obviously holds for $|t| \leq1$, we assume $|t|>1$ in what follows. 
Then we have $\log(1-s) \ll \log(|t|+2)$. 
For any $\sigma_{1} \leq \sigma \leq \sigma_{2}$, the distance between $\sigma+it$ and the zeros of $\zeta(s)$ is greater than $\sigma_{1}-1/2$ assuming RH. 
Hence, by \eqref{eq:05081444}, the estimate
\begin{align*}
\sum_{|\gamma-t| \leq1}
\log(\sigma+it-\rho)
\ll
\log(|t|+2)
\end{align*}
holds with the implied constant depending only on $\sigma_{1}$. 
Therefore the desired result follows from \eqref{eq:06011532}. 
\end{proof}

Recall that $\log^{*}{\zeta}(s)$ is a holomorphic function on the strip $1/2<\RE(s)<1$ if we assume RH. 
For every $t \in \mathbb{R}$, we then define a holomorphic function on $D_{1/2} \setminus \mathbb{R}$ by
\begin{align*}
\varphi_{t}(z)
=
\begin{cases}
0
&
\text{for $0<\IM(z)<1/2$}, 
\\
\log^{*}{\zeta}(1/2+iz+it)
&
\text{for $-1/2<\IM(z)<0$},  
\end{cases}
\end{align*}
which belongs to the space $\widetilde{\mathfrak{P}}_{1/2}$ by Lemma \ref{lem:2.4}. 
Furthermore, we define 
\begin{align*}
\log^{*}{\zeta}(1/2+0+ix+it)
=
\varphi_{t}(x-i0) 
\end{align*}
using the notation as in \eqref{eq:05071345}. 
Let $T$ be a positive real number, and take a random variable $\upsilon$ uniformly distributed on the interval $[0,1]$. 
Then we see that
\begin{align*}
\log{\zeta}_{T}^{\mathsf{unif}}(1/2+0+ix)
:=
\log^{*}{\zeta}(1/2+0+ix+iT \upsilon) 
\end{align*}
is a random hyperfunctio by the following lemma. 

\begin{lemma}\label{lem:2.5}
Assume RH. 
The map $\mathbb{R} \to \mathfrak{Q}$ defined by $t \mapsto \log^{*}{\zeta}(1/2+0+ix+it)$ is continuous. 
\end{lemma}

\begin{proof}
Let $(t_{n})$ be any sequence of real numbers such that $t_{n} \to t$ as $n \to \infty$. 
By the definition of the topology of $\mathfrak{Q}$, if we obtain
\begin{align}\label{eq:05081533}
\sup_{\| f \|_{m} \leq 1}
\Big|
\big\langle f \mid \log^{*}{\zeta}(1/2+0+ix+it_{n})-\log^{*}{\zeta}(1/2+0+ix+it) \big\rangle
\Big|
\to 
0
\end{align}
as $n \to \infty$ for all $m \geq1$, then we deduce that $\log^{*}{\zeta}(1/2+0+ix+it_{n})$ converges to $\log^{*}{\zeta}(1/2+0+ix+it)$ as $n \to \infty$ in $\mathfrak{Q}$. 
Let $f \in \mathfrak{B}_{m}^{*}$ with $\|f\|_{m} \leq 1$. 
Then we have $|f(z)| \leq e^{-|z|/m}$ for $|\IM(z)| \leq 2^{-m}$. 
Using this, we derive
\begin{align*}
&
\big\langle f \mid \log^{*}{\zeta}(1/2+0+ix+it_{n})-\log^{*}{\zeta}(1/2+0+ix+it) \big\rangle
\\
&
=
-\int_{-\infty-ic}^{\infty-ic}
f(z)
\big\{\log^{*}{\zeta}(1/2+iz+it_{n})-\log^{*}{\zeta}(1/2+iz+it)\big\}
\,dz
\\
&
\ll
\int_{-\infty}^{\infty}
e^{-|x-ic|/m}
\big|\log^{*}{\zeta}(1/2+c+ix+it_{n})-\log^{*}{\zeta}(1/2+c+ix+it)\big|
\,dx
\end{align*}
for $0<c<2^{-m}$, where the implied constant is absolute. 
Note that $|t_{n}|, |t| \leq K$ for all $n \geq1$ with some constant $K>0$. 
Then we apply Lemma \ref{lem:2.4} to derive
\begin{align}\label{eq:05090053}
\log^{*}{\zeta}(1/2+c+ix+it_{n})-\log^{*}{\zeta}(1/2+c+ix+it)
\ll
\log(|x|+K+2)
\end{align}
with the implied constant depending only on $c$. 
By the continuity of $\log^{*}{\zeta}(s)$, 
\begin{align*}
\log^{*}{\zeta}(1/2+iz+it_{n})
\to 
\log^{*}{\zeta}(1/2+iz+it)
\end{align*}
as $n \to \infty$ on the line $\IM(z)=-c$. 
Finally, the dominated convergence theorem yields \eqref{eq:05081533} by noting that $e^{-|x-ic|/m} \log(|x|+K+2)$ is integrable over $\mathbb{R}$. 
\end{proof}

Take a sequence of independent random variables $\mathsf{X}(p)$ uniformly distributed on the unit circle. 
We assume that they are defined on a probability space $(\Omega,\mathcal{F},\mathbb{P})$ in what follows. 
According to \eqref{eq:05081554}, we define
\begin{align*}
\log{\zeta}^{\mathsf{rand}}(s)
=
\sum_{p} \sum_{k=1}^{\infty} \frac{1}{k}\, \mathsf{X}(p)^{k} p^{-ks}. 
\end{align*}
We shall define the random hyperfunction $\log{\zeta}^{\mathsf{rand}}(1/2+0+ix)$ as follows. 

\begin{lemma}\label{lem:2.6} 
Almost surely, we have the followings. 
\begin{enumerate}
\item \label{lem:2.6_1}
$\log{\zeta}^{\mathsf{rand}}(s)$ is a holomorphic function on the half-plane $\RE(s)>1/2$. 
\item \label{lem:2.6_2}
Let $\sigma_{1}$ and $\sigma_{2}$ be fixed with $1/2<\sigma_{1}<\sigma_{2}<\infty$. 
Then there exists a positive real valued random variable $\mathsf{S}$ determined only from $\sigma_{1}$ such that 
\begin{align*}
\sup_{\sigma_{1} \leq \sigma \leq \sigma_{2}}
\big|\log{\zeta}^{\mathsf{rand}}(\sigma+it) \big|
=
O \big((|t|+1) \mathsf{S}\big)
\end{align*}
for all $t \in \mathbb{R}$ with the implied constant depending only on $\sigma_{1}$ and $\sigma_{2}$. 
\end{enumerate}
\end{lemma}

\begin{proof}
Picking up the terms for $k=1$ in the inner sum of $\log{\zeta}^{\mathsf{rand}}(s)$, we have 
\begin{align*}
\log{\zeta}^{\mathsf{rand}}(s)
&
=
\sum_{p} \mathsf{X}(p) p^{-s} 
+
\sum_{p} \sum_{k=2}^{\infty} \frac{1}{k}\, \mathsf{X}(p)^{k} p^{-ks}
\\
&
=
\mathsf{P}(s)+\mathsf{Q}(s), 
\end{align*}
say. 
For all samples of $\Omega$, the latter Dirichlet series $\mathsf{Q}(s)$ converges absolutely and uniformly on any compact set in the half-plane $\RE(s)>1/2$. 
Hence we derive that $\mathsf{Q}(s)$ is a holomorphic function on $\RE(s)>1/2$ which satisfies
\begin{align*}
\sup_{\sigma_{1} \leq \sigma \leq \sigma_{2}}
\big|\mathsf{Q}(\sigma+it) \big|
\leq
\sum_{p} \sum_{k=2}^{\infty} \frac{1}{k}\, p^{-k \sigma_{1}}
=
O(1)
\end{align*}
for all $t \in \mathbb{R}$ with the implied constant depending only on $\sigma_{1}$. 
Then, we consider the convergence for $\mathsf{P}(s)$. 
For each $\eta>0$, the random variables $\mathsf{Y}_{\eta}(p):=\mathsf{X}(p) p^{-1/2-\eta}$ are independent for prime numbers $p$. 
They satisfy $\mathbb{E}\big[ \mathsf{Y}_{\eta}(p) \big]=0$ and 
\begin{align*}
\sum_{p} \mathbb{E}\big[ |\mathsf{Y}_{\eta}(p)|^2 \big]
=
\sum_{p} p^{-1-2\eta}
<
\infty. 
\end{align*}
Thus the series $\sum_{p} \mathsf{Y}_{\eta}(p)$ converges almost surely by applying \cite[Theorem B.10.1]{Kowalski2021}. 
In other words, we can take a set $\Omega_{0}(\eta) \in \mathcal{F}$ with $\mathbb{P}(\Omega_{0}(\eta))=1$ such that $\mathsf{P}(s)$ is convergent at $s=1/2+\eta$ for any sample of $\Omega_{0}(\eta)$. 
Then the set
\begin{align*}
\Omega_{0}
:=
\bigcap_{\eta>0} \Omega_{0}(\eta) 
\end{align*}
satisfies $\Omega_{0} \in \mathcal{F}$ and $\mathbb{P}(\Omega_{0})=1$. 
For any sample of $\Omega_{0}$, the Dirichlet series $\mathsf{P}(s)$ converges absolutely and uniformly on any compact set in $\RE(s)>1/2$. 
Therefore, almost surely, $\mathsf{P}(s)$ is a holomorphic function on the half-plane $\RE(s)>1/2$. 
Lastly, we prove an almost sure upper bound of $\big|\mathsf{P}(\sigma+it) \big|$ for $\sigma \in [\sigma_{1},\sigma_{2}]$ and $t \in \mathbb{R}$. 
In what follows, we consider a sample of the above set $\Omega_{0}$. 
Put $\sigma_{0}=\frac{1}{2}(\sigma_{1}+\frac{1}{2})$ so that $\frac{1}{2}<\sigma_{0}<\sigma_{1}$ is satisfied. 
By summation by parts, we obtain
\begin{align*}
\mathsf{P}(\sigma+it)
=
(\sigma+it-\sigma_{0})
\int_{2}^{\infty}
\mathsf{S}(u) u^{\sigma_{0}-\sigma-it-1}
\,du, 
\end{align*}
where $\mathsf{S}(u)$ is a random variable defined as
\begin{align*}
\mathsf{S}(u)
=
\sum_{p \leq u} \mathsf{X}(p) p^{-\sigma_{0}}. 
\end{align*}
By the definition of $\Omega_{0}$, it follows that $\mathsf{S}(u)$ converges as $u \to \infty$ due to $\sigma_{0}>1/2$. 
Hence we see that
\begin{align*}
\mathsf{S}
=
\sup_{u \geq1} \big|\mathsf{S}(u)\big|+1 
\end{align*}
is a finite value. 
Furthermore, by letting $\mathsf{S}=1$ for any sample of the null set $(\Omega_{0})^{c}$, we can define 
\begin{align*}
\mathsf{S}: \Omega \to [1,\infty)
\end{align*}
as a random variable. 
As a result, we have almost surely 
\begin{align*}
\sup_{\sigma_{1} \leq \sigma \leq \sigma_{2}}
\big|\mathsf{P}(\sigma+it) \big|
\leq
(|t|+\sigma_{2}-\sigma_{0})
\int_{1}^{\infty}
|\mathsf{S}(u)\big| u^{\sigma_{0}-\sigma_{1}-1}
\,du
=
O \big((|t|+1) \mathsf{S}\big)
\end{align*}
for all $t \in \mathbb{R}$ with the implied constant depending only on $\sigma_{1}$ and $\sigma_{2}$. 
From the above, the desired properties for $\log{\zeta}^{\mathsf{rand}}(s)$ are valid almost surely. 
\end{proof}

By Lemma \ref{lem:2.6}, we can define almost surely 
\begin{align*}
\varphi^{\mathsf{rand}}(z)
=
\begin{cases}
0
&
\text{for $0<\IM(z)<1/2$}, 
\\
\log{\zeta}^{\mathsf{rand}}(1/2+iz)
&
\text{for $-1/2<\IM(z)<0$} 
\end{cases}
\end{align*}
as an element of the space $\widetilde{\mathfrak{P}}_{1/2}$. 
Furthermore, we define almost surely
\begin{align*}
\log{\zeta}^{\mathsf{rand}}(1/2+0+ix)
=
\varphi^{\mathsf{rand}}(x-i0) 
\end{align*}
using the notation as in \eqref{eq:05071345}. 
By letting $\log{\zeta}^{\mathsf{rand}}(1/2+0+ix)=0$ for samples of a nullset $N \in \mathcal{F}$ where $\varphi^{\mathsf{rand}}(z)$ is not defined, we obtain a map
\begin{align*}
\log{\zeta}^{\mathsf{rand}}(1/2+0+ix): \Omega \to \mathfrak{Q}, 
\end{align*}
which is indeed a random hyperfunction. 
In order to confirm this, we define
\begin{align*}
\log{\zeta}_{p}(s,w)
=
\sum_{k=1}^{\infty} \frac{1}{k}\, w^{k} p^{-ks} 
\end{align*}
for each prime number $p$, where $w$ is any complex number with $|w| \leq1$.  
Then the sum of the right-hand side converges absolutely and uniformly on any compact set in the half-plane $\RE(s)>1/2$. 
It can be easily seen that
\begin{align*}
\varphi_{p}(z,w)
=
\begin{cases}
0
&
\text{for $0<\IM(z)<1/2$}, 
\\
\log{\zeta}_{p}(1/2+iz,w)
&
\text{for $-1/2<\IM(z)<0$}, 
\end{cases}
\end{align*}
is an element of the space $\widetilde{\mathfrak{P}}_{1/2}$, and we then define 
\begin{align*}
\log{\zeta}_{p}(1/2+0+ix,w)
=
\varphi_{p}(x-i0,w)
\in \mathfrak{Q}. 
\end{align*} 
Similarly to the proof of Lemma \ref{lem:2.5}, we can confirm that the map $\{|\omega| \leq 1\} \to \mathfrak{Q}$ defined by $w \mapsto \log{\zeta}_{p}(1/2+0+ix,w)$ is continuous, and hence, 
\begin{align}\label{eq:05202054}
\log{\zeta}_{p}^{\mathsf{rand}}(1/2+0+ix)
:=
\log{\zeta}_{p}(1/2+0+ix,\mathsf{X}(p))
\end{align}
is a random hyperfunction for every prime number $p$. 

\begin{lemma}\label{lem:2.7}
Almost surely, we have 
\begin{align}\label{eq:05090141}
\sum_{p \leq y}
\log{\zeta}_{p}^{\mathsf{rand}}(1/2+0+ix)
\to
\log{\zeta}^{\mathsf{rand}}(1/2+0+ix)
\end{align}
as $y \to \infty$ in the space $\mathfrak{Q}$. 
\end{lemma}

\begin{proof}
Let $\mathsf{P}(s)$ and $\mathsf{Q}(s)$ be the same as in the proof of Lemma \ref{lem:2.6}. 
We also define
\begin{align*}
\mathsf{P}_{y}(s)
=
\sum_{p \leq y} \mathsf{X}(p) p^{-s} 
\quad\text{and}\quad
\mathsf{Q}_{y}(s)
=
\sum_{p \leq y} \sum_{k=2}^{\infty} \frac{1}{k}\, \mathsf{X}(p)^{k} p^{-ks}. 
\end{align*}
As seen in the proof of Lemma \ref{lem:2.6}, we can take a set $\Omega_{0} \in \mathcal{F}$ with $\mathbb{P}(\Omega_{0})=1$ such that $\mathsf{P}_{y}(s) \to \mathsf{P}(s)$ and $\mathsf{Q}_{y}(s) \to \mathsf{Q}(s)$ as $y \to \infty$ on the half-plane $\RE(s)>1/2$ for any sample of $\Omega_{0}$. 
Note that we have
\begin{align*}
\sum_{p \leq y}
\log{\zeta}_{p}^{\mathsf{rand}}(1/2+0+ix)
&
=
\mathsf{P}_{y}(1/2+0+ix)+\mathsf{Q}_{y}(1/2+0+ix), 
\\
\log{\zeta}^{\mathsf{rand}}(1/2+0+ix)
&
=
\mathsf{P}(1/2+0+ix)+\mathsf{Q}(1/2+0+ix)
\end{align*}
for any sample of $\Omega_{0}$, where the Fourier hyperfunctions on the right-hand sides are defined similarly by using the functions $\mathsf{P}_{y}(s)$, $\mathsf{Q}_{y}(s)$, $\mathsf{P}(s)$, and $\mathsf{Q}(s)$. 
By the definition of the topology of $\mathfrak{Q}$, the result follows if we obtain almost surely 
\begin{align*}
&
\sup_{\| f \|_{m} \leq 1}
\Big|
\big\langle f \mid \mathsf{P}_{y}(1/2+0+ix)-\mathsf{P}(1/2+0+ix) \big\rangle
\Big|
\to 
0,
\\
&
\sup_{\| f \|_{m} \leq 1}
\Big|
\big\langle f \mid \mathsf{Q}_{y}(1/2+0+ix)-\mathsf{Q}(1/2+0+ix) \big\rangle
\Big|
\to 
0
\end{align*}
as $y \to \infty$ for all $m \geq1$. 
These can be shown by the method using the dominated convergence theorem as in the proof of Lemma \ref{lem:2.5}. 
Since the arguments are quite similar, we will only prove here the following almost sure upper bounds
\begin{align*}
\mathsf{P}_{y}(1/2+c+ix)-\mathsf{P}(1/2+c+ix)
&
\ll
(|x|+1) \mathsf{S}, 
\\
\mathsf{Q}_{y}(1/2+c+ix)-\mathsf{Q}(1/2+c+ix)
&
\ll
1,
\end{align*}
which correspond to \eqref{eq:05090053}. 
Let $\mathsf{S}(u)$ and $\mathsf{S}$ be the same random variables as in the proof of Lemma \ref{lem:2.6}. 
Furthermore, we consider a sample of $\Omega_{0}$ in what follows. 
Put $\sigma_{0}=\frac{1}{2}+\frac{c}{2}$ so that $\frac{1}{2}<\sigma_{0}<\frac{1}{2}+c$ is satisfied. 
By summation by parts, we obtain 
\begin{align*}
\mathsf{P}_{y}(1/2+c+ix)-\mathsf{P}(1/2+c+ix)
=
-\Big(\frac{c}{2}+ix \Big) 
\int_{y}^{\infty}
\mathsf{S}(u) u^{-c/2-ix-1}
\,du. 
\end{align*}
This is evaluated as
\begin{align*}
&
\big| \mathsf{P}_{y}(1/2+c+ix)-\mathsf{P}(1/2+c+ix) \big|
\\
&
\leq 
\Big( |x|+\frac{c}{2} \Big) 
\int_{1}^{\infty}
|\mathsf{S}(u)\big| u^{-c/2-1}
\,du
=
O \big((|x|+1) \mathsf{S}\big)
\end{align*}
for all $x \in \mathbb{R}$ with the implied constant depending only on $c$. 
Furthermore, we also obtain that
\begin{align*}
&
\big| \mathsf{Q}_{y}(1/2+c+ix)-\mathsf{Q}(1/2+c+ix) \big|
\\
&
=
\bigg|\sum_{p>y} \sum_{k=2}^{\infty} \frac{1}{k}\, \mathsf{X}(p)^{k} p^{-k(1/2+c+ix)}\bigg|
\leq
\sum_{p} \sum_{k=2}^{\infty} \frac{1}{k}\, p^{-k(1/2+c)}
<
\infty
\end{align*}
for all $x \in \mathbb{R}$. 
Therefore the desired upper bounds are proved. 
\end{proof}

We finally conclude that $\log{\zeta}^{\mathsf{rand}}(1/2+0+ix)$ is a random hyperfunction. 
Take a set $\Omega_{0} \in \mathcal{F}$ with $\mathbb{P}(\Omega_{0})=1$ such that \eqref{eq:05090141} is satisfied for any sample of $\Omega_{0}$. 
Letting $\log{\zeta}^{\mathsf{rand}}(1/2+0+ix)=0$ for any sample of the null set $(\Omega_{0})^{c}$, we have
\begin{align*}
\mathbf{1}_{\Omega_{0}} \cdot
\sum_{p \leq y}
\log{\zeta}_{p}^{\mathsf{rand}}(1/2+0+ix)
\to 
\log{\zeta}^{\mathsf{rand}}(1/2+0+ix)
\end{align*}
as $y \to \infty$ for all samples of $\Omega$, where $\mathbf{1}_{\Omega_{0}}$ denotes the characteristic function of $\Omega_{0}$. 
Note that $\mathbf{1}_{\Omega_{0}}:\Omega \to \{0,1\}$ and every $\log{\zeta}_{p}^{\mathsf{rand}}(1/2+0+ix): \Omega \to \mathfrak{Q}$ are measurable. 
Thus, by the continuities of the sum and scalar multiple of $\mathfrak{Q}$, the map 
\begin{align*}
\mathbf{1}_{\Omega_{0}} \cdot
\sum_{p \leq y}
\log{\zeta}_{p}^{\mathsf{rand}}(1/2+0+ix)
: \Omega \to \mathfrak{Q}
\end{align*}
is measurable for any $y>0$. 
Hence $\log{\zeta}^{\mathsf{rand}}(1/2+0+ix)$ is also measurable since it is obtained as the pointwise limit as $y \to \infty$.

\subsection{Random hyperfunctions for CUE characteristic polynomials}\label{sec:2.3}
In this section, we introduce certain random hyperfunctions related to the characteristic polynomial of a random matrix from the circular unitary ensemble (CUE). 

\begin{lemma}\label{lem:2.8}
Let $\rho_{1}$ and $\rho_{2}$ be fixed with $0<\rho_{1}<\sigma_{2}<1$. 
If $U \in \mathcal{U}(N)$, then 
\begin{align*}
\sup_{\rho_{1} \leq \rho \leq \rho_{2}}
\big|\log{Z}(\rho e^{i \theta};U)\big|
=
O(1)
\end{align*}
for all $\theta \in \mathbb{R}$ with the implied constant depending only on $N$ and $\rho_{2}$. 
\end{lemma}

\begin{proof}
Let $\rho \in [\rho_{1},\rho_{2}]$ and $\theta \in \mathbb{R}$. 
By using \eqref{eq:05091441}, we have 
\begin{align*}
\big|\log{Z}(\rho e^{i \theta};U)\big|
\leq 
\sum_{k=1}^{\infty} \frac{N}{k} \rho_{2}^{k}
= 
-N \log(1-\rho_{2})
\end{align*}
since $|\tr(U^k)| \leq N$ for $U \in \mathcal{U}(N)$. 
This yields the result. 
\end{proof}

Recall that $\log{Z}(s;U)$ is a holomorphic function on $|s|<1$. 
For every $U \in \mathcal{U}(N)$, we define a holomorphic function on $D_{1} \setminus \mathbb{R}$ by
\begin{align*}
\varphi_{U}(z)
=
\begin{cases}
0
&
\text{for $0<\IM(z)<1$}, 
\\
\log{Z}(e^{-iz};U)
&
\text{for $-1<\IM(z)<0$},  
\end{cases}
\end{align*}
which belongs to the space $\widetilde{\mathfrak{P}}_{1}$ by Lemma \ref{lem:2.8}. 
Then we define 
\begin{align*}
\log{Z}(e^{-ix-0};U)
=
\varphi_{U}(x-i0) 
\end{align*}
using the notation as in \eqref{eq:05071345}. 
Take an $N \times N$ random matrix $\mathsf{U}$ from CUE, that is, a random element distributed on $\mathcal{U}(N)$ according to the normalized Haar measure. 
By the following lemma, we derive that
\begin{align*}
\log{Z}_{N}^{\mathsf{CUE}}(e^{-ix-0})
:=
\log{Z}(e^{-ix-0};\mathsf{U})
\end{align*}
is a random hyperfunction. 

\begin{lemma}\label{lem:2.9}
The map $\mathcal{U}(N) \to \mathfrak{Q}$ defined by $U \mapsto \log{Z}(e^{-ix-0};U)$ is continuous. 
\end{lemma}

\begin{proof}
Let $(U_{n})$ be any sequence of $N \times N$ unitary matrices such that $U_{n} \to U$ as $n \to \infty$. 
Then it is sufficient to show that 
\begin{align}\label{eq:05091531}
\sup_{\| f \|_{m} \leq 1}
\Big|
\big\langle f \mid \log{Z}(e^{-ix-0};U_{n})-\log{Z}(e^{-ix-0};U) \big\rangle
\Big|
\to 
0
\end{align}
as $n \to \infty$ for all $m \geq1$. 
For any $f \in \mathfrak{B}_{m}^{*}$ with $\|f\|_{m} \leq 1$, we have $|f(z)| \leq e^{-|z|/m}$ for $|\IM(z)| \leq 2^{-m}$. 
Thus we obtain
\begin{align*}
&
\big\langle f \mid \log{Z}(e^{-ix-0};U_{n})-\log{Z}(e^{-ix-0};U) \big\rangle
\\
&
=
-\int_{-\infty-ic}^{\infty-ic}
f(z)
\big\{\log{Z}(e^{-iz};U_{n})-\log{Z}(e^{-iz};U)\big\}
\,dz
\\
&
\ll
\int_{-\infty}^{\infty}
e^{-|x-ic|/m}
\big|\log{Z}(e^{-c-ix};U_{n})-\log{Z}(e^{-c-ix};U)\big|
\,dx
\end{align*}
for $0<c<2^{-m}$, where the implied constant is absolute. 
By Lemma \ref{lem:2.8}, we have 
\begin{align*}
\log{Z}(e^{-c-ix};U_{n})-\log{Z}(e^{-c-ix};U)
\ll
1,
\end{align*}
where the implied constant depends only on $N$ and $c$. 
Here, $e^{-|x-ic|/m}$ is integrable over $\mathbb{R}$. 
Hence \eqref{eq:05091531} follows by the dominated convergence theorem since 
\begin{align*}
\log{Z}(e^{-iz};U_{n})-\log{Z}(e^{-iz};U)
=
\sum_{k=1}^{\infty} \frac{\tr(U_{n}^{k})-\tr(U^{k})}{k} e^{-ikz}
\to 
0
\end{align*}
as $n \to \infty$ on the line $\IM(z)=-c$. 
\end{proof}

Take a sequence of independent random variables $\mathsf{W}(k)$ distributed on $\mathbb{C}$ according to the standard complex Gaussian measure. 
We assume that they are defined on a probability space $(\Omega,\mathcal{F},\mathbb{P})$, but it does not need to be the same as in Section \ref{sec:2.2}. 
According to Diaconis and Evans \cite{DiaconisEvans2001}, it is known that $\tr(U^{k})$ behaves like $\sqrt{k}\, \mathsf{W}(k)$ as $U$ varies over $\mathcal{U}(N)$ and as $N \to \infty$. 
Hence we expect that the statistical behavior of $\log{Z}(s;U)$ is associated with the random analytic function
\begin{align*}
\log{Z}^{\mathsf{Gauss}}(s)
=
\sum_{k=1}^{\infty} \frac{1}{\sqrt{k}} \mathsf{W}(k) s^{k}. 
\end{align*}
Similarly to $\log{\zeta}^{\mathsf{rand}}(1/2+0+ix)$ in Section \ref{sec:2.2}, we define a random hyperfunction $\log{Z}^{\mathsf{Gauss}}(e^{-ix-0})$ in the remaining part of this section. 

\begin{lemma}\label{lem:2.10}
Almost surely, we have the followings: 
\begin{enumerate}
\item \label{lem:2.10_1}
$\log{Z}^{\mathsf{Gauss}}(s)$ is a holomorphic function on the disc $|s|<1$. 
\item \label{lem:2.10_2}
Let $\rho_{1}$ and $\rho_{2}$ be fixed with $0<\rho_{1}<\rho_{2}<1$. 
Then there exists a positive real valued random variable $\mathsf{M}$ determined only from $\rho_{2}$ such that 
\begin{align*}
\sup_{\rho_{1} \leq \rho \leq \rho_{2}}
\big|\log{Z}^{\mathsf{Gauss}}(\rho e^{i \theta})\big|
=
O(\mathsf{M})
\end{align*}
for all $\theta \in \mathbb{R}$ with the implied constant depending only on $\rho_{2}$. 
\end{enumerate}
\end{lemma}

\begin{proof}
Put $\mathsf{V}_{\eta}(k)=\mathsf{W}(k) \eta^{k} /\sqrt{k}$ for $0<\eta<1$. 
Then the random variables $\mathsf{V}_{\eta}(k)$ are independent for integers $k \geq1$, which satisfy $\mathbb{E}\big[ \mathsf{V}_{\eta}(k) \big]=0$ and 
\begin{align*}
\sum_{k=1}^{\infty} \mathbb{E}\big[ |\mathsf{V}_{\eta}(k)|^2 \big]
=
\sum_{k=1}^{\infty} \frac{1}{k}\, \eta^{2k}
<
\infty. 
\end{align*}
Hence we deduce from \cite[Theorem B.10.1]{Kowalski2021} that the series $\sum_{k=1}^{\infty} \mathsf{V}_{\eta}(k)$ converges almost surely. 
Take a set $\Omega_{0}(\eta) \in \mathcal{F}$ with $\mathbb{P}(\Omega_{0}(\eta))=1$ such that $\log{Z}^{\mathsf{Gauss}}(s)$ converges at $s=\eta$ for any sample of $\Omega_{0}(\eta)$. 
Then the set
\begin{align*}
\Omega_{0}
:=
\bigcap_{0<\eta<1} \Omega_{0}(\eta), 
\end{align*}
satisfies $\Omega_{0} \in \mathcal{F}$ and $\mathbb{P}(\Omega_{0})=1$. 
For any sample of $\Omega_{0}$, the power series $\log{Z}^{\mathsf{Gauss}}(s)$ converges absolutely and uniformly on any compact set in $|s|<1$. 
As a result, almost surely, $\log{Z}^{\mathsf{Gauss}}(s)$ is holomorphic on the disc $|s|<1$. 
Furthermore, for any sample of $\Omega_{0}$, we see that
\begin{align*}
\mathsf{M}
=
\max_{|s| \leq \rho_{2}} |\log{Z}^{\mathsf{Gauss}}(s)| 
\end{align*}
is a finite value. 
By letting $\mathsf{M}=1$ for any sample of the null set $(\Omega_{0})^{c}$, we can define 
\begin{align*}
\mathsf{M}:\Omega \to (0,\infty)
\end{align*}
as a random variable, for which the desired almost sure upper bound is satisfied. 
\end{proof}

By Lemma \ref{lem:2.10}, we can define almost surely 
\begin{align*}
\varphi^{\mathsf{Gauss}}(z)
=
\begin{cases}
0
&
\text{for $0<\IM(z)<1$}, 
\\
\log{Z}^{\mathsf{Gauss}}(e^{-iz})
&
\text{for $-1<\IM(z)<0$} 
\end{cases}
\end{align*}
as an element of the space $\widetilde{\mathfrak{P}}_{1}$. 
Then we define almost surely
\begin{align*}
\log{Z}^{\mathsf{Gauss}}(e^{-ix-0})
=
\varphi^{\mathsf{Gauss}}(x-i0). 
\end{align*}
Furthermore, we define an element of the space $\mathfrak{P}_{1}$ as 
\begin{align*}
\varphi_{k}(z,w)
=
\begin{cases}
0
&
\text{for $0<\IM(z)<1$}, 
\\
\displaystyle{\frac{1}{\sqrt{k}} w e^{-ikz}}
&
\text{for $-1<\IM(z)<0$}, 
\end{cases}
\end{align*}
where $w$ is any complex number. 
Then we define $g_{k}(e^{-ix-0},w)=\varphi_{k}(x-i0,w) \in \mathfrak{Q}$. 
Obviously, the map $\mathbb{C} \to \mathfrak{Q}$ defined by $w \mapsto g_{k}(e^{-ix-0},w)$ is continuous. 
Hence we see that $g_{k}^{\mathsf{Gauss}}(e^{-ix-0}):=g_{k}(e^{-ix-0},\mathsf{W}(k))$ is a random hyperfunction for every integer $k \geq1$. 
Then, we obtain similarly to Lemma \ref{lem:2.7} that almost surely
\begin{align*}
\sum_{k \leq y}
g_{k}^{\mathsf{Gauss}}(e^{-ix-0})
\to
\log{Z}^{\mathsf{Gauss}}(e^{-ix-0})
\end{align*}
as $y \to \infty$ in the space $\mathfrak{Q}$. 
Therefore, we can define the map
\begin{align*}
\log{Z}^{\mathsf{Gauss}}(e^{-ix-0}): \Omega \to \mathfrak{Q}
\end{align*}
as a random hyperfunction in the same way as at the end of Section \ref{sec:2.2}.

\section{Limit theorems revisited}\label{sec:3}

\subsection{Preliminary on the convergence of random elements}\label{sec:3.1}
Let $(E,d)$ be a metric space, and denote by $\mathcal{B}(E)$ the Borel $\sigma$-algebra of $E$ with respect to the topology induced from the metric $d$. 
Let $\mu, \mu_{1}, \mu_{2},\ldots$ be probability measure on $(E,\mathcal{B}(E))$. 
We say that $\mu_{n} \to \mu$ weakly as $n \to \infty$ if 
\begin{align*}
\lim_{n \to \infty}
\int_{E}
F
\,d \mu_{n}
=
\int_{E}
F
\,d \mu
\end{align*}
for all bounded continuous functions $F:E \to \mathbb{C}$.
Furthermore, let $X, X_{1}, X_{2},\ldots$ be random elements taking values in $E$, where we do not necessarily assume that they are defined on a common probability space. 
Denote by $P_{X}, P_{X_{1}}, P_{X_{2}}, \ldots$ the distribution of those random elements, that is, 
\begin{align*}
\int_{E}
F
\,d P_{X}
=
\mathbb{E}[F(X)]
\quad\text{and}\quad
\int_{E}
F
\,d P_{X_{n}}
=
\mathbb{E}[F(X_{n})]
\quad
(n=1,2,\ldots)
\end{align*} 
for all measurable functions $F:E \to \mathbb{C}$. 
We say that $X_{n}$ converges in distribution to $X$ as $n \to \infty$, formally 
\begin{align*}
X_{n} \xrightarrow{\mathcal{D}} X
\quad \text{as} \quad 
n \to \infty, 
\end{align*}
if $P_{X_{n}} \to P_{X}$ weakly as $n \to \infty$. 
Then the following lemmas are fundamental to investigating the convergence of random elements. 

\begin{lemma}\label{lem:3.1}
Let $X, X_{1}, X_{2},\ldots$ be random elements taking values in a metric space $(E,d)$. 
Then $X_{n} \xrightarrow{\mathcal{D}} X$ as $n \to \infty$ if and only if 
\begin{align*}
\lim_{n \to \infty} 
\mathbb{E}[F(X_{n})]
=
\mathbb{E}[F(X)]
\end{align*}
for all bounded Lipschitz continuous functions $F:E \to \mathbb{C}$.
\end{lemma}

\begin{proof}
This is known as a part of the Portmanteau theorem. 
We can find a proof in \cite[Theorem 13.16]{Klenke2020} for example. 
\end{proof}

\begin{lemma}\label{lem:3.2}
Let $\mu, \mu_{1}, \mu_{2},\ldots$ be probability measures on a separable metric space $(E,d)$. 
If $\mu_{n} \to \mu$ weakly as $n \to \infty$, then there exist random elements $X, X_{1}, X_{2},\ldots$ taking values in $E$ defined on a common probability space $(\Omega,\mathcal{F},\mathbb{P})$ such that 
\begin{align*}
P_{X}
=
\mu, 
\quad
P_{X_{n}}
=
\mu_{n}
\quad
(n=1,2,\ldots), 
\end{align*}
and $X_{n} \to X$ almost surely as $n \to \infty$. 
\end{lemma}

\begin{proof}
This is known as the Skorohod coupling. 
See \cite[Theorem 5.31]{Kallenberg2021} for a proof. 
\end{proof}

\begin{lemma}\label{lem:3.3}
Let $X, X_{1}, X_{2},\ldots$ be random variables taking values in $\mathbb{C}$. 
Assume 
\begin{align*}
\sup_{n \geq1} 
\mathbb{E}[H(|X_{n}|)]
<
\infty
\end{align*}
for some Borel measurable function $H: [0,\infty) \to [0,\infty)$ such that $H(x)/x \to \infty$ as $x \to \infty$. 
If $X_{n} \xrightarrow{\mathcal{D}} X$ as $n \to \infty$, then we have 
\begin{align*}
\mathbb{E}[|X|]
<
\infty
\quad\text{and}\quad
\lim_{n \to \infty} 
\mathbb{E}[X_{n}]
=
\mathbb{E}[X]. 
\end{align*}
\end{lemma}

\begin{proof}
If $X_{n} \xrightarrow{\mathcal{D}} X$ as $n \to \infty$, then $P_{X_{n}} \to P_{X}$ weakly as $n \to \infty$. 
By Lemma \ref{lem:3.2}, there exist random variables $Y, Y_{1}, Y_{2},\ldots$ taking values in $\mathbb{C}$ defined on a common probability space $(\Omega,\mathcal{F},\mathbb{P})$ such that 
\begin{align}\label{eq:05102215}
P_{X}=P_{Y}, 
\quad
P_{X_{n}}=P_{Y_{n}}
\quad
(n=1,2,\ldots), 
\end{align}
and $Y_{n} \to Y$ almost surely as $n \to \infty$. 
Since $P_{X_{n}}=P_{Y_{n}}$ for every $n \geq1$, we have 
\begin{align*}
\sup_{n \geq1} 
\mathbb{E}[H(|Y_{n}|)]
<
\infty
\end{align*}
by the assumption. 
Then we deduce from \cite[Theorem 6.19]{Klenke2020} that $Y_{n}$ are uniformly integrable for $n \geq1$. 
Therefore, the Vitali convergence theorem yields that
\begin{align*}
\mathbb{E}[|Y|]
<
\infty
\quad\text{and}\quad
\lim_{n \to \infty} 
\mathbb{E}[|Y_{n}-Y|]
=
0. 
\end{align*}
Hence the result follows immediately due to \eqref{eq:05102215}. 
\end{proof}

\subsection{Limit theorem for the Riemann zeta-function}\label{sec:3.2}
In this section, we prove a relationship between $\log{\zeta}_{T}^{\mathsf{unif}}(1/2+0+ix)$ and $\log{\zeta}^{\mathsf{rand}}(1/2+0+ix)$ defined in Section \ref{sec:2.2} in terms of the convergence of random elements. 
The ultimate goal is the following limit theorem which extends the result of Bohr and Jessen \cite{BohrJessen1930, BohrJessen1932} to the space of the Fourier hyperfunctions. 

\begin{theorem}\label{thm:3.4}
Assume RH. 
Then $\log{\zeta}_{T}^{\mathsf{unif}}(1/2+0+ix) \xrightarrow{\mathcal{D}} \log{\zeta}^{\mathsf{rand}}(1/2+0+ix)$ as $T \to \infty$. 
\end{theorem}

We begin with showing the following basic lemma. 

\begin{lemma}\label{lem:3.5}
Assume RH. 
Let $0<c<1/2$ and $K>1/2$. 
Then we have
\begin{align*}
&
\int_{-\infty-ic}^{\infty-ic}
\log^{*}{\zeta}(1/2+iz)\, \psi(z)
\,dz
\\
&
=
\int_{-\infty-iK}^{\infty-iK}
\log{\zeta}(1/2+iz)\, \psi(z)
\,dz
-
2\pi i
\int_{-\infty-i/2}^{0-i/2}
\psi(z)
\,dz, 
\end{align*}
where $\psi(z)$ is any holomorphic function on the half-plane $\IM(z)<0$ that satisfies $|\psi(z)| \ll (|z|+1)^{-2}$ in the strip $-b \leq \IM(z) \leq -a$ for any $0<a<b$.  
\end{lemma}

\begin{proof}
Let $1/2<\sigma_{1}<\sigma_{2}<\infty$ be fixed. 
For any $\sigma_{1} \leq \RE(s) \leq \sigma_{2}$, the distance between $s$ and the zeros of $\zeta(s)$ is greater than $\sigma_{1}-1/2$ assuming RH. 
Hence we deduce from \eqref{eq:05081444} and \eqref{eq:06011532} the formula
\begin{align}\label{eq:05161725}
\log^{*}{\zeta}(s)
=
\log(1-s)
+
O \big(\log(|\IM(s)|+2)\big)
\end{align}
for $\sigma_{1} \leq \RE(s) \leq \sigma_{2}$ such that $s \notin [1,\infty)$, where the implied constant depends only on $\sigma_{1}$ and $\sigma_{2}$. 
We have similarly
\begin{align}\label{eq:05161726}
\log{\zeta}(s)
=
-\log(s-1)
+
O \big(\log(|\IM(s)|+2)\big)
\end{align}
for $\sigma_{1} \leq \RE(s) \leq \sigma_{2}$ such that $s \notin (1/2,1]$, where the implied constant also depends only on $\sigma_{1}$ and $\sigma_{2}$.  
Hence we derive that for each $\epsilon>0$,
\begin{align*}
\log^{*}{\zeta}(1/2+iz)\, \psi(z)
\ll
|z|^{-2+\epsilon}
\quad\text{and}\quad
\log{\zeta}(1/2+iz)\, \psi(z)
\ll
|z|^{-2+\epsilon}
\end{align*}
as $|\RE(z)| \to \infty$ in the strip $-b \leq \IM(z) \leq -a$ for any $0<a<b$. 
By shifting the contour, we obtain
\begin{align*}
I
&
:=
\int_{-\infty-ic}^{\infty-ic}
\log^{*}{\zeta}(1/2+iz)\, \psi(z)
\,dz
\\
&
=
\bigg(
\int_{-\infty-i/2}^{-\delta-i/2}
+
\int_{C^{*}_{\delta}-i/2}
+
\int_{\delta-i/2}^{\infty-i/2}
\bigg)
\log^{*}{\zeta}(1/2+iz)\, \psi(z)
\,dz, 
\end{align*}
where $C^{*}_{\delta}$ denotes the semi-circle $\{\delta e^{i \theta} \mid -\pi \leq \theta \leq 0 \}$ with positive orientation. 
Then we also obtain
\begin{align*}
J
&
:=
\int_{-\infty-iK}^{\infty-iK}
\log{\zeta}(1/2+iz)\, \psi(z)
\,dz
\\
&
=
\bigg(
\int_{-\infty-i/2}^{-\delta-i/2}
+
\int_{-C_{\delta}-i/2}
+
\int_{\delta-i/2}^{\infty-i/2}
\bigg)
\log{\zeta}(1/2+iz)\, \psi(z)
\,dz, 
\end{align*}
where $-C_{\delta}$ denotes the semi-circle $\{\delta e^{i \theta} \mid 0 \leq \theta \leq \pi \}$ with negative orientation. 
Using \eqref{eq:05152038}, we evaluate the difference of these integrals as
\begin{align*}
I-J
&
=
\int_{C^{*}_{\delta}-i/2}
\log^{*}{\zeta}(1/2+iz)\, \psi(z)
\,dz
+
\int_{-C_{\delta}-i/2}
\log{\zeta}(1/2+iz)\, \psi(z)
\,dz 
\\
&
\quad
-2\pi i
\int_{-\infty+i/2}^{-\delta+i/2}
\psi(z)
\,dz. 
\end{align*}
Finally, we apply \eqref{eq:05161725} and \eqref{eq:05161726} to derive 
\begin{align*}
\int_{C^{*}_{\delta}-i/2}
\log^{*}{\zeta}(1/2+iz)\, \psi(z)
\,dz
&
\ll
\delta |\log{\delta}|
\to 
0, 
\\
\int_{-C_{\delta}-i/2}
\log{\zeta}(1/2+iz)\, \psi(z)
\,dz 
&
\ll
\delta |\log{\delta}|
\to 
0
\end{align*}
as $\delta \to 0$. 
Therefore we obtain the conclusion. 
\end{proof}

Denote by $\mathbb{T}^{\infty}$ the set of all arithmetic functions $X:\mathbb{N} \to \mathbb{T}$, where $\mathbb{T}$ stands for the unit circle $\{z \in \mathbb{C} \mid |z|=1 \}$. 
Let $N,y \geq 1$ and $X \in \mathbb{T}^{\infty}$. 
Then, we introduce a Dirichlet series $\Phi_{N}(z;X)$ and a Dirichlet polynomial $\Phi_{N,y}(z;X)$ as \begin{align*}
\Phi_{N}(z;X)
&
=
\sum_{n=1}^{\infty}
\frac{\Lambda(n)}{\log{n}} e^{-n/N} X(n) n^{-1/2-iz}, 
\\
\Phi_{N,y}(z;X)
&
=
\sum_{n \leq y}
\frac{\Lambda(n)}{\log{n}} e^{-n/N} X(n) n^{-1/2-iz}, 
\end{align*}
where $\Lambda(n)$ is the von-Mangoldt function. 
Note that the Dirichlet series $\Phi_{N}(z;X)$ converges absolutely for all $z \in \mathbb{C}$. 
Furthermore, its restriction to $D_{1} \setminus \mathbb{R}$ is an element of the space $\widetilde{\mathfrak{P}}_{1}$. 
The restriction of $\Phi_{N,y}(z;X)$ to $D_{1} \setminus \mathbb{R}$ is also an element of $\widetilde{\mathfrak{P}}_{1}$. 
Then we obtain the Fourier hyperfunctions $\Phi_{N}(x-i0;X)$ and $\Phi_{N,y}(x-i0;X)$. 
Recall that $\mathbb{T}^{\infty}$ is equipped with the topology induced from the metric
\begin{align*}
d(X,Y)
=
\sum_{n=1}^{\infty} 
2^{-n} 
\frac{|X(n)-Y(n)|}{1+|X(n)-Y(n)|}
\end{align*}
with $X,Y \in \mathbb{T}^{\infty}$. 
Hence, $X_{j} \to X$ as $j \to \infty$ in $\mathbb{T}^{\infty}$ if and only if $X_{j}(n) \to X(n)$ as $j \to \infty$ for all $n \in \mathbb{N}$. 
The maps $\mathbb{T}^{\infty} \to \mathfrak{Q}$ defined by $X \mapsto \Phi_{N}(x-i0;X)$ and $X \mapsto \Phi_{N,y}(x-i0;X)$ are continuous, which can be show by the method applying the dominated convergence theorem that we have repeatedly used in Section \ref{sec:2}. 
As a result, we derive that $\Phi_{N}(x-i0; \mathsf{X}_{T})$, $\Phi_{N,y}(x-i0; \mathsf{X}_{T})$, $\Phi_{N}(x-i0; \mathsf{X})$, and $\Phi_{N,y}(x-i0; \mathsf{X})$ are random hyperfunctions, where $\mathsf{X}_{T}$ is defined by letting 
\begin{align*}
\mathsf{X}_{T}(n)
=
n^{-iT \upsilon}
\end{align*}
with a random variable $\upsilon$ uniformly distributed on the interval $[0,1]$, and $\mathsf{X}$ is defined by extending $\mathsf{X}(p)$ to all positive integers $n$ by complete multiplicativity. 
Then we obtain the following propositions. 

\begin{proposition}\label{prop:3.6}
$\Phi_{N}(x-i0; \mathsf{X}_{T}) \xrightarrow{\mathcal{D}} \Phi_{N}(x-i0; \mathsf{X})$ as $T \to \infty$ for any $N \geq1$. 
\end{proposition}

\begin{proof}
Let $F: \mathfrak{Q} \to \mathbb{C}$ be a bounded Lipschitz continuous function. 
Then we have a constant $C_{F}>0$ such that 
\begin{align}\label{eq:05142151}
|F(g_{1})-F(g_{2})|
\leq 
C_{F}\, d(g_{1},g_{2}) 
\end{align}
for any $g_{1}, g_{2} \in \mathfrak{Q}$, where $d$ is the metric defined by \eqref{eq:04170251}. 
To obtain the result, by Lemma \ref{lem:3.1}, it suffices to show the limit formula
\begin{align*}
\lim_{T \to \infty} 
\mathbb{E} \big[F(\Phi_{N}(x-i0; \mathsf{X}_{T}))\big]
=
\mathbb{E} \big[F(\Phi_{N}(x-i0; \mathsf{X}))\big]
\end{align*}
for any $N \geq1$. 
Furthermore, this follows if we obtain the followings: 
\begin{align}
\label{eq:05142215}
\lim_{y \to \infty}
\limsup_{T \to \infty} 
\Big| 
\mathbb{E} \big[F(\Phi_{N}(x-i0; \mathsf{X}_{T}))\big]
-
\mathbb{E} \big[F(\Phi_{N,y}(x-i0; \mathsf{X}_{T}))\big]
\Big|
&
=
0, 
\\
\label{eq:05142216}
\sup_{y \geq1}
\lim_{T \to \infty} 
\Big| 
\mathbb{E} \big[F(\Phi_{N,y}(x-i0; \mathsf{X}_{T}))\big]
-
\mathbb{E} \big[F(\Phi_{N,y}(x-i0; \mathsf{X}))\big]
\Big|
&
=
0, 
\\
\label{eq:05142217}
\lim_{y \to \infty}
\Big| 
\mathbb{E} \big[F(\Phi_{N,y}(x-i0; \mathsf{X}))\big]
-
\mathbb{E} \big[F(\Phi_{N}(x-i0; \mathsf{X}))\big]
\Big|
&
=
0. 
\end{align}
By \eqref{eq:05142151} and the definition of the metric of $\mathfrak{Q}$, we see that \eqref{eq:05142215} holds if 
\begin{align*}
\lim_{y \to \infty}
\limsup_{T \to \infty} 
\mathbb{E} \big[\Delta_{T,N,y}(m)\big] 
=
0
\end{align*}
for all $m \geq1$, where $\Delta_{T,N,y}(m)$ is a random variable such that
\begin{align*}
\Delta_{T,N,y}(m)
=
\sup_{\|f\|_{m} \leq 1}
\Big|\big \langle f \mid \Phi_{N}(x-i0; \mathsf{X}_{T})-\Phi_{N,y}(x-i0; \mathsf{X}_{T}) \big \rangle\Big|. 
\end{align*}
Let $0<c<2^{-m}$. 
Since $|f(z)| \leq e^{-|z|/m}$ on the line $\IM(z)=-c$ if $f \in \mathfrak{B}_{m}^{*}$ satisfies $\|f\|_{m} \leq 1$, we have 
\begin{align*}
\Delta_{T,N,y}(m)
&
\leq
\int_{-\infty}^{\infty}
e^{-|x-ic|/m}
\big|\Phi_{N}(x-ic; \mathsf{X}_{T})-\Phi_{N,y}(x-ic; \mathsf{X}_{T})\big|
\,dx
\\
& 
\leq
\sum_{n>y}
\frac{\Lambda(n)}{\log{n}} e^{-n/N} n^{-1/2-c}
\int_{-\infty}^{\infty}
e^{-|x-ic|/m}
\,dx
\end{align*}
by recalling that the Dirichlet series $\Phi_{N}(z;\mathsf{X}_{T})$ converges absolutely for all $z \in \mathbb{C}$. 
This yields that
\begin{align*}
\limsup_{T \to \infty} 
\mathbb{E} \big[\Delta_{T,N,y}(m)\big] 
\leq 
\sum_{n>y}
\frac{\Lambda(n)}{\log{n}} e^{-n/N} n^{-1/2-c} 
\int_{-\infty}^{\infty}
e^{-|x-ic|/m}
\,dx
\to 
0
\end{align*}
as $y \to \infty$. 
Hence we obtain \eqref{eq:05142215} as noted above. 
Then, we will show \eqref{eq:05142216}.  
Recall that the Kronecker–Weyl equidistribution theorem yields
\begin{align*}
(\mathsf{X}_{T}(n_{1}), \ldots, \mathsf{X}_{T}(n_{k}))
\xrightarrow{\mathcal{D}}
(\mathsf{X}(n_{1}), \ldots, \mathsf{X}(n_{k}))
\end{align*}
as $T \to \infty$ for any $n_{1},\ldots,n_{k} \in \mathbb{N}$. 
See \cite[Proposition 3.2.5]{Kowalski2021} for example. 
Since $\Phi_{N,y}(x-i0; \mathsf{X}_{T})$ and $\Phi_{N,y}(x-i0; \mathsf{X})$ are represented as linear combinations of the random variables $\mathsf{X}_{T}(n)$ and $\mathsf{X}(n)$ only for $n \leq y$, we derive that
\begin{align*}
\Phi_{N,y}(x-i0; \mathsf{X}_{T})
\xrightarrow{\mathcal{D}}
\Phi_{N,y}(x-i0; \mathsf{X})
\end{align*}
as $T \to \infty$ for any $y \geq1$. 
Hence we obtain \eqref{eq:05142216} by the definition of the convergence of random elements in distribution. 
Lastly, \eqref{eq:05142217} can be proved in a way similar to \eqref{eq:05142215}.    
Indeed, it suffices to show
\begin{align*}
\lim_{y \to \infty}
\mathbb{E} \big[\Delta_{N,y}(m)\big] 
=
0
\end{align*}
for all $m \geq1$, where $\Delta_{N,y}(m)$ is a random variable such that
\begin{align*}
\Delta_{N,y}(m)
=
\sup_{\|f\|_{m} \leq 1}
\Big|\big \langle f \mid \Phi_{N}(x-i0; \mathsf{X})-\Phi_{N,y}(x-i0; \mathsf{X}) \big \rangle\Big|. 
\end{align*}
Then we have 
\begin{align*}
\Delta_{N,y}(m)
&
\leq
\int_{-\infty}^{\infty}
e^{-|x-ic|/m}
\big|\Phi_{N}(x-ic; \mathsf{X})-\Phi_{N,y}(x-ic; \mathsf{X})\big| 
\,dx
\\
& 
\leq
\sum_{n>y}
\frac{\Lambda(n)}{\log{n}} e^{-n/N} n^{-1/2-c}
\int_{-\infty}^{\infty}
e^{-|x-ic|/m}
\,dx
\end{align*}
with $0<c<2^{-m}$, and therefore, 
\begin{align*}
\mathbb{E} \big[\Delta_{N,y}(m)\big] 
\leq
\sum_{n>y}
\frac{\Lambda(n)}{\log{n}} e^{-n/N} n^{-1/2-c} 
\int_{-\infty}^{\infty}
e^{-|x-ic|/m}
\,dx
\to 
0
\end{align*}
as $y \to \infty$. 
Hence we obtain \eqref{eq:05142217}. 
From the above, the proof is completed. 
\end{proof}

\begin{proposition}\label{prop:3.7}
Assume RH. 
Take any function $f \in \mathfrak{B}_{m}^{*}$ for $m \geq1$. 
Let $c$ and $d$ be real numbers with $0<d<c<2^{-m-1}$. 
For any $N \geq1$, we have 
\begin{align*}
&
\big\langle
f
\mid
\log{\zeta}_{T}^{\mathsf{unif}}(1/2+0+ix)-\Phi_{N}(x-i0; \mathsf{X}_{T})
\big\rangle
\\
&
=
\frac{1}{2\pi}
\int_{-\infty-ic}^{\infty-ic}
\int_{-\infty-id}^{\infty-id}
f(z)
\log^{*}{\zeta}(1/2+i(z-w)+iT \upsilon)
\Gamma(-iw)
N^{-iw}
\,dw
\,dz
\\
&
\quad
+
O \big(N^{1/2-c}\, (T \upsilon+1)^{-1}\big)
\end{align*}
with the implied constant depending only on $m$ and $c$, and furthermore, we have almost surely 
\begin{align*}
&
\big\langle
f
\mid
\log{\zeta}^{\mathsf{rand}}(1/2+0+ix)-\Phi_{N}(x-i0; \mathsf{X})
\big\rangle
\\
&
=
\frac{1}{2\pi}
\int_{-\infty-ic}^{\infty-ic}
\int_{-\infty-id}^{\infty-id}
f(z)
\log{\zeta}^{\mathsf{rand}}(1/2+i(z-w))
\Gamma(-iw)
N^{-iw}
\,dw
\,dz. 
\end{align*}
\end{proposition}

\begin{proof}
Define $\psi_{N}(w)=-iw \Gamma(-iw) N^{-iw}$ for $N \geq1$. 
This is a holomorphic function on the half-plane $\IM(w)>-1$. 
Furthermore, we apply the Stirling formula to derive
\begin{align*}
\sup_{|\IM(w)| \leq \delta_{1}} 
|\psi_{N}(w)| e^{\epsilon |w|}
<
\infty
\end{align*}
for any $0<\delta_{1}<1$, where we can take $0<\epsilon<\pi/2$ arbitrarily. 
Hence $\psi_{N}$ belongs to the space $\mathfrak{P}^{*}$. 
Then Lemma \ref{lem:2.4} yields that the function
\begin{align*}
\log^{*}{\zeta}(1/2+i(z-w)+it)\, \psi_{N}(w)
\end{align*}
also belongs to $\mathfrak{P}^{*}$ for any $t \in \mathbb{R}$ and any $z \in \mathbb{C}$ with $-1/2<\IM(z)<0$. 
Let $\delta \in \mathfrak{Q}$ denote the Dirac delta defined as
\begin{align*}
\delta(x)
=
-\frac{1}{2\pi i} \Big(\frac{1}{x+i0}-\frac{1}{x-i0}\Big). 
\end{align*}
Note that $\psi_{N}(0)=1$ since $\Gamma(s)$ has a simple pole at $s=0$ with residue $1$. 
Therefore, $\log^{*}{\zeta}(1/2+iz+it)$ is represented as
\begin{align*}
\log^{*}{\zeta}(1/2+iz+it)
&
=
\big\langle
\log^{*}{\zeta}(1/2+i(z-w)+it)\, \psi_{N}(w)
\mid
\delta\,
\big\rangle
\\
&
=
-\frac{1}{2\pi i}
\int_{-\infty+id}^{\infty+id}
\log^{*}{\zeta}(1/2+i(z-w)+it)\, \psi_{N}(w)
\,\frac{dw}{w}
\\
&
\quad
+\frac{1}{2\pi i}
\int_{-\infty-id}^{\infty-id}
\log^{*}{\zeta}(1/2+i(z-w)+it)\, \psi_{N}(w)
\,\frac{dw}{w}
\end{align*}
for any $t \in \mathbb{R}$ and any $z \in \mathbb{C}$ on the line $\IM(z)=-c$ with $0<c<1/2$, where $d$ is taken so that $0<d<\min(c, 1/2-c)$. 
Put $K=3/2-c$. 
By Lemma \ref{lem:3.5}, it yields
\begin{align}\label{eq:05152155}
\log^{*}{\zeta}(1/2+iz+it)
&
=
-\frac{1}{2\pi i}
\int_{-\infty+iK}^{\infty+iK}
\log{\zeta}(1/2+i(z-w)+it)\, \psi_{N}(w)
\,\frac{dw}{w}
\\
\nonumber
&
\quad
+
\int_{t+\RE(z)+i(1/2-c)}^{\infty+i(1/2-c)}
\psi_{N}(w)
\,\frac{dw}{w}
\\
\nonumber
&
\quad
+
\frac{1}{2\pi i}
\int_{-\infty-id}^{\infty-id}
\log^{*}{\zeta}(1/2+i(z-w)+it)\, \psi_{N}(w)
\,\frac{dw}{w}
\\
\nonumber
&
=
I_{1}+I_{2}+I_{3}, 
\end{align}
say. 
For any $z,w \in \mathbb{C}$ with $\IM(z)=-c$ and $\IM(w)=K$, we apply \eqref{eq:05081554} to obtain
\begin{align*}
\log{\zeta}(1/2+i(z-w)+it)
=
\sum_{n=1}^{\infty} 
\frac{\Lambda(n)}{\log{n}} n^{iw} n^{-it} n^{-1/2-iz}. 
\end{align*}
By this and Fubini's theorem, the first integral in \eqref{eq:05152155} is calculated as
\begin{align*}
I_{1}
&
=
\sum_{n=1}^{\infty} 
\frac{\Lambda(n)}{\log{n}} 
\bigg\{
\frac{1}{2\pi}
\int_{-\infty+iK}^{\infty+iK}
\Gamma(-iw)
\Big(\frac{n}{N}\Big)^{iw} 
\,dw
\bigg\}
n^{-it} n^{-1/2-iz}
\\
&
=
\sum_{n=1}^{\infty} 
\frac{\Lambda(n)}{\log{n}} 
e^{-n/N} n^{-it} n^{-1/2-iz}
\end{align*}
due to the definition of $\psi_{N}$, where we used the formula
\begin{align}\label{eq:05160232}
e^{-y}
=
\frac{1}{2\pi i}
\int_{K-i\infty}^{K+i\infty}
\Gamma(s)
y^{-s} 
\,ds
=
\frac{1}{2\pi}
\int_{-\infty+iK}^{\infty+iK}
\Gamma(-iw)
y^{iw} 
\,dw 
\end{align}
which is valid for all $y>0$. 
The second and third integrals in \eqref{eq:05152155} are calculated as
\begin{align*}
I_{2}
&
=
-i
\int_{t+\RE(z)+i(1/2-c)}^{\infty+i(1/2-c)}
\Gamma(-iw) N^{-iw}
\,dw, 
\\
I_{3}
&
=
-\frac{1}{2\pi}
\int_{-\infty-id}^{\infty-id}
\log^{*}{\zeta}(1/2+i(z-w)+it)\, \Gamma(-iw) N^{-iw}
\,dw. 
\end{align*}
Furthermore, as an application of the Stirling formula, we have
\begin{align*}
I_{2}
\ll
N^{1/2-c}
\int_{t+\RE(z)}^{\infty}
e^{-\frac{\pi}{2}|u|}
\,du
\ll
N^{1/2-c}\, 
\min(1,e^{-t-\RE(z)}) 
\end{align*}
with the implied constants depending only on $c$. 
Inserting the above results on the integrals $I_{1},I_{2},I_{3}$ to \eqref{eq:05152155}, we obtain
\begin{align}\label{eq:05160131}
&
\log^{*}{\zeta}(1/2+iz+it)
-
\sum_{n=1}^{\infty} 
\frac{\Lambda(n)}{\log{n}} 
e^{-n/N} n^{-it} n^{-1/2-iz}
\\
\nonumber
&
=
-\frac{1}{2\pi}
\int_{-\infty-id}^{\infty-id}
\log^{*}{\zeta}(1/2+i(z-w)+it)\, \Gamma(-iw) N^{-iw}
\,dw 
\\
\nonumber
&
\quad
+
O\big(N^{1/2-c}\, \min(1,e^{-t-\RE(z)})\big)
\end{align}
for any $t \in \mathbb{R}$ and any $z \in \mathbb{C}$ on the line $\IM(z)=-c$. 
Assume $0<d<c<2^{-m-1}$ below. 
Then we have $0<c<1/2$ and $0<d<\min(c,1/2-c)$. 
Finally, we apply \eqref{eq:05160131} to derive 
\begin{align*}
&
\big\langle
f
\mid
\log{\zeta}_{T}^{\mathsf{unif}}(1/2+0+ix)-\Phi_{N}(x-i0; \mathsf{X}_{T})
\big\rangle
\\
&
=
-
\int_{-\infty-ic}^{\infty-ic}
f(z)
\big\{ 
\log^{*}{\zeta}(1/2+iz+iT \upsilon)-\Phi_{N}(z; \mathsf{X}_{T})
\big\}
\,dz
\\
&
=
\frac{1}{2\pi}
\int_{-\infty-ic}^{\infty-ic}
\int_{-\infty-id}^{\infty-id}
f(z)
\log^{*}{\zeta}(1/2+i(z-w)+iT \upsilon)
\Gamma(-iw)
N^{-iw}
\,dw
\,dz
\\
&
\quad
+
O \bigg(
N^{1/2-c}\, 
\int_{-\infty}^{\infty}
|f(x-ic)| \min(1,e^{-T \upsilon-x})
\,dx
\bigg)
\end{align*}
by recalling the definition of the Dirichlet series $\Phi_{N}(z; \mathsf{X}_{T})$. 
For $f \in \mathfrak{B}_{m}^{*}$, we have $|f(x-ic)| \ll e^{-|x|/m}$ with the implied constant depending only on $m$ and $c$. 
Hence the integral in the error term is estimated as
\begin{align*}
\int_{-\infty}^{\infty}
|f(x-ic)| \min(1,e^{-T \upsilon-x})
\,dx
\ll
e^{-T \upsilon/(2m)}. 
\end{align*}
Therefore, the desired formula for $\log{\zeta}_{T}^{\mathsf{unif}}(1/2+0+ix)$ follows. 
Then the result for $\log{\zeta}^{\mathsf{rand}}(1/2+0+ix)$ is obtained more easily since the Dirichlet series $\log{\zeta}^{\mathsf{rand}}(s)$ converges almost surely for $\RE(s)>1/2$. 
As an analogue for \eqref{eq:05152155}, almost surely, $\log{\zeta}^{\mathsf{rand}}(1/2+iz)$ is represented as
\begin{align*}
\log{\zeta}^{\mathsf{rand}}(1/2+iz)
&
=
-\frac{1}{2\pi i}
\int_{-\infty+id}^{\infty+id}
\log{\zeta}^{\mathsf{rand}}(1/2+i(z-w))\, \psi_{N}(w)
\,\frac{dw}{w}
\\
&
\quad
+\frac{1}{2\pi i}
\int_{-\infty-id}^{\infty-id}
\log{\zeta}^{\mathsf{rand}}(1/2+i(z-w))\, \psi_{N}(w)
\,\frac{dw}{w}
\\
&
=
\widetilde{I}_{1}+\widetilde{I}_{2}, 
\end{align*}
say. 
By the method using Fubini's theorem, the integral $\widetilde{I}_{1}$ is calculated as
\begin{align*}
\widetilde{I}_{1}
&
=
\sum_{n=1}^{\infty} 
\frac{\Lambda(n)}{\log{n}} 
\bigg\{
\frac{1}{2\pi}
\int_{-\infty+id}^{\infty+id}
\Gamma(-iw)
\Big(\frac{n}{N}\Big)^{iw} 
\,dw
\bigg\}
\mathsf{X}(n) n^{-1/2-iz}
\\
&
=
\sum_{n=1}^{\infty} 
\frac{\Lambda(n)}{\log{n}} 
e^{-n/N} \mathsf{X}(n) n^{-1/2-iz}
\end{align*}
due to \eqref{eq:05160232}. 
Thus $\widetilde{I}_{1}=\Phi_{N}(z; \mathsf{X})$ holds. 
Therefore, we obtain almost surely
\begin{align*}
&
\log{\zeta}^{\mathsf{rand}}(1/2+iz)
-
\Phi_{N}(z; \mathsf{X})
\\
&
=
-\frac{1}{2\pi}
\int_{-\infty-id}^{\infty-id}
\log{\zeta}^{\mathsf{rand}}(1/2+i(z-w))\, \Gamma(-iw) N^{-iw}
\,dw, 
\end{align*}
from which the desired conclusion follows immediately. 
\end{proof}

\begin{proposition}\label{prop:3.8}
Assume RH. 
Let $\sigma_{1}$ and $\sigma_{2}$ be fixed with $1/2<\sigma_{1}<\sigma_{2}<1$. 
Then we have
\begin{align*}
\sup_{\sigma_{1} \leq \sigma \leq \sigma_{2}}
\frac{1}{T} \int_{0}^{T}
\big|\log^{*}{\zeta}(\sigma+ix+it)\big|
\,dt
&
=
O (|x|+1),
\\
\sup_{\sigma_{1} \leq \sigma \leq \sigma_{2}} 
\mathbb{E} \Big[
\big|\log{\zeta}^{\mathsf{rand}}(\sigma+ix)\big|
\Big]
&
=
O (1)
\end{align*}
for all $x \in \mathbb{R}$ with the implied constants depending only on $\sigma_{1}$ and $\sigma_{2}$. 
\end{proposition}

\begin{proof}
Let $\sigma_{1} \leq \sigma \leq \sigma_{2}$ and $x \in \mathbb{R}$. 
We have obviously
\begin{align*}
\frac{1}{T} \int_{0}^{T}
\big|\log^{*}{\zeta}(\sigma+ix+it)\big|
\,dt
\leq
\frac{1}{T}
\Big(
\int_{-T-|x|}^{-1}
+
\int_{-1}^{1}
+
\int_{1}^{T+|x|}
\Big)
\big|\log^{*}{\zeta}(\sigma+it)\big|
\,dt.
\end{align*}
By \eqref{eq:05081506}, the integral for $1 \leq t \leq T+|x|$ is estimated as 
\begin{align}\label{eq:05170004}
\frac{1}{T} \int_{1}^{T+|x|}
\big|\log^{*}{\zeta}(\sigma+it)\big|
\,dt
&
\leq
\frac{1}{T} \int_{1}^{T+|x|}
\bigg\{
\int_{\sigma}^{2}
\Big|\frac{\zeta'}{\zeta}(u+it)\Big|
\,du
+
O(1)
\bigg\}
\,dt
\\
\nonumber
&
=
\frac{1}{T}
\int_{\sigma}^{2}
\int_{1}^{T+|x|}
\Big|\frac{\zeta'}{\zeta}(u+it)\Big|
\,dt
\,du
+
O (|x|+1). 
\end{align}
Then we apply the Cauchy--Schwarz inequality and the second moment result
\begin{align*}
\int_{0}^{T}
\Big|\frac{\zeta'}{\zeta}(u+it)\Big|^{2}
\,dt
=
O \Big(
\frac{T}{(u-1/2)^{2}}
\Big)
\end{align*}
with $1/2+10(\log{T})^{-1} \leq u \leq 3/4$; see \cite[Lemma 4]{Selberg1943} for a proof. 
Furthermore, by the same proof, this result can be extended to $3/4<u \leq 2$ by replacing the interval $[0,T]$ on the integral with $[1,T]$. 
It is derived that
\begin{align*}
\int_{1}^{T+|x|}
\Big|\frac{\zeta'}{\zeta}(u+it)\Big|
\,dt
=
O(|T|+x)
\end{align*}
for any $\sigma \leq u \leq 2$, where the implied constant depends only on $\sigma_{1}$. 
This yields
\begin{align*}
\frac{1}{T} \int_{1}^{T+|x|}
\big|\log^{*}{\zeta}(\sigma+it)\big|
\,dt
=
O(|x|+1) 
\end{align*}
by \eqref{eq:05170004}. 
The integral for $-T-|x| \leq t \leq -1$ can be estimated in a similar way. 
Lastly, we see that
\begin{align*}
\frac{1}{T} \int_{-1}^{1}
\big|\log^{*}{\zeta}(\sigma+it)\big|
\,dt
=
O(1)
\end{align*}
with the implied constant depending only on $\sigma_{1}$ and $\sigma_{2}$. 
Hence the desired result for $\log^{*}{\zeta}(\sigma+ix+it)$ is proved. 
Then, we obtain
\begin{align}\label{eq:05170007}
\mathbb{E} \Big[
\big|\log{\zeta}^{\mathsf{rand}}(\sigma+ix)\big|
\Big]
\leq
\mathbb{E} \bigg[
\Big|
\sum_{n=1}^{\infty} 
\frac{\Lambda(n)}{\log{n}} 
\mathsf{X}(n) n^{-\sigma-ix}
\Big|^{2}
\bigg]^{1/2}
\end{align}
by the Cauchy--Schwarz inequality and the definition of $\log{\zeta}^{\mathsf{rand}}(s)$. 
Since 
\begin{align*}
\mathbb{E} \big[\mathsf{X}(m) \overline{\mathsf{X}(n)} \Big]
=
\begin{cases}
1
& \text{if $m=n$},
\\
0
& \text{if $m \neq n$}, 
\end{cases}
\end{align*}
we calculate the second moment as
\begin{align*}
\mathbb{E} \bigg[
\Big|
\sum_{n=1}^{\infty} 
\frac{\Lambda(n)}{\log{n}} 
\mathsf{X}(n) n^{-\sigma-ix}
\Big|^{2}
\bigg]
=
\sum_{n=1}^{\infty} 
\frac{\Lambda(n)^{2}}{(\log{n})^2} n^{-2\sigma}
=
O(1)
\end{align*}
with the implied constant depending only on $\sigma_{1}$. 
By \eqref{eq:05170007}, the desired result for $\log{\zeta}^{\mathsf{rand}}(\sigma+ix)$ is also proved. 
\end{proof}

Now we are ready to prove Theorem \ref{thm:3.4}. 

\begin{proof}[Proof of Theorem \ref{thm:3.4}]
By Lemma \ref{lem:3.1}, the result holds if we obtain
\begin{align*}
\lim_{T \to \infty} 
\mathbb{E} \big[F(\log{\zeta}_{T}^{\mathsf{unif}}(1/2+0+ix))\big]
=
\mathbb{E} \big[F(\log{\zeta}^{\mathsf{rand}}(1/2+0+ix))\big], 
\end{align*}
where $F: \mathfrak{Q} \to \mathbb{C}$ is a bounded Lipschitz continuous function. 
This follows from
\begin{align}
\label{eq:05070011}
\lim_{N \to \infty}
\limsup_{T \to \infty} 
\Big| 
\mathbb{E} \big[F(\log{\zeta}_{T}^{\mathsf{unif}}(1/2+0+ix))\big]
-
\mathbb{E} \big[F(\Phi_{N}(x-i0; \mathsf{X}_{T}))\big]
\Big|
&
=
0, 
\\
\label{eq:05070012}
\sup_{N \geq1}
\lim_{T \to \infty} 
\Big| 
\mathbb{E} \big[F(\Phi_{N}(x-i0; \mathsf{X}_{T}))\big]
-
\mathbb{E} \big[F(\Phi_{N}(x-i0; \mathsf{X}))\big]
\Big|
&
=
0, 
\\
\label{eq:05070013}
\lim_{N \to \infty}
\Big| 
\mathbb{E} \big[F(\Phi_{N}(x-i0; \mathsf{X}))\big]
-
\mathbb{E} \big[F(\log{\zeta}^{\mathsf{rand}}(1/2+0+ix))\big]
\Big|
&
=
0
\end{align}
as in the proof of Proposition \ref{prop:3.6}. 
In order to prove \eqref{eq:05070011}, it suffices to show
\begin{align}\label{eq:05170349}
\lim_{N \to \infty}
\limsup_{T \to \infty} 
\mathbb{E} \big[\Delta_{T,N}(m)\big] 
=
0
\end{align}
for all $m \geq1$, where $\Delta_{T,N}(m)$ is a random variable such that
\begin{align*}
\Delta_{T,N}(m)
=
\sup_{\|f\|_{m} \leq 1}
\Big|\big \langle f \mid \log{\zeta}_{T}^{\mathsf{unif}}(1/2+0+ix)-\Phi_{N}(x-i0; \mathsf{X}_{T}) \big \rangle\Big|. 
\end{align*}
Let $c$ and $d$ be real numbers with $0<d<c<2^{-m-1}$. 
Recall that $|f(z)| \leq e^{-|z|/m}$ on the line $\IM(z)=-c$ for any $f \in \mathfrak{B}_{m}^{*}$ such that $\|f\|_{m} \leq 1$. 
Therefore, we apply Proposition \ref{prop:3.7} to derive 
\begin{align*}
\Delta_{T,N}(m)
&
\leq
\frac{N^{-d}}{2\pi}
\int_{-\infty}^{\infty}
\int_{-\infty}^{\infty}
\Big\{
e^{-|x-ic|/m}
\big|\log^{*}{\zeta}(1/2+(c-d)+i(x-u)+iT \upsilon)\big|
\\
&
\qquad\qquad\qquad\qquad
\times
|\Gamma(-d-iu)|
\Big\}
\,du
\,dx
+
O \big(N^{1/2-c}\, (T \upsilon+1)^{-1}\big), 
\end{align*}
where the implied constant depends only on $m$ and $c$. 
By Fubini's theorem and Proposition \ref{prop:3.8}, the expected value of $\Delta_{T,N}(m)$ is evaluated as 
\begin{align*}
\mathbb{E} \big[\Delta_{T,N}(m)\big] 
&
\ll
N^{-d}
\int_{-\infty}^{\infty}
\int_{-\infty}^{\infty}
e^{-|x-ic|/m}
(|x-u|+1)
|\Gamma(-d-iu)|
\,du
\,dx
\\
&
\quad
+
N^{1/2-c}\, T^{-1} \log{T}
\end{align*}
with the implied constant depending only on $m$, $c$, and $d$. 
Using the Stirling formula, we arrive at the estimate
\begin{align*}
\mathbb{E} \big[\Delta_{T,N}(m)\big] 
\ll
N^{-d}
+
N^{1/2-c}\, T^{-1} \log{T}, 
\end{align*}
which yields \eqref{eq:05170349}. 
Then, we see that \eqref{eq:05070012} holds by Proposition \ref{prop:3.6}.
The proof of \eqref{eq:05070013} remains, but this can be show in a way quite similar to \eqref{eq:05070011}. 
Therefore we omit it, and the proof of the theorem is completed.  
\end{proof}

Furthermore, we apply Lemma \ref{lem:3.3} to derive the following corollary. 

\begin{corollary}\label{cor:3.9}
Assume RH. 
For any $\vec{\kappa}=(\kappa_{1},\kappa_{2}) \in \mathbb{C}^{2}$ and $L>0$, 
\begin{align*}
&
\lim_{T \to \infty} 
\mathbb{E} \Big[
\exp \Big(
\vec{\kappa} \circ
\big\langle
-f_{L}
\mid
\log{\zeta}_{T}^{\mathsf{unif}}(1/2+0+ix)
\big\rangle
\Big)
\Big]
\\
&
=
\mathbb{E} \Big[
\exp \Big(
\vec{\kappa} \circ
\big\langle
-f_{L}
\mid
\log{\zeta}^{\mathsf{rand}}(1/2+0+ix)
\big\rangle
\Big)
\Big], 
\end{align*}
where $f_{L} \in \mathfrak{P}^{*}$ is defined as \eqref{eq:05170402}. 
\end{corollary}

\begin{proof}
The function $F: \mathfrak{Q} \to \mathbb{C}$ defined as $g \mapsto \exp \big(\vec{\kappa} \circ \langle -f_{L} \mid g \rangle\big)$ is continuous for any $\vec{\kappa}=(\kappa_{1},\kappa_{2}) \in \mathbb{C}^{2}$ and $L>0$. 
Hence Theorem \ref{thm:3.4} implies
\begin{align*}
&
\exp \Big(
\vec{\kappa} \circ
\big\langle
-f_{L}
\mid
\log{\zeta}_{T}^{\mathsf{unif}}(1/2+0+ix)
\big\rangle
\Big)
\\
&
\xrightarrow{\mathcal{D}}
\exp \Big(
\vec{\kappa} \circ
\big\langle
-f_{L}
\mid
\log{\zeta}^{\mathsf{rand}}(1/2+0+ix)
\big\rangle
\Big)
\end{align*}
as $T \to \infty$. 
By Lemma \ref{lem:3.3}, the desired result holds if we obtain
\begin{align}\label{eq:05170449}
\sup_{T>0} 
\mathbb{E} \bigg[
\Big|\exp \Big(
\vec{\kappa} \circ
\big\langle
-f_{L}
\mid
\log{\zeta}_{T}^{\mathsf{unif}}(1/2+0+ix)
\big\rangle
\Big)\Big|^{2}
\bigg]
<
\infty. 
\end{align}
Assume $|\kappa_{1}| \leq M$ and $|\kappa_{2}| \leq M$ for some $M>0$. 
Then we have
\begin{align}\label{eq:05170442}
\big|\exp (
\vec{\kappa} \circ
\lambda
)\big|^{2}
\leq
\exp \big( 2M |\lambda| \big)
\end{align}
for any $\lambda \in \mathbb{C}$. 
Let $0<c<1/2$ and $K>1/2$. 
We deduce from Lemma \ref{lem:3.5} that
\begin{align*}
&
\big\langle
-f_{L}
\mid
\log{\zeta}_{T}^{\mathsf{unif}}(1/2+0+ix)
\big\rangle
\\
&
=
\int_{-\infty-ic}^{\infty-ic}
\log^{*}{\zeta}(1/2+iz+iT \upsilon)\, f_{L}(z)
\,dz
\\
&
=
\int_{-\infty-iK}^{\infty-iK}
\log{\zeta}(1/2+iz+iT \upsilon)\, f_{L}(z)
\,dz
-
2\pi i
\int_{-\infty-i/2}^{-T \upsilon-i/2}
f_{L}(z)
\,dz. 
\end{align*}
We can apply \eqref{eq:05081554} to derive that $|\log{\zeta}(1/2+iz+iT \upsilon)| \leq \log{\zeta}(1/2+K)$ on the line $\IM(z)=-K$. 
Hence we obtain
\begin{align}\label{eq:05170443}
&
\Big|\big\langle
-f_{L}
\mid
\log{\zeta}_{T}^{\mathsf{unif}}(1/2+0+ix)
\big\rangle\Big|
\\
\nonumber
&
\leq
\log{\zeta}(1/2+K)
\int_{-\infty}^{\infty}
|f_{L}(x-iK)|
\,dx
+
2\pi
\int_{-\infty}^{\infty}
|f_{L}(x-i/2)|
\,dx. 
\end{align}
Combining \eqref{eq:05170442} and \eqref{eq:05170443}, we conclude that the random variable
\begin{align*}
\exp \Big(
\vec{\kappa} \circ
\big\langle
-f_{L}
\mid
\log{\zeta}_{T}^{\mathsf{unif}}(1/2+0+ix)
\big\rangle
\Big)
\end{align*}
is bounded uniformly in $T$. 
Therefore we obtain \eqref{eq:05170449}. 
\end{proof}

\subsection{Limit theorem for CUE characteristic polynomials}\label{sec:3.3}
In this section, we prove the random matrix analogues of Theorem \ref{thm:3.4} and Corollary \ref{cor:3.9} by using $\log{Z}_{N}^{\mathsf{CUE}}(e^{-ix-0})$ and $\log{Z}^{\mathsf{Gauss}}(e^{-ix-0})$ defined in Section \ref{sec:2.2}. 

\begin{theorem}\label{thm:3.10}
$\log{Z}_{N}^{\mathsf{CUE}}(e^{-ix-0}) \xrightarrow{\mathcal{D}} \log{Z}^{\mathsf{Gauss}}(e^{-ix-0})$ as $N \to \infty$. 
\end{theorem}

\begin{proof}
Let $F: \mathfrak{Q} \to \mathbb{C}$ be a bounded Lipschitz continuous function. 
By Lemma \ref{lem:3.1}, it suffices to show
\begin{align}\label{eq:05171523}
\lim_{N \to \infty} 
\mathbb{E} \big[F(\log{Z}_{N}^{\mathsf{CUE}}(e^{-ix-0}))\big]
=
\mathbb{E} \big[F(\log{Z}^{\mathsf{Gauss}}(e^{-ix-0}))\big].
\end{align}
For $y \geq 1$ and a sequence of complex numbers $W(k)$, we introduce a trigonometric polynomial $\Psi_{y}(z;W)$ as 
\begin{align*}
\Psi_{y}(z;W)
=
\sum_{k \leq y} 
\frac{1}{\sqrt{k}} W(k) e^{-ikz}. 
\end{align*}
Furthermore, we define $\mathsf{W}_{N}(k)=\tr(\mathsf{U}^{k})/\sqrt{k}$ with an $N \times N$ random matrix $\mathsf{U}$ from CUE, and take a sequence of independent random variables $\mathsf{W}(k)$ distributed on $\mathbb{C}$ according to the standard complex Gaussian measure. 
Then we obtain the random hyperfunctions $\Psi_{y}(x-i0;\mathsf{W}_{N})$ and $\Psi_{y}(x-i0;\mathsf{W})$ for each $y \geq1$. 
As in the proof of Proposition \ref{prop:3.6}, we note that \eqref{eq:05171523} holds if we obtain
\begin{align}
\label{eq:05171524}
\lim_{y \to \infty}
\limsup_{N \to \infty} 
\Big| 
\mathbb{E} \big[F(\log{Z}_{N}^{\mathsf{CUE}}(e^{-ix-0}))\big]
-
\mathbb{E} \big[F(\Psi_{y}(x-i0; \mathsf{W}_{N}))\big]
\Big|
&
=
0, 
\\
\label{eq:05171525}
\sup_{y \geq1}
\lim_{N \to \infty} 
\Big| 
\mathbb{E} \big[F(\Psi_{y}(x-i0; \mathsf{W}_{N}))\big]
-
\mathbb{E} \big[F(\Psi_{y}(x-i0; \mathsf{W}))\big]
\Big|
&
=
0, 
\\
\label{eq:05171526}
\lim_{y \to \infty}
\Big| 
\mathbb{E} \big[F(\Psi_{y}(x-i0; \mathsf{W}))\big]
-
\mathbb{E} \big[F(\log{Z}^{\mathsf{Gauss}}(e^{-ix-0}))\big]
\Big|
&
=
0. 
\end{align}
Furthermore, in order to prove \eqref{eq:05171524}, it suffices to show
\begin{align*}
\lim_{y \to \infty}
\limsup_{N \to \infty} 
\mathbb{E} \big[\Delta_{N,y}(m)\big] 
=
0
\end{align*}
for all $m \geq1$, where $\Delta_{N,y}(m)$ is a random variable such that
\begin{align*}
\Delta_{N,y}(m)
&
=
\sup_{\|f\|_{m} \leq 1}
\Big|\big \langle f \mid \log{Z}_{N}^{\mathsf{CUE}}(e^{-ix-0})-\Psi_{y}(x-i0; \mathsf{W}_{N}) \big \rangle\Big|
\\
&
\leq
\int_{-\infty}^{\infty}
e^{-|x-ic|/m}
\big|\log{Z}_{N}^{\mathsf{CUE}}(e^{-c-ix})-\Psi_{y}(x-ic; \mathsf{W}_{N})\big|
\,dx
\\
& 
\leq
\sum_{k>y} 
\frac{1}{\sqrt{k}} |\mathsf{W}_{N}(k)| e^{-ck}
\int_{-\infty}^{\infty}
e^{-|x-ic|/m}
\,dx
\end{align*}
with $0<c<2^{-m}$. 
By the Cauchy--Schwarz inequality and the result of Diaconis and Evans \cite[Theorem 2.1 b]{DiaconisEvans2001}, the expected value of $|\mathsf{W}_{N}(k)|$ is estimated as 
\begin{align*}
\mathbb{E} \big[|\mathsf{W}_{N}(k)|\big]
\leq
\sqrt{\mathbb{E} \big[|\mathsf{W}_{N}(k)|^{2}\big]}
=
\sqrt{\frac{1}{k} \min(k,N)}
\end{align*}
for any $k,N \geq1$. 
Therefore, we derive 
\begin{align*}
\limsup_{N \to \infty} 
\mathbb{E} \big[\Delta_{N,y}(m)\big] 
\leq
\sum_{k>y} 
\frac{1}{\sqrt{k}} e^{-ck}
\int_{-\infty}^{\infty}
e^{-|x-ic|/m}
\,dx
\to 
0
\end{align*}
as $y \to \infty$. 
Hence \eqref{eq:05171524} is proved. 
Furthermore, we also apply \cite[Theorem 2.1 a]{DiaconisEvans2001} to derive that
\begin{align*}
(\mathsf{W}_{N}(k_{1}), \ldots, \mathsf{W}_{N}(k_{n}))
\xrightarrow{\mathcal{D}}
(\mathsf{W}(k_{1}), \ldots, \mathsf{X}(k_{n}))
\end{align*}
as $N \to \infty$ for any $k_{1},\ldots,k_{n} \in \mathbb{N}$. 
By the definition of the random hyperfunctions $\Psi_{y}(x-i0; \mathsf{W}_{N})$ and $\Psi_{y}(x-i0; \mathsf{W})$, this yields
\begin{align*}
\Psi_{y}(x-i0; \mathsf{W}_{N})
\xrightarrow{\mathcal{D}}
\Psi_{y}(x-i0; \mathsf{W})
\end{align*}
as $N \to \infty$ for any $y \geq1$. 
Hence we obtain \eqref{eq:05171525} since $F$ is a bounded continuous function. 
The proof of \eqref{eq:05171526} is quite similar to that of \eqref{eq:05171524}. 
We need to show
\begin{align*}
\lim_{y \to \infty}
\mathbb{E} \big[\Delta_{y}(m)\big] 
=
0
\end{align*}
for all $m \geq1$, where $\Delta_{y}(m)$ is a random variable such that
\begin{align*}
\Delta_{y}(m)
&
=
\sup_{\|f\|_{m} \leq 1}
\Big|\big \langle f \mid \Psi_{y}(x-i0; \mathsf{W})-\log{Z}^{\mathsf{Gauss}}(e^{-ix-0}) \big \rangle\Big|
\\
&
\leq
\int_{-\infty}^{\infty}
e^{-|x-ic|/m}
\big|\log{Z}^{\mathsf{Gauss}}(e^{-c-ix})-\Psi_{y}(x-ic; \mathsf{W})\big|
\,dx
\\
& 
\leq
\sum_{k>y} 
\frac{1}{\sqrt{k}} |\mathsf{W}(k)| e^{-ck}
\int_{-\infty}^{\infty}
e^{-|x-ic|/m}
\,dx. 
\end{align*}
Recall that $\mathbb{E} \big[|\mathsf{W}(k)|^{2}\big]=1$ for any $k \geq1$. 
Thus we have 
\begin{align*}
\mathbb{E} \big[|\mathsf{W}(k)|\big]
\leq
\sqrt{\mathbb{E} \big[|\mathsf{W}(k)|^{2}\big]}
=
1
\end{align*}
by the Cauchy--Schwarz inequality. 
Therefore, we derive
\begin{align*}
\mathbb{E} \big[\Delta_{y}(m)\big] 
\leq
\sum_{k>y} 
\frac{1}{\sqrt{k}} e^{-ck}
\int_{-\infty}^{\infty}
e^{-|x-ic|/m}
\,dx
\to 
0
\end{align*}
as $y \to \infty$. 
Hence \eqref{eq:05171526} also follows. 
From the above, we complete the proof of the theorem. 
\end{proof}

\begin{corollary}\label{cor:3.11}
For any $\vec{\kappa}=(\kappa_{1},\kappa_{2}) \in \mathbb{C}^{2}$ and $L>0$, 
\begin{align*}
&
\lim_{N \to \infty} 
\mathbb{E} \Big[
\exp \Big(
\vec{\kappa} \circ
\big\langle
-f_{L}
\mid
\log{Z}_{N}^{\mathsf{CUE}}(e^{-ix-0})
\big\rangle
\Big)
\Big]
\\
&
=
\mathbb{E} \Big[
\exp \Big(
\vec{\kappa} \circ
\big\langle
-f_{L}
\mid
\log{Z}^{\mathsf{Gauss}}(e^{-ix-0})
\big\rangle
\Big)
\Big], 
\end{align*}
where $f_{L} \in \mathfrak{P}^{*}$ is defined as \eqref{eq:05170402}. 
\end{corollary}

\begin{proof}
Note that Theorem \ref{thm:3.10} yields
\begin{align*}
\exp \Big(
\vec{\kappa} \circ
\big\langle
-f_{L}
\mid
\log{Z}_{N}^{\mathsf{CUE}}(e^{-ix-0})
\big\rangle
\Big)
\xrightarrow{\mathcal{D}}
\exp \Big(
\vec{\kappa} \circ
\big\langle
-f_{L}
\mid
\log{Z}^{\mathsf{Gauss}}(e^{-ix-0})
\big\rangle
\Big)
\end{align*}
as $N \to \infty$ for any $\vec{\kappa}=(\kappa_{1},\kappa_{2}) \in \mathbb{C}^{2}$ and $L>0$. 
Hence, by Lemma \ref{lem:3.3}, the result follows if we obtain
\begin{align}\label{eq:05181521}
\sup_{N \geq1}
\mathbb{E} \bigg[
\Big|\exp \Big(
\vec{\kappa} \circ
\big\langle
-f_{L}
\mid
\log{Z}_{N}^{\mathsf{CUE}}(e^{-ix-0})
\big\rangle
\Big)\Big|^{2}
\bigg]
<
\infty. 
\end{align}
By the definition of the notation $\vec{\kappa} \circ \lambda$, we have 
\begin{align*}
\big|
\exp (\vec{\kappa} \circ \lambda)
\big|^2
&
=
\exp \big( 
\kappa_{1} \lambda
+ 
\kappa_{2} \overline{\lambda}
\big)
\exp \Big( 
\overline{
\kappa_{1} \lambda
+ 
\kappa_{2} \overline{\lambda}
}
\Big)
\\
&
=
\exp \big( 
(\kappa_{1}+\overline{\kappa_{2}}, \overline{\kappa_{1}}+\kappa_{2}) 
\circ \lambda 
\big)
\end{align*}
for any $\vec{\kappa}=(\kappa_{1},\kappa_{2}) \in \mathbb{C}^{2}$ and $\lambda \in \mathbb{C}$. 
Using the random variables $\mathsf{W}_{N}(k)$ as in the proof of Theorem \ref{thm:3.10}, we have 
\begin{align*}
\log{Z}_{N}^{\mathsf{CUE}}(e^{-ix-0})
&
=
\sum_{k \leq N}
\frac{1}{\sqrt{k}} \mathsf{W}_{N}(k) e^{-ik(x-i0)}
+
\sum_{k>N}
\frac{1}{\sqrt{k}} \mathsf{W}_{N}(k) e^{-ik(x-i0)}
\\
&
=
\Psi_{N}(x-i0;\mathsf{W}_{N})
+
E_{N}(x-i0;\mathsf{W}_{N}), 
\end{align*}
say. 
Then we derive
\begin{align}\label{eq:05181418}
&
\Big|\exp \Big(
\vec{\kappa} \circ
\big\langle
-f_{L}
\mid
\log{Z}_{N}^{\mathsf{CUE}}(e^{-ix-0})
\big\rangle
\Big)\Big|^{2}
\\
\nonumber
&
=
\exp \Big(
\vec{\kappa'} \circ
\big\langle
-f_{L}
\mid
\Psi_{N}(x-i0;\mathsf{W}_{N})
\big\rangle
\Big)
\\
\nonumber
&
\qquad
\times
\exp \Big(
\vec{\kappa'} \circ
\big\langle
-f_{L}
\mid
E_{N}(x-i0;\mathsf{W}_{N})
\big\rangle
\Big), 
\end{align}
where we put $\vec{\kappa'}=(\kappa'_{1},\kappa'_{2}):=(\kappa_{1}+\overline{\kappa_{2}}, \overline{\kappa_{1}}+\kappa_{2})$. 
Let $c$ be any positive real number. 
By the definition of $\Psi_{N}(x-i0;\mathsf{W}_{N})$, we obtain
\begin{align}\label{eq:05181419}
\big\langle
-f_{L}
\mid
\Psi_{N}(x-i0;\mathsf{W}_{N})
\big\rangle
&
=
\int_{-\infty-ic}^{\infty-ic}
\sum_{k \leq N}
\frac{1}{\sqrt{k}} \mathsf{W}_{N}(k) e^{-ikz}
\, f_{L}(z)
\,dz
\\
\nonumber
&
=
\sum_{k \leq N}
\frac{\mathsf{W}_{N}(k)}{\sqrt{k}} 
\exp \Big(
- \frac{k^{2}}{2L}
\Big) 
\end{align}
due to the formula
\begin{align}\label{eq:05202234}
\int_{-\infty-ic}^{\infty-ic}
f_{L}(z)
e^{-ikz}
\,dz
=
\int_{-\infty}^{\infty}
f_{L}(x)
e^{-ikx}
\,dx
=
\exp \Big(
- \frac{k^{2}}{2L}
\Big). 
\end{align}
Hence the first term of the right-hand side of \eqref{eq:05181418} is calculated as
\begin{align*}
&
\exp \Big(
\vec{\kappa'} \circ
\big\langle
-f_{L}
\mid
\Psi_{N}(x-i0;\mathsf{W}_{N})
\big\rangle
\Big)
\\
=
&
\sum_{\mu,\nu=0}^{\infty} 
\frac{(\kappa'_{1})^{\mu} (\kappa'_{2})^{\nu}}{\mu!\, \nu!}
\bigg\{ 
\sum_{k \leq N}
\frac{\mathsf{W}_{N}(k)}{\sqrt{k}} 
\exp \Big(
- \frac{k^{2}}{2L}
\Big) 
\bigg\}^{\mu}
\bigg\{
\sum_{k \leq N}
\frac{\overline{\mathsf{W}_{N}(k)}}{\sqrt{k}}  
\exp \Big(
- \frac{k^{2}}{2L}
\Big) 
\bigg\}^{\nu}
\\
=
&
\sum_{\mu,\nu=0}^{\infty} 
\frac{(\kappa'_{1})^{\mu} (\kappa'_{2})^{\nu}}{\mu!\, \nu!}
\sum_{\substack{k_{1},\ldots,k_{\mu} \leq N \\ \ell_{1},\ldots,\ell_{\nu} \leq N}} 
\frac{\mathsf{W}_{N}(k_{1}) \cdots \mathsf{W}_{N}(k_{\mu})}{\sqrt{k_{1} \cdots k_{\mu}}} 
\frac{\overline{\mathsf{W}_{N}(\ell_{1}) \cdots \mathsf{W}_{N}(\ell_{\nu}) }}{\sqrt{\ell_{1} \cdots \ell_{\nu}}}
\\
&
\qquad\qquad\qquad\qquad\qquad
\times
\exp \Big(
- \frac{k_{1}^{2}+\cdots+k_{\mu}^{2}+\ell_{1}^{2}+\cdots+\ell_{\nu}^{2}}{2L}
\Big) 
\end{align*}
by the Taylor expansion of the exponential. 
For $k_{1},\ldots,k_{\mu} \leq N$ and $\ell_{1},\ldots,\ell_{\nu} \leq N$, 
\begin{align*}
N 
\geq
\max(k_{1}+\cdots+k_{\mu}, \ell_{1}+\cdots+\ell_{\nu}) 
\end{align*}
is satisfied. 
Therefore, we deduce from \cite[Theorem 2.1 a]{DiaconisEvans2001} that
\begin{align*}
\mathbb{E} \Big[
\mathsf{W}_{N}(k_{1}) \cdots \mathsf{W}_{N}(k_{\mu})
\overline{\mathsf{W}_{N}(\ell_{1}) \cdots \mathsf{W}_{N}(\ell_{\nu}) }
\Big]
=
\mathbb{E} \Big[
\mathsf{W}(k_{1}) \cdots \mathsf{W}(k_{\mu})
\overline{\mathsf{W}(\ell_{1}) \cdots \mathsf{W}(\ell_{\nu}) }
\Big], 
\end{align*}
where the random variables $\mathsf{W}(k)$ on the right-hand side are also the same as in the proof of Theorem \ref{thm:3.10}. 
This yields the identity
\begin{align*}
\mathbb{E} \Big[
\exp \Big(
\vec{\kappa'} \circ
\big\langle
-f_{L}
\mid
\Psi_{N}(x-i0;\mathsf{W}_{N})
\big\rangle
\Big)
\Big]
&
=
\mathbb{E} \Big[
\exp \Big(
\vec{\kappa'} \circ
\big\langle
-f_{L}
\mid
\Psi_{N}(x-i0;\mathsf{W})
\big\rangle
\Big)
\Big]
\end{align*}
from the above calculation. 
Furthermore, by using the independence of the random variables $\mathsf{W}(k)$, the right-hand side is evaluated as 
\begin{align*}
&
\mathbb{E} \Big[
\exp \Big(
\vec{\kappa'} \circ
\big\langle
-f_{L}
\mid
\Psi_{N}(x-i0;\mathsf{W})
\big\rangle
\Big)
\Big]
\\
&
=
\prod_{k \leq N} 
\mathbb{E} \Big[
\exp \Big(
\vec{\kappa'} \circ
\big\langle
-f_{L}
\mid
\frac{1}{\sqrt{k}} \mathsf{W}(k) e^{-ik(x-i0)}
\big\rangle
\Big)
\Big]
\\
&
=
\prod_{k \leq N} 
\mathbb{E} \Big[
\exp \Big(
\vec{\kappa'} \circ
\frac{\mathsf{W}(k)}{\sqrt{k}}  
\exp \Big(
- \frac{k^{2}}{2L}
\Big)
\Big)
\Big]
\\
&
=
\prod_{k \leq N} 
\exp \bigg(
|\kappa_{1}+\overline{\kappa_{2}}|^{2}
\frac{1}{k} 
\exp \Big(
- \frac{k^{2}}{L}
\Big)
\bigg), 
\end{align*}
where we used the formula
\begin{align}\label{eq:05252109}
\mathbb{E} \big[\exp \big(z_{1} \mathsf{W}(k)+z_{2} \overline{\mathsf{W}(k)} \big)\big]
=
\exp \big(z_{1}z_{2} \big)
\end{align}
which is valid for any $z_{1},z_{2} \in \mathbb{C}$. 
As a result, we obtain
\begin{align}\label{eq:05181515}
&
\mathbb{E} \Big[
\exp \Big(
\vec{\kappa'} \circ
\big\langle
-f_{L}
\mid
\Psi_{N}(x-i0;\mathsf{W}_{N})
\big\rangle
\Big)
\Big]
\\
\nonumber
&
\leq
\exp \bigg(
|\kappa_{1}+\overline{\kappa_{2}}|^{2}
\sum_{k=1}^{\infty} \frac{1}{k} 
\exp \Big(
- \frac{k^{2}}{L}
\Big)
\bigg)
<
\infty
\end{align}
for any $N \geq1$. 
Then, we evaluate the second term of the right-hand side of \eqref{eq:05181418}. 
Similarly to \eqref{eq:05181419}, we have
\begin{align*}
\big\langle
-f_{L}
\mid
E_{N}(x-i0;\mathsf{W}_{N})
\big\rangle
=
\sum_{k>N}
\frac{\mathsf{W}_{N}(k)}{\sqrt{k}}  
\exp \Big(
- \frac{k^{2}}{2L}
\Big). 
\end{align*}
Note that the random variable $\mathsf{W}_{N}(k)$ satisfies $|\mathsf{W}_{N}(k)| \leq N/\sqrt{k}$ for any $k,N \geq1$. 
Using \eqref{eq:05170442}, we therefore obtain
\begin{align}\label{eq:05181516}
0
&
\leq
\exp \Big(
\vec{\kappa'} \circ
\big\langle
-f_{L}
\mid
E_{N}(x-i0;\mathsf{W}_{N})
\big\rangle
\Big)
\\
\nonumber
&
\leq
\exp \bigg( 2|\kappa_{1}+\overline{\kappa_{2}}|
\sum_{k>N}
\frac{N}{k}
\exp \Big(
- \frac{k^{2}}{2L}
\Big)
\bigg)
\\
\nonumber
&
\leq
\exp \bigg( 2|\kappa_{1}+\overline{\kappa_{2}}|
\sum_{k=1}^{\infty}
\exp \Big(
- \frac{k^{2}}{2L}
\Big)
\bigg)
<
\infty
\end{align}
for any $N \geq1$. 
Combining \eqref{eq:05181515} and \eqref{eq:05181516}, we deduce from \eqref{eq:05181418} that \eqref{eq:05181521} is satisfied. 
Hence the proof is completed. 
\end{proof}

\section{Proof of the main result}\label{sec:4}

\subsection{Relationship between the random hyperfunctions}\label{sec:4.1}
The limit theorems obtained in Section \ref{sec:3} derive the relationships between $\log{\zeta}_{T}^{\mathsf{unif}}(1/2+0+ix)$ and $\log{\zeta}^{\mathsf{rand}}(1/2+0+ix)$, and between $\log{Z}_{N}^{\mathsf{CUE}}(e^{-ix-0})$ and $\log{Z}^{\mathsf{Gauss}}(e^{-ix-0})$. 
Then we consider the relationship between $\log{\zeta}^{\mathsf{rand}}(1/2+0+ix)$ and $\log{Z}^{\mathsf{Gauss}}(e^{-ix-0})$, which plays an important role in the final step of the proof of Theorem \ref{thm:1.2}. 
The purpose of this section is to show the following result. 

\begin{theorem}\label{thm:4.1}
For any $\vec{\kappa}=(\kappa_{1},\kappa_{2}) \in \mathbb{C}^{2}$, 
\begin{align*}
\lim_{L \to \infty}
\frac{
\displaystyle{
\mathbb{E} \Big[
\exp \Big(
\vec{\kappa} \circ
\big\langle
-f_{L}
\mid
\log{\zeta}^{\mathsf{rand}}(1/2+0+ix)
\big\rangle
\Big)
\Big]
}}{
\displaystyle{
\mathbb{E} \Big[
\exp \Big(
\vec{\kappa} \circ
\big\langle
-f_{L}
\mid
\log{Z}^{\mathsf{Gauss}}(e^{-ix-0})
\big\rangle
\Big)
\Big]
}}
=
a(\vec{\kappa}), 
\end{align*}
where $a(\vec{\kappa})$ is the same as in Conjecture \ref{conj:1.1}, and $f_{L} \in \mathfrak{P}^{*}$ is defined as \eqref{eq:05170402}. 
\end{theorem}

For this, we introduce another random hyperfunction $\Xi^{\mathsf{Gauss}}(x-i0)$ defined from the random analytic function
\begin{align*}
\Xi^{\mathsf{Gauss}}(z)
=
\sum_{p} \sqrt{\log \big(1-p^{-1}\big)^{-1}}\, \mathsf{W}(p) p^{-iz}, 
\end{align*}
where $(\mathsf{W}(p))_{\text{$p$ prime}}$ is a sequence of independent random variables distributed on $\mathbb{C}$ according to the standard complex Gaussian measure. 
Then, almost surely, we have the followings:
\begin{enumerate}
\item \label{lem:2.10_1'}
$\Xi^{\mathsf{Gauss}}(z)$ is a holomorphic function on the half-plane $\IM(z)<0$. 
\item \label{lem:2.10_2'}
Let $\delta_{1}$ and $\delta_{2}$ be fixed with $0<\delta_{1}<\delta_{2}<\infty$. 
Then there exists a positive real valued random variable $\widetilde{\mathsf{M}}$ determined only from $\delta_{1}$ such that 
\begin{align*}
\sup_{-\delta_{2} \leq y \leq -\delta_{1}}
\big|\Xi^{\mathsf{Gauss}}(x+iy)\big|
=
O \big( \widetilde{\mathsf{M}} \big)
\end{align*}
for all $x \in \mathbb{R}$ with the implied constant depending only on $\delta_{1}$. 
\end{enumerate}
We omit the proofs of these properties since they can be shown in a way quite similar to Lemma \ref{lem:2.10}. 
As a result, we can define almost surely
\begin{align*}
\widetilde{\varphi}^{\mathsf{Gauss}}(z)
=
\begin{cases}
0
&
\text{for $0<\IM(z)<1$}, 
\\
\Xi^{\mathsf{Gauss}}(z)
&
\text{for $-1<\IM(z)<0$} 
\end{cases}
\end{align*}
as an element of $\widetilde{\mathfrak{P}}_{1}$. 
Furthermore, we also define $\Xi^{\mathsf{Gauss}}(x-i0)=\widetilde{\varphi}^{\mathsf{Gauss}}(x-i0)$ almost surely. 
In order to see that $\Xi^{\mathsf{Gauss}}(x-i0)$ yields a random hyperfunction, we also introduce the function
\begin{align*}
\widetilde{\varphi}_{p}(z,w)
=
\begin{cases}
0
&
\text{for $0<\IM(z)<1$}, 
\\
\sqrt{\log \big(1-p^{-1} \big)^{-1}}\, w\, p^{-iz}
&
\text{for $-1<\IM(z)<0$} 
\end{cases}
\end{align*}
for each prime number $p$ and $w \in \mathbb{C}$. 
Then it is an element of $\mathfrak{P}_{1}$, and furthermore, the map $\mathbb{C} \to \mathfrak{Q}$ defined by $w \mapsto \widetilde{\varphi}_{p}(x-i0,w)$ is obviously continuous. 
Hence we deduce that $\widetilde{g}_{p}^{\mathsf{Gauss}}(x-i0):=\widetilde{\varphi}_{p}(x-i0,\mathsf{W}(p))$ is a random hyperfunction for every prime number $p$. 
Similarly to Lemma \ref{lem:2.7}, we obtain almost surely 
\begin{align*}
\sum_{p \leq y}
\widetilde{g}_{p}^{\mathsf{Gauss}}(x-i0)
\to
\Xi^{\mathsf{Gauss}}(x-i0)
\end{align*}
as $y \to \infty$ in the space $\mathfrak{Q}$.
Hence we can define 
\begin{align*}
\Xi^{\mathsf{Gauss}}(x-i0): \Omega \to \mathfrak{Q}
\end{align*}
as a random hyperfunction in the same way as at the end of Section \ref{sec:2.2}. 
Then, we begin with local analysis on $\widetilde{g}_{p}^{\mathsf{Gauss}}(x-i0)$ at each prime $p$. 

\begin{lemma}\label{lem:4.2}
Let $p$ be any prime number and $L \geq1$. 
For any $\vec{\kappa}=(\kappa_{1},\kappa_{2}) \in \mathbb{C}^{2}$, 
\begin{align}\label{eq:05202240}
\mathbb{E} \Big[
\exp \Big(
\vec{\kappa} \circ
\big\langle
-f_{L}
\mid
\widetilde{g}_{p}^{\mathsf{Gauss}}(x-i0)
\big\rangle
\Big)
\Big]^{-1}
=
\sum_{m=0}^{\infty}
G_{m}(-\kappa_{1}\kappa_{2};p,L)
p^{-m}, 
\end{align}
where $f_{L} \in \mathfrak{P}^{*}$ is defined as \eqref{eq:05170402}, and we define $G_{m}(\kappa;p,L)$ for $\kappa \in \mathbb{C}$ as
\begin{align*}
G_{m}(\kappa;p,L)
=
\frac{
\displaystyle{\Gamma \Big(\kappa \exp \Big(-\frac{(\log{p})^{2}}{L} \Big)+m \Big)}
}{
\displaystyle{m!\, \Gamma \Big(\kappa \exp \Big(-\frac{(\log{p})^{2}}{L} \Big) \Big)}
}.
\end{align*}
\end{lemma}

\begin{proof}
By the definition of $\widetilde{g}_{p}^{\mathsf{Gauss}}(x-i0)$, we have 
\begin{align*}
\big\langle
-f_{L}
\mid
\widetilde{g}_{p}^{\mathsf{Gauss}}(x-i0)
\big\rangle
=
\sqrt{\log \big(1-p^{-1} \big)^{-1}}\, \mathsf{W}(p)
\exp \Big(-\frac{(\log{p})^{2}}{2L}\Big)
\end{align*}
by using \eqref{eq:05202234}. 
Hence the left-hand side of \eqref{eq:05202240} is calculated as
\begin{align*}
&
\mathbb{E} \Big[
\exp \Big(
\vec{\kappa} \circ
\big\langle
-f_{L}
\mid
\widetilde{g}_{p}^{\mathsf{Gauss}}(x-i0)
\big\rangle
\Big)
\Big]^{-1}
\\
&
=
\exp \bigg(
-\kappa_{1} \kappa_{2}
\log \big(1-p^{-1} \big)^{-1}\, 
\exp \Big(-\frac{(\log{p})^{2}}{L}\Big)
\bigg)
\end{align*}
due to \eqref{eq:05252109}. 
Therefore, applying the Taylor expansion
\begin{align}\label{eq:05211442}
\exp \big(z \log(1-t)^{-1}\big)
=
\sum_{m=0}^{\infty} 
\frac{\Gamma(z+m)}{m!\, \Gamma(z)} t^{m}
\quad
(|t|<1),
\end{align}
we obtain the desired result. 
\end{proof}

Let $\log{\zeta}_{p}^{\mathsf{rand}}(1/2+0+ix)$ denote the random hyperfunction as in \eqref{eq:05202054}. 

\begin{lemma}\label{lem:4.3}
Let $p$ be any prime number and $L \geq1$. 
For any $\vec{\kappa}=(\kappa_{1},\kappa_{2}) \in \mathbb{C}^{2}$, 
\begin{align*}
\mathbb{E} \Big[
\exp \Big(
\vec{\kappa} \circ
\big\langle
-f_{L}
\mid
\log{\zeta}_{p}^{\mathsf{rand}}(1/2+0+ix)
\big\rangle
\Big)
\Big]
=
\sum_{m=0}^{\infty}
H_{m}(\kappa_{1};p,L) H_{m}(\kappa_{2};p,L)
p^{-m}, 
\end{align*}
where $f_{L} \in \mathfrak{P}^{*}$ is defined as \eqref{eq:05170402}, and we define $H_{m}(\kappa;p,L)$ for $\kappa \in \mathbb{C}$ as 
\begin{align*}
H_{m}(\kappa;p,L)
=
\sum_{n=1}^{m} 
\frac{\kappa^{n}}{n!}
\sum_{\substack{k_{1}+\cdots+k_{n}=m \\ \forall j,~ k_{j} \geq 1}}
\frac{1}{k_{1} \cdots k_{n}}
\exp \Big(-\frac{k_{1}^2+\cdots+k_{n}^2}{2L} (\log{p})^{2}\Big)
\end{align*}
with the convention $H_{0}(\kappa;p,L)=1$. 
\end{lemma}

\begin{proof}
Firstly, we calculate the pairing of $-f_{L}$ and $\log{\zeta}_{p}^{\mathsf{rand}}(1/2+0+ix)$ as 
\begin{align*}
\big\langle
-f_{L}
\mid
\log{\zeta}_{p}^{\mathsf{rand}}(1/2+0+ix)
\big\rangle
&
=
\int_{-\infty-ic}^{\infty-ic}
\sum_{k=1}^{\infty}
\frac{1}{k} 
\mathsf{X}(p)^{k}
p^{-k(1/2+iz)}
f_{L}(z)
\,dz
\\
&
=
\sum_{k=1}^{\infty}
\frac{1}{k} 
p^{-k/2} \mathsf{X}(p)^{k}
\exp \Big(-\frac{(\log{p^{k}})^{2}}{2L}\Big) 
\end{align*}
by using \eqref{eq:05202234}. 
By the Taylor expansion of the exponential, we therefore obtain
\begin{align*}
&
\exp \Big(
\kappa 
\big\langle
-f_{L}
\mid
\log{\zeta}_{p}^{\mathsf{rand}}(1/2+0+ix)
\big\rangle
\Big)
\\
=
&
\sum_{n=0}^{\infty} 
\frac{\kappa^{n}}{n!}
\bigg\{ 
\sum_{k=1}^{\infty}
\frac{1}{k} 
p^{-k/2} \mathsf{X}(p)^{k}
\exp \Big(-\frac{(\log{p^{k}})^{2}}{2L}\Big)
\bigg\}^{n}
\\
=
&
1
+
\sum_{n=1}^{\infty} 
\frac{\kappa^{n}}{n!}
\sum_{k_{1},\ldots,k_{n} \geq 1} 
\frac{p^{-(k_{1}+\cdots+k_{n})/2}}{k_{1} \cdots k_{n}} 
\mathsf{X}(p)^{k_{1}+\cdots+k_{n}}
\\
&
\qquad\qquad\qquad\qquad
\times
\exp \Big(
- \frac{k_{1}^{2}+\cdots+k_{n}^{2}}{2L}(\log{p})^{2}
\Big) 
\\
=
&
\sum_{m=0}^{\infty} 
H_{m}(\kappa;p,L) \mathsf{X}(p)^{m} p^{-m/2}
\end{align*}
for any $\kappa \in \mathbb{C}$. 
Furthermore, this yields
\begin{align*}
&
\exp \Big(
\vec{\kappa} \circ
\big\langle
-f_{L}
\mid
\log{\zeta}_{p}^{\mathsf{rand}}(1/2+0+ix)
\big\rangle
\Big)
\\
&
=
\exp \Big(
\kappa_{1} 
\big\langle
-f_{L}
\mid
\log{\zeta}_{p}^{\mathsf{rand}}(1/2+0+ix)
\big\rangle
\Big)
\\
&
\quad
\times
\exp \Big(
\kappa_{2} 
\overline{\big\langle
-f_{L}
\mid
\log{\zeta}_{p}^{\mathsf{rand}}(1/2+0+ix)
\big\rangle}
\Big)
\\
&
=
\sum_{m,n=0}^{\infty} 
H_{m}(\kappa_{1};p,L) \overline{H_{n}(\overline{\kappa_{2}};p,L)}\, 
\mathsf{X}(p)^{m} \overline{\mathsf{X}(p)}^{n} p^{-(m+n)/2}. 
\end{align*}
Note that $\overline{H_{n}(\overline{\kappa_{2}};p,L)}=H_{n}(\kappa_{2};p,L)$ by definition. 
Then, recalling that 
\begin{align*}
\mathbb{E} \Big[
\mathsf{X}(p)^{m} \overline{\mathsf{X}(p)}^{n}
\Big]
=
\begin{cases}
1
&
\text{if $m=n$}, 
\\
0
&
\text{if $m \neq n$}, 
\end{cases}
\end{align*}
we calculate the expected value of the above as
\begin{align*}
&
\mathbb{E} \Big[
\exp \Big(
\vec{\kappa} \circ
\big\langle
-f_{L}
\mid
\log{\zeta}_{p}^{\mathsf{rand}}(1/2+0+ix)
\big\rangle
\Big)
\Big]
\\
&
=
\sum_{m,n=0}^{\infty} 
H_{m}(\kappa_{1};p,L) H_{n}(\overline{\kappa_{2}};p,L)\, 
\mathbb{E} \Big[
\mathsf{X}(p)^{m} \overline{\mathsf{X}(p)}^{n}
\Big]
p^{-(m+n)/2}. 
\\
&
=
\sum_{m=0}^{\infty}
H_{m}(\kappa_{1};p,L) H_{m}(\kappa_{2};p,L)
p^{-m},
\end{align*}
which is the desired result. 
\end{proof}

We also require the following global result which is obtained as an application of the prime number theorem. 

\begin{lemma}\label{lem:4.4}
For any $L \geq1$, 
\begin{align}\label{eq:05202330}
\sum_{p} 
\log \big(1-p^{-1}\big)^{-1}
\exp \Big(- \frac{(\log{p})^{2}}{L}\Big)
=
\sum_{k=1}^{\infty}
\frac{1}{k}
\exp \Big(- \frac{k^{2}}{L}\Big)
+
O \Big(\frac{1}{\sqrt{L}}\Big)
\end{align}
with the absolute implied constant. 
\end{lemma}

\begin{proof}
By summation by parts, the left-hand side of \eqref{eq:05202330} is evaluated as
\begin{align*}
\sum_{p} 
\log \big(1-p^{-1}\big)^{-1}
\exp \Big(- \frac{(\log{p})^{2}}{L}\Big)
=
\frac{2}{L}
\int_{2}^{\infty}
A(x) \exp \Big(- \frac{(\log{x})^{2}}{L}\Big)
\frac{\log{x}}{x}
\,dx, 
\end{align*}
where we put $A(x)= \sum_{p \leq x} \log \big(1-p^{-1}\big)^{-1}$ for $x \geq2$. 
On the other hand, the main term of the right-hand side of \eqref{eq:05202330} is also evaluated as
\begin{align*}
\sum_{k=1}^{\infty}
\frac{1}{k}
\exp \Big(- \frac{k^{2}}{L}\Big)
&
=
\frac{2}{L}
\int_{1}^{\infty}
B(x) \exp \Big(- \frac{x^{2}}{L}\Big)
x 
\,dx
\\
&
=
\frac{2}{L}
\int_{e}^{\infty}
B(\log{x}) \exp \Big(- \frac{(\log{x})^{2}}{L}\Big)
\frac{\log{x}}{x} 
\,dx, 
\end{align*}
where we put $B(x)=\sum_{k \leq x} 1/k$ for $x \geq1$. 
These imply the formula
\begin{align*}
\sum_{p} 
\log \big(1-p^{-1}\big)^{-1}
\exp \Big(- \frac{(\log{p})^{2}}{L}\Big)
=
\sum_{k=1}^{\infty}
\frac{1}{k}
\exp \Big(- \frac{k^{2}}{L}\Big)
+
E,
\end{align*}
where $E$ is estimated as
\begin{align}\label{eq:05210055}
E
\ll
\frac{1}{L}
+
\frac{1}{L}
\int_{e}^{\infty}
|A(x)-B(\log{x})| 
\exp \Big(- \frac{(\log{x})^{2}}{L}\Big)
\frac{\log{x}}{x} 
\,dx. 
\end{align}
Let $\gamma=0.5772...$ denote Euler's constant. 
Then we have the asymptotic formulas
\begin{align*}
A(x)
=
\log\log{x}
+
\gamma
+
O((\log{x})^{-1})
\quad\text{and}\quad
B(x)
=
\log{x}
+
\gamma
+
O(x^{-1})
\end{align*}
which are valid for $x \geq K$ with some absolute constant $K \geq 2$. 
Indeed, the former immediately follows from the prime number theorem, and the latter is due to the Euler–Maclaurin formula. 
Inserting them to \eqref{eq:05210055}, we derive
\begin{align*}
E
&
\ll
\frac{1}{L}
+
\frac{1}{L}
\int_{K}^{\infty}
\exp \Big(- \frac{(\log{x})^{2}}{L}\Big)
\,\frac{dx}{x} 
=
\frac{1}{L}
+
\frac{1}{\sqrt{L}}
\int_{\frac{\log{K}}{\sqrt{L}}}^{\infty}
e^{-x^{2}}
\,dx
\ll
\frac{1}{\sqrt{L}}
\end{align*}
as desired. 
\end{proof}

Then, before proceeding to the proof of Theorem \ref{thm:4.1}, we associate the random hyperfunctions $\log{Z}^{\mathsf{Gauss}}(e^{-ix-0})$ and $\Xi^{\mathsf{Gauss}}(x-i0)$ by using Lemma \ref{lem:4.4}. 

\begin{proposition}\label{prop:4.5}
For any $\vec{\kappa}=(\kappa_{1},\kappa_{2}) \in \mathbb{C}^{2}$, 
\begin{align*}
\lim_{L \to \infty}
\frac{
\displaystyle{
\mathbb{E} \Big[
\exp \Big(
\vec{\kappa} \circ
\big\langle
-f_{L}
\mid
\log{Z}^{\mathsf{Gauss}}(e^{-ix-0})
\big\rangle
\Big)
\Big]
}}{
\displaystyle{
\mathbb{E} \Big[
\exp \Big(
\vec{\kappa} \circ
\big\langle
-f_{L}
\mid
\Xi^{\mathsf{Gauss}}(x-i0)
\big\rangle
\Big)
\Big]
}}
=
1, 
\end{align*}
where $f_{L} \in \mathfrak{P}^{*}$ is defined as \eqref{eq:05170402}. 
\end{proposition}

\begin{proof}
By the definitions of $\log{Z}^{\mathsf{Gauss}}(e^{-ix-0})$ and $\Xi^{\mathsf{Gauss}}(x-i0)$, we have 
\begin{align*}
\big\langle
-f_{L}
\mid
\log{Z}^{\mathsf{Gauss}}(e^{-ix-0})
\big\rangle
&
=
\sum_{k=1}^{\infty} 
\frac{1}{\sqrt{k}}
\mathsf{W}(k)
\exp \Big(- \frac{k^{2}}{2L}\Big), 
\\
\big\langle
-f_{L}
\mid
\Xi^{\mathsf{Gauss}}(x-i0)
\big\rangle
&
=
\sum_{p} 
\sqrt{\log \big(1-p^{-1}\big)^{-1}}
\mathsf{W}(p)
\exp \Big(- \frac{(\log{p})^{2}}{2L}\Big) 
\end{align*}
by using \eqref{eq:05202234}. 
Therefore, by the independences of $(\mathsf{W}(k))_{k}$ and $(\mathsf{W}(p))_{p}$, we derive
\begin{align*}
&
\frac{
\displaystyle{
\mathbb{E} \Big[
\exp \Big(
\vec{\kappa} \circ
\big\langle
-f_{L}
\mid
\log{Z}^{\mathsf{Gauss}}(e^{-ix-0})
\big\rangle
\Big)
\Big]
}}{
\displaystyle{
\mathbb{E} \Big[
\exp \Big(
\vec{\kappa} \circ
\big\langle
-f_{L}
\mid
\Xi^{\mathsf{Gauss}}(x-i0)
\big\rangle
\Big)
\Big]
}}
\\
&
=
\exp\bigg(
\kappa_{1} \kappa_{2}
\bigg\{
\sum_{k=1}^{\infty}
\frac{1}{k}
\exp \Big(- \frac{k^{2}}{L}\Big)
-
\sum_{p} 
\log \big(1-p^{-1}\big)^{-1}
\exp \Big(- \frac{(\log{p})^{2}}{L}\Big)
\bigg\}
\bigg)
\end{align*} 
using \eqref{eq:05252109}. 
By Lemma \ref{lem:4.4}, this converges to $1$ as $L \to \infty$. 
\end{proof}

\begin{proof}[Proof of Theorem \ref{thm:4.1}]
Let $p$ be a prime number and $L \geq1$. 
We define
\begin{align*}
a_{p,L}(\vec{\kappa})
&
=
\frac{
\displaystyle{
\mathbb{E} \Big[
\exp \Big(
\vec{\kappa} \circ
\big\langle
-f_{L}
\mid
\log{\zeta}_{p}^{\mathsf{rand}}(1/2+0+ix)
\big\rangle
\Big)
\Big]
}}{
\displaystyle{
\mathbb{E} \Big[
\exp \Big(
\vec{\kappa} \circ
\big\langle
-f_{L}
\mid
\widetilde{g}_{p}^{\mathsf{Gauss}}(x-i0)
\big\rangle
\Big)
\Big]
}},
\\
a_{p}(\vec{\kappa})
&
=
\big(1-p^{-1}\big)^{\kappa_{1} \kappa_{2}}
\sum_{m=0}^{\infty}
\frac{\Gamma(\kappa_{1}+m)}{m!\, \Gamma(\kappa_{1})}
\frac{\Gamma(\kappa_{2}+m)}{m!\, \Gamma(\kappa_{2})}
p^{-m}. 
\end{align*}
To begin with, we will show $a_{p,L}(\vec{\kappa}) \to a_{p}(\vec{\kappa})$ as $L \to \infty$ for each prime number $p$. 
By Lemmas \ref{lem:4.2} and \ref{lem:4.3}, we represent $a_{p,L}(\vec{\kappa})$ as 
\begin{align}\label{eq:05211455}
a_{p,L}(\vec{\kappa})
=
\sum_{m=0}^{\infty}
G_{m}(-\kappa_{1}\kappa_{2};p,L)
p^{-m}
\cdot
\sum_{m=0}^{\infty}
H_{m}(\kappa_{1};p,L) H_{m}(\kappa_{2};p,L)
p^{-m}. 
\end{align}
The coefficients $G_{m}(-\kappa_{1}\kappa_{2};p,L)$ satisfy the conditions
\begin{align}\label{eq:05211514}
\sup_{L \geq1}
\big| G_{m}(-\kappa_{1}\kappa_{2};p,L) \big|
&
\leq
\frac{\Gamma(|\kappa_{1}\kappa_{2}|+m)}{m!\, \Gamma(|\kappa_{1}\kappa_{2}|)}, 
\\
\nonumber
\lim_{L \to \infty}
G_{m}(-\kappa_{1}\kappa_{2};p,L)
&
=
\frac{\Gamma(-\kappa_{1}\kappa_{2}+m)}{m!\, \Gamma(-\kappa_{1}\kappa_{2})} 
\end{align}
for all $m \geq 0$. 
Hence the dominated convergence theorem yields
\begin{align}\label{eq:05211456}
\lim_{L \to \infty}
\sum_{m=0}^{\infty}
G_{m}(-\kappa_{1}\kappa_{2};p,L)
p^{-m}
=
\sum_{m=0}^{\infty}
\frac{\Gamma(-\kappa_{1}\kappa_{2}+m)}{m!\, \Gamma(-\kappa_{1}\kappa_{2})} 
p^{-m}
=
\big(1-p^{-1}\big)^{\kappa_{1}\kappa_{2}}
\end{align}
due to \eqref{eq:05211442}. 
Then, we also see that the coefficients $H_{m}(\kappa_{1};p,L) H_{m}(\kappa_{2};p,L)$ satisfy 
\begin{align}\label{eq:05211515}
\sup_{L \geq1}
\big| H_{m}(\kappa_{1};p,L) H_{m}(\kappa_{2};p,L) \big|
&
\leq
\frac{\Gamma(|\kappa_{1}|+m)}{m!\, \Gamma(|\kappa_{1}|)} 
\frac{\Gamma(|\kappa_{2}|+m)}{m!\, \Gamma(|\kappa_{2}|)}, 
\\
\nonumber
\lim_{L \to \infty}
H_{m}(\kappa_{1};p,L) H_{m}(\kappa_{2};p,L)
&
=
\frac{\Gamma(\kappa_{1}+m)}{m!\, \Gamma(\kappa_{1})} 
\frac{\Gamma(\kappa_{2}+m)}{m!\, \Gamma(\kappa_{2})} 
\end{align}
for all $m \geq 0$ by recalling the identity
\begin{align*}
\sum_{n=1}^{m} 
\frac{\kappa^{n}}{n!}
\sum_{\substack{k_{1}+\cdots+k_{n}=m \\ \forall j,~ k_{j} \geq 1}}
\frac{1}{k_{1} \cdots k_{n}}
=
\frac{\Gamma(\kappa+m)}{m!\, \Gamma(\kappa)}. 
\end{align*}
Therefore, we again apply the dominated convergence theorem to derive
\begin{align}\label{eq:05211457}
\lim_{L \to \infty}
\sum_{m=0}^{\infty}
H_{m}(\kappa_{1};p,L) H_{m}(\kappa_{2};p,L)
p^{-m} 
=
\sum_{m=0}^{\infty}
\frac{\Gamma(\kappa_{1}+m)}{m!\, \Gamma(\kappa_{1})}
\frac{\Gamma(\kappa_{2}+m)}{m!\, \Gamma(\kappa_{2})}
p^{-m}. 
\end{align}
By \eqref{eq:05211455}, \eqref{eq:05211456} and \eqref{eq:05211457}, we have $a_{p,L}(\vec{\kappa}) \to a_{p}(\vec{\kappa})$ as $L \to \infty$. 
Then, we note that
\begin{align*}
\frac{
\displaystyle{
\mathbb{E} \Big[
\exp \Big(
\vec{\kappa} \circ
\big\langle
-f_{L}
\mid
\log{\zeta}^{\mathsf{rand}}(1/2+0+ix)
\big\rangle
\Big)
\Big]
}}{
\displaystyle{
\mathbb{E} \Big[
\exp \Big(
\vec{\kappa} \circ
\big\langle
-f_{L}
\mid
\Xi^{\mathsf{Gauss}}(x-i0)
\big\rangle
\Big)
\Big]
}}
=
\prod_{p} 
a_{p,L}(\vec{\kappa})
\end{align*}
by the independences of $(\mathsf{X}(p))_{p}$ and $(\mathsf{W}(p))_{p}$. 
Furthermore, we have 
\begin{align*}
a(\vec{\kappa})
=
\prod_{p} 
a_{p}(\vec{\kappa})
\end{align*}
by definition. 
In order to show that $\prod_{p} a_{p,L}(\vec{\kappa}) \to \prod_{p} a_{p}(\vec{\kappa})$ as $L \to \infty$, it suffices to find a sequence of positive real numbers $\beta_{p}(\vec{\kappa})$ such that
\begin{align}\label{eq:05211529}
\sup_{L \geq1}
\big|a_{p,L}(\vec{\kappa})-1\big|
\leq
\beta_{p}(\vec{\kappa})
\quad\text{and}\quad
\sum_{p} \beta_{p}(\vec{\kappa})
<
\infty 
\end{align}
by the dominated convergence theorem. 
By \eqref{eq:05211455}, we obtain
\begin{align*}
a_{p,L}(\vec{\kappa})
=
\sum_{m=0}^{\infty}
D_{m}(\vec{\kappa};p,L)
p^{-m}
\end{align*}
with the coefficients $D_{m}(\vec{\kappa};p,L)$ determined by
\begin{align*}
D_{m}(\vec{\kappa};p,L)
=
\sum_{n=0}^{m}
G_{n}(-\kappa_{1}\kappa_{2};p,L)
H_{m-n}(\kappa_{1};p,L) H_{m-n}(\kappa_{2};p,L). 
\end{align*}
Obviously, we have $D_{0}(\vec{\kappa};p,L)=1$. 
We also obtain
\begin{align*}
D_{1}(\vec{\kappa};p,L)
=
G_{1}(-\kappa_{1}\kappa_{2};p,L)
+
H_{1}(\kappa_{1};p,L) H_{1}(\kappa_{2};p,L)
=
0.
\end{align*}
Furthermore, using \eqref{eq:05211514} and \eqref{eq:05211515}, 
\begin{align*}
\sup_{L \geq1}
\big|D_{m}(\vec{\kappa};p,L)\big|
&
\leq
\sum_{n=0}^{m}
\frac{\Gamma(|\kappa_{1}\kappa_{2}|+n)}{n!\, \Gamma(|\kappa_{1}\kappa_{2}|)}
\frac{\Gamma(|\kappa_{1}|+m-n)}{(m-n)!\, \Gamma(|\kappa_{1}|)} 
\frac{\Gamma(|\kappa_{2}|+m-n)}{(m-n)!\, \Gamma(|\kappa_{2}|)} 
\\
&
=
B_{m}(\vec{\kappa}),
\end{align*}
say, for all $m \geq 0$. 
From the above, we have 
\begin{align*}
\sup_{L \geq1}
\big|a_{p,L}(\vec{\kappa})-1\big|
\leq
\sum_{m=2}^{\infty}
B_{m}(\vec{\kappa})
p^{-m}
\leq
4p^{-2}
\sum_{m=0}^{\infty}
B_{m}(\vec{\kappa})
2^{-m}. 
\end{align*}
Then it can be confirmed that
\begin{align*}
\sum_{m=0}^{\infty}
B_{m}(\vec{\kappa})
2^{-m}
&
=
\sum_{m=0}^{\infty}
\frac{\Gamma(|\kappa_{1}\kappa_{2}|+m)}{m!\, \Gamma(|\kappa_{1}\kappa_{2}|)}
2^{-m}
\cdot
\sum_{m=0}^{\infty}
\frac{\Gamma(|\kappa_{1}|+m)}{m!\, \Gamma(|\kappa_{1}|)} 
\frac{\Gamma(|\kappa_{2}|+m)}{m!\, \Gamma(|\kappa_{2}|)} 
2^{-m}
\\
&
<
\infty. 
\end{align*}
Thus we obtain \eqref{eq:05211529} due to $\sum_{p} p^{-2}<\infty$. 
As a result, we arrive at the formula
\begin{align*}
\lim_{L \to \infty}
\frac{
\displaystyle{
\mathbb{E} \Big[
\exp \Big(
\vec{\kappa} \circ
\big\langle
-f_{L}
\mid
\log{\zeta}^{\mathsf{rand}}(1/2+0+ix)
\big\rangle
\Big)
\Big]
}}{
\displaystyle{
\mathbb{E} \Big[
\exp \Big(
\vec{\kappa} \circ
\big\langle
-f_{L}
\mid
\Xi^{\mathsf{Gauss}}(x-i0)
\big\rangle
\Big)
\Big]
}}
=
a_{p}(\vec{\kappa}). 
\end{align*} 
Then, combined with Proposition \ref{prop:4.5}, the conclusion follows. 
\end{proof}

\subsection{Completion of the proof}\label{sec:4.2}

\begin{lemma}\label{lem:4.6}
Let $F:\mathbb{R} \to \mathbb{C}$ be a Borel measurable function which is continuous at the origin and satisfies
\begin{align*}
\sup_{L \geq 1} 
\int_{-\infty}^{\infty}
f_{L}(x) H \big(|F(x)| \big)
\,dx
<
\infty
\end{align*}
for some Borel measurable function $H: [0,\infty) \to [0,\infty)$ such that $H(x)/x \to \infty$ as $x \to \infty$, where $f_{L} \in \mathfrak{P}^{*}$ is defined as \eqref{eq:05170402}. 
Then we obtain
\begin{align*}
\lim_{L \to \infty}
\int_{-\infty}^{\infty} 
f_{L}(x) F(x) 
\,dx
=
F(0). 
\end{align*}
\end{lemma}

\begin{proof}
Assume that $F$ is continuous and bounded. 
Then we have
\begin{align*}
\int_{-\infty}^{\infty} 
f_{L}(x) F(x) 
\,dx
=
\int_{-\infty}^{\infty} 
f_{1}(x) F \big(x/\sqrt{L} \big) 
\,dx
\to 
\int_{-\infty}^{\infty} 
f_{1}(x) F(0) 
\,dx
=
F(0)
\end{align*}
as $L \to \infty$ by the dominated convergence theorem. 
This shows that the probability measure $f_{L}(x) \,dx$ converges weakly as $L \to \infty$ to the Dirac measure $\delta$ on $(\mathbb{R},\mathcal{B}(\mathbb{R}))$. 
By Lemma \ref{lem:3.2}, we have random variables $X_{L}$ and $X$ taking values in $\mathbb{R}$ defined on a common probability space $(\Omega,\mathcal{F},\mathbb{P})$ such that
\begin{align}\label{eq:05211648}
P_{X_{L}}
=
f_{L}(x) \,dx, 
\quad
P_{X}
=
\delta, 
\end{align}
and $X_{L} \to X$ almost surely as $L \to \infty$. 
Let $F: \mathbb{R} \to \mathbb{C}$ be as in the statement. 
Denote by $U_{F}$ the set of points of discontinuity of $F$. 
Then \eqref{eq:05211648} derives
\begin{align*}
\mathbb{P}(X \in U_{F})
=
\delta(U_{F})
=
0
\end{align*}
since $F$ is continuous at the origin. 
Hence, we deduce from \cite[Theorem 13.25]{Klenke2020} that $F(X_{L}) \xrightarrow{\mathcal{D}} F(X)$ as $L \to \infty$. 
Furthermore, we have
\begin{align*}
\sup_{L \geq1} 
\mathbb{E} \big[H \big( |F(X_{L})| \big)\big]
=
\sup_{L \geq 1} 
\int_{-\infty}^{\infty}
f_{L}(x) H \big(|F(x)| \big)
\,dx
<
\infty
\end{align*}
by \eqref{eq:05211648} and the assumption. 
Therefore Lemma \ref{lem:3.3} yields 
\begin{align*}
\lim_{L \to \infty}
\mathbb{E}[F(X_{L})]
=
\mathbb{E}[F(X)], 
\end{align*}
which is just the desired formula by \eqref{eq:05211648}. 
\end{proof}

Then we apply this lemma to $\log^{*}{\zeta}(1/2+ix+it)$ and $\log{Z}(e^{-ix};U)$. 

\begin{lemma}\label{lem:4.7}
Let $f_{L} \in \mathfrak{P}^{*}$ be defined as \eqref{eq:05170402}. 
Assuming RH, we have
\begin{align*}
\lim_{L \to \infty}
\int_{-\infty}^{\infty} 
f_{L}(x) \log^{*}{\zeta}(1/2+ix+it)
\,dx
=
\log{\zeta}(1/2+it)
\end{align*}
for any $t>0$ with $\zeta(1/2+it) \neq 0$. 
Furthermore, we have
\begin{align*}
\lim_{L \to \infty}
\int_{-\infty}^{\infty} 
f_{L}(x) \log{Z}(e^{-ix};U)
\,dx
=
\log{Z}(1;U)
\end{align*}
for any $U \in \mathcal{U}(N)$ with $Z(1;U) \neq 0$. 
\end{lemma}

\begin{proof}
Assuming RH, we denote by $\rho=1/2+i \gamma$ the nontrivial zeros of $\zeta(s)$.
Let $t \neq \gamma$ for all $\gamma$. 
Then $f_{L}(x) \log^{*}{\zeta}(1/2+ix+it)$ is regarded as a Borel measurable function which is continuous at $x=0$. 
By Lemma \ref{lem:4.6}, the result follows if we have
\begin{align}\label{eq:05232155}
\sup_{L \geq1}
\int_{-\infty}^{\infty}
f_{L}(x) |\log^{*}{\zeta}(1/2+ix+it)|^{2}
\,dx
<
\infty 
\end{align}
due to $\log^{*}{\zeta}(1/2+it)=\log{\zeta}(1/2+it)$ for $t>0$ by \eqref{eq:05152038}. 
Let $x \in \mathbb{R}$ with $x+t \neq \gamma$ for all $\gamma$. 
Then we deduce from \eqref{eq:06011532} that
\begin{align*}
&
|\log^{*}{\zeta}(1/2+ix+it)|^{2}
\\
&
\leq
\bigg\{
\sum_{ |\gamma-x-t| \leq 1}
\big	| \log|\gamma-x-t| \big|
+
O \big(\log(|x+t|+2)\big)
\bigg\}^{2}
\\
&
\ll
\log(|x+t|+2)
\sum_{ |\gamma-x-t| \leq 1}
(\log|\gamma-x-t|)^{2}
+
\big(\log(|x+t|+2)\big)^{2}
\end{align*}
by the Cauchy--Schwarz inequality and \eqref{eq:05081444}, where the implied constant is absolute. 
This yields that the integral of the left-hand side of \eqref{eq:05232155} is estimated as
\begin{align*}
&
\int_{-\infty}^{\infty}
f_{L}(x) |\log^{*}{\zeta}(1/2+ix+it)|^{2}
\,dx
\\
&
\ll
\int_{-\infty}^{\infty}
f_{L}(x) \log(|x+t|+2)
\sum_{ |\gamma-x-t| \leq 1}
(\log|\gamma-x-t|)^{2}
\,dx
\\
&
\quad
+
\int_{-\infty}^{\infty}
f_{L}(x) \big(\log(|x+t|+2)\big)^{2}
\,dx
\\
&
=
A_{L}(t)+B_{L}(t),
\end{align*}
say. 
Firstly, we consider the contribution of the first term. 
It is evaluated as
\begin{align}\label{eq:05232146}
A_{L}(t)
&
=
\sum_{\gamma}
\int_{-1}^{1}
f_{L}(y+\gamma-t) \log(|y+\gamma|+2)
(\log|y|)^{2}
\,dy
\\
\nonumber
&
\leq
\sum_{\gamma}
\log(|\gamma|+3)
\int_{-1}^{1}
f_{L}(y+\gamma-t) 
(\log|y|)^{2}
\,dy. 
\end{align}
Then we divide the sum over $\gamma$ into those for $|t-\gamma| \leq 2$ and those for $|t-\gamma|>2$. 
For each $\gamma$ such that $|t-\gamma| \leq 2$, we obtain
\begin{align*}
&
\int_{-1}^{1}
f_{L}(y+\gamma-t) 
(\log|y|)^{2}
\,dy
\\
&
=
\bigg(
\int_{-1}^{-|t-\gamma|/3}
+
\int_{-|t-\gamma|/3}^{|t-\gamma|/3}
+
\int_{|t-\gamma|/3}^{1}
\bigg)
f_{L}(y+\gamma-t) 
(\log|y|)^{2}
\,dy
\\
&
\leq
\Big(\log{\frac{|t-\gamma|}{3}}\Big)^{2}
\int_{-\infty}^{\infty}
f_{L}(y+\gamma-t) 
\,dy
+
f_{L} \Big(\frac{2|t-\gamma|}{3} \Big) 
\int_{-|t-\gamma|/3}^{|t-\gamma|/3}
(\log|y|)^{2}
\,dy
\\
&
\ll
(\log{|t-\gamma|})^{2}
\end{align*}
with the absolute implied constant, where we used the estimate 
\begin{align}\label{eq:005252129}
f_{L}(x) 
=
\sqrt{\frac{L}{2\pi}} 
\exp \Big(-\frac{L x^{2}}{2} \Big)
\leq
\sqrt{\frac{L}{2\pi}} 
\Big(\frac{L x^{2}}{2}\Big)^{-1/2}
\leq 
|x|^{-1}
\end{align}
which is valid for all $x \neq 0$ and $L \geq1$. 
On the other hand, for each $\gamma$ such that $|t-\gamma|>2$, we obtain uniformly
\begin{align*}
\int_{-1}^{1}
f_{L}(y+\gamma-t) 
(\log|y|)^{2}
\,dy
&
\leq
f_{L}(|t-\gamma|-1)
\int_{-1}^{1}
(\log{|y|})^{2}
\,dy
\\
&
\ll
(t-\gamma)^{-2}
\end{align*}
by using the estimate
\begin{align}\label{eq:005252130}
f_{L}(x) 
=
\sqrt{\frac{L}{2\pi}} 
\exp \Big(-\frac{L x^{2}}{2} \Big)
\leq
\sqrt{\frac{L}{2\pi}} 
\Big(\frac{L x^{2}}{2}\Big)^{-1}
\leq 
x^{-2}
\end{align}
which is valid for all $x \neq 0$ and $L \geq1$. 
Combining these results, we deduce from \eqref{eq:05232146} that
\begin{align*}
A_{L}(t)
\ll
\sum_{|t-\gamma| \leq 2} 
\log(|\gamma|+3)(\log{|t-\gamma|})^{2}
+
\sum_{|t-\gamma|>2} 
\frac{\log(|\gamma|+3)}{(t-\gamma)^2}
<
\infty, 
\end{align*}
where the implied constant is absolute. 
Furthermore, we derive
\begin{align*}
B_{L}(t)
&
=
\int_{-\infty}^{\infty}
f_{1}(x) \Big\{\log \Big( \Big|\frac{x}{\sqrt{L}}+t \Big| +2 \Big)\Big\}^{2}
\,dx
\\
&
\leq
\int_{-\infty}^{\infty}
f_{1}(x) \big(\log(|x|+|t|+2)\big)^{2}
\,dx
<
\infty
\end{align*}
by changing the variables. 
Hence \eqref{eq:05232155} is obtained. 
The result for $\log{Z}(e^{-ix};U)$ is obtained similarly. 
Denote by $e^{i \theta_{1}},\ldots,e^{i \theta_{N}}$ the eigenvalues of $U \in \mathcal{U}(N)$ with $Z(1;U)\neq0$. 
Then we can assume $-\pi<\theta_{n} \leq \pi$ with $\theta_{n} \neq 0$ for all $n=1,\ldots,N$. 
In order to show the desired result, it suffices to show
\begin{align}\label{eq:05232246}
\sup_{L \geq1}
\int_{-\infty}^{\infty}
f_{L}(x) |\log{Z}(e^{-ix};U)|^{2}
\,dx
<
\infty
\end{align}
by Lemma \ref{lem:4.6}. 
Let $x \in \mathbb{R}$ with $x \neq \theta_{n}+k \pi$ for all $n=1,\ldots,N$ and $k \in \mathbb{Z}$. 
Then we deduce from \eqref{eq:05091441} that 
\begin{align*}
|\log{Z}(e^{-ix};U)|^{2}
&
\leq
\bigg\{
\sum_{n=1}^{N}
\Big|
\log \big|
1- e^{i(\theta_{n}-x)}
\big|
\Big|
+
O(N)
\bigg\}^{2}
\\
&
\ll
N
\sum_{n=1}^{N}
\Big(
\log \Big|\sin \frac{\theta_{n}-x}{2}\Big|
\Big)^{2}
+
N^{2}
\end{align*}
by the Cauchy--Schwarz inequality, where the implied constant is absolute.  
Hence the integral of the left-hand side of \eqref{eq:05232246} is estimated as
\begin{align}\label{eq:06040336}
&
\int_{-\infty}^{\infty}
f_{L}(x) |\log{Z}(e^{-ix};U)|^{2}
\,dx
\\
\nonumber
&
\ll
N 
\sum_{n=1}^{N}
\int_{-\infty}^{\infty}
f_{L}(x) 
\Big(
\log \Big|\sin \frac{\theta_{n}-x}{2}\Big|
\Big)^{2}
\,dx
+
N^{2}
\\
\nonumber
&
=
N 
\sum_{n=1}^{N}
\sum_{k \in \mathbb{Z}}
\int_{-\pi/2}^{\pi/2}
f_{L}(2y+\theta_{n}+2k \pi) 
\big(
\log |\sin y|
\big)^{2}
\,dy
+
N^{2}
\end{align}
with the absolute implied constant. 
Firstly, we consider the contribution of the terms for $|k| \leq 1$. 
For each $n \in \{1,\ldots,N\}$, we have
\begin{align*}
&
\int_{-\pi/2}^{\pi/2}
f_{L}(2y+\theta_{n}+2k \pi) 
\big(
\log |\sin y|
\big)^{2}
\,dy
\\
&
= 
\bigg(
\int_{-\pi/2}^{-|\theta_{n}|/3}
+
\int_{-|\theta_{n}|/3}^{|\theta_{n}|/3}
+
\int_{|\theta_{n}|/3}^{\pi/2}
\bigg)
f_{L}(2y+\theta_{n}+2k \pi) 
\big(
\log |\sin y|
\big)^{2}
\,dy
\\
&
\leq
\Big(
\log \Big|\sin \frac{\theta_{n}}{3}\Big|
\Big)^{2}
\int_{-\infty}^{\infty}
f_{L}(2y+\theta_{n}+2k \pi) 
\,dy
+
f_{L} \Big(\frac{|\theta_{n}|}{3}\Big)
\int_{-|\theta_{n}|/3}^{|\theta_{n}|/3}
\big(
\log |\sin y|
\big)^{2}
\,dy
\\
&
\ll
(\log{|\theta_{n}|})^{2}
\end{align*}
by using \eqref{eq:005252129} and the inequality $|\sin y| \geq (2/\pi) |y|$ which is valid for $|y| \leq \pi/2$. 
On the other hand, for $|k| \geq 2$, we have
\begin{align*}
\int_{-\pi/2}^{\pi/2}
f_{L}(2y+\theta_{n}+2k \pi) 
\big(
\log |\sin y|
\big)^{2}
\,dy
&
\leq
f_{L} \big( (2|k|-2) \pi \big) 
\int_{-\pi/2}^{\pi/2}
\big(
\log |\sin y|
\big)^{2}
\,dy
\\
&
\ll
k^{-2}
\end{align*}
by using \eqref{eq:005252130}. 
Combining these estimates, we obtain
\begin{align*}
\sum_{k \in \mathbb{Z}}
\int_{-\pi/2}^{\pi/2}
f_{L}(2y+\theta_{n}+2k \pi) 
\big(
\log |\sin y|
\big)^{2}
\,dy
\ll
(\log{|\theta_{n}|})^{2}
+
\sum_{|k| \geq 2} k^{-2}
<
\infty
\end{align*}
for each $n \in \{1,\ldots,N\}$, where the implied constant is absolute.  
Hence, by \eqref{eq:06040336}, we arrive at
\begin{align*}
\int_{-\infty}^{\infty}
f_{L}(x) |\log{Z}(e^{-ix};U)|^{2}
\,dx
\ll
N^{2} 
\max_{1 \leq n \leq N}
(\log{|\theta_{n}|})^{2}
+
N^{2}. 
\end{align*}
This completes the proof of \eqref{eq:05232246}. 
\end{proof}

In order to conclude this paper, we lastly prove the following propositions. 

\begin{proposition}\label{prop:4.8}
Assume RH. 
For any $\vec{\kappa}=(\kappa_{1},\kappa_{2}) \in \mathbb{C}^{2}$ with $\RE(\kappa_{1}+\kappa_{2})\geq0$, 
\begin{align*}
&
\lim_{L \to \infty}
\frac{1}{T}
\int_{0}^{T} 
\exp \Big(
\vec{\kappa} \circ
\int_{-\infty}^{\infty} 
f_{L}(x) \log^{*}{\zeta}(1/2+ix+it)
\,dx
\Big)
\,dt
\\
&
=
\frac{1}{T}
\int_{0}^{T} 
\exp \big(
\vec{\kappa} \circ \log{\zeta}(1/2+it)
\big)
\,dt,
\end{align*}
where $T$ is any positive real number, and $f_{L} \in \mathfrak{P}^{*}$ is defined as \eqref{eq:05170402}. 
The same result remains true even in the case where $-1<\RE(\kappa_{1}+\kappa_{2})<0$ if we assume further that all nontrivial zeros of $\zeta(s)$ are simple. 
\end{proposition}

\begin{proof}
By the dominated convergence theorem, the result follows if we find a Borel measurable function $M(t;\vec{\kappa})$ which is integrable over $[0,T]$ and satisfies
\begin{align}\label{eq:06040515}
\sup_{L \geq1}
\bigg|
\exp \Big(
\vec{\kappa} \circ
\int_{-\infty}^{\infty} 
f_{L}(x) \log^{*}{\zeta}(1/2+ix+it)
\,dx
\Big)
\bigg|
\leq
M(t;\vec{\kappa})
\end{align}
for $t \in [0,T]$ almost everywhere.
Note that we have
\begin{align}\label{eq:06041340}
\big|
\exp (\vec{\kappa} \circ \lambda)
\big|
&
=
\exp \big( 
\RE(\kappa_{1} \lambda
+ 
\kappa_{2} \overline{\lambda})
\big)
\\
\nonumber
&
=
\exp \big( 
\RE(\kappa_{1}+\kappa_{2}) \RE(\lambda)
-
\IM(\kappa_{1}-\kappa_{2}) \IM(\lambda)
\big)
\end{align}
for any $\vec{\kappa}=(\kappa_{1},\kappa_{2}) \in \mathbb{C}^{2}$ and $\lambda \in \mathbb{C}$. 
Let $x \in \mathbb{R}$ with $x+t \neq \gamma$ for all $\gamma$. 
Then we deduce from \eqref{eq:06011532} that
\begin{align*}
\log^{*}{\zeta}(1/2+ix+it)
=
\sum_{ |\gamma-x-t| \leq 1}
\log(\gamma-x-t) 
+
O \big(\log(|x+t|+2)\big)
\end{align*}
with the absolute implied constant. 
From the above, we derive
\begin{align}\label{eq:06040532}
&
\bigg|
\exp \Big(
\vec{\kappa} \circ
\int_{-\infty}^{\infty} 
f_{L}(x) \log^{*}{\zeta}(1/2+ix+it)
\,dx
\Big)
\bigg|
\\
\nonumber
&
\leq
\exp \bigg(
\RE(\kappa_{1}+\kappa_{2})
\int_{-\infty}^{\infty} 
f_{L}(x)
\sum_{ |\gamma-x-t| \leq 1}
\log|\gamma-x-t| 
\,dx
\bigg)
\\
\nonumber
&
\quad
\times
\exp \bigg(
K (|\kappa_{1}|+|\kappa_{2}|)
\int_{-\infty}^{\infty} 
f_{L}(x) \log(|x+t|+2)
\,dx
\bigg)
\\
\nonumber
&
=
C_{L}(t;\vec{\kappa})
D_{L}(t;\vec{\kappa}),
\end{align}
say, where $K>0$ is an absolute constant. 
By changing the variables, we obtain
\begin{align*}
D_{L}(t;\vec{\kappa})
&
=
\exp \bigg(
K (|\kappa_{1}|+|\kappa_{2}|)
\int_{-\infty}^{\infty}
f_{1}(x) \log \Big( \Big|\frac{x}{\sqrt{L}}+t \Big| +2 \Big)
\,dx
\bigg)
\\
&
\leq
\exp \bigg(
K (|\kappa_{1}|+|\kappa_{2}|)
\int_{-\infty}^{\infty}
f_{1}(x) \log(|x|+T+2)
\,dx
\bigg)
\\
&
=
D(\vec{\kappa},T), 
\end{align*}
say, for any $t \in [0,T]$ and $L \geq1$. 
Then we evaluate $C_{L}(t;\vec{\kappa})$ as follows. 
Note that 
\begin{align*}
\int_{-\infty}^{\infty} 
f_{L}(x)
\sum_{ |\gamma-x-t| \leq 1}
\log|\gamma-x-t| 
\,dx
\leq 
0
\end{align*}
due to $f_{L}(x)>0$ and $\log|\gamma-x-t| \leq 0$ for $|\gamma-x-t| \leq 1$. 
Thus, in the case where $\RE(\kappa_{1}+\kappa_{2}) \geq0$, we have $C_{L}(t;\vec{\kappa}) \leq 1$ for any $t \in [0,T]$ and $L \geq1$. 
By \eqref{eq:06040532}, we obtain \eqref{eq:06040515} for any $t \in [0,T]$ with the function $M(t;\vec{\kappa})=D(\vec{\kappa},T)$, which is clearly integrable over $[0,T]$. 
The matter is the case where $-1<\RE(\kappa_{1}+\kappa_{2}) <0$. 
Let $0<\epsilon<1/2$ be a parameter chosen later. 
As in the proof of Lemma \ref{lem:4.7}, we derive the identity
\begin{align*}
\int_{-\infty}^{\infty} 
f_{L}(x)
\sum_{ |\gamma-x-t| \leq 1}
\log|\gamma-x-t| 
\,dx
=
\sum_{\gamma}
\int_{-1}^{1} 
f_{L}(y+\gamma-t)
\log|y| 
\,dy. 
\end{align*}
Assume $t \neq \gamma$ for all $\gamma$. 
For each $\gamma$ such that $|t-\gamma| \leq 2$, we obtain
\begin{align*}
0
&
\geq
\int_{-1}^{1}
f_{L}(y+\gamma-t) 
\log|y|
\,dy
\\
&
=
\bigg(
\int_{-1}^{-\epsilon |t-\gamma|}
+
\int_{-\epsilon|t-\gamma|}^{\epsilon|t-\gamma|}
+
\int_{\epsilon|t-\gamma|}^{1}
\bigg)
f_{L}(y+\gamma-t) 
\log|y|
\,dy
\\
&
\geq
\log(\epsilon|t-\gamma|)
\int_{-\infty}^{\infty}
f_{L}(y+\gamma-t) 
\,dy
+
f_{L} \big((1-\epsilon)|t-\gamma| \big) 
\int_{-\epsilon|t-\gamma|}^{\epsilon|t-\gamma|}
\log|y|
\,dy
\\
&
\geq
\frac{1+\epsilon}{1-\epsilon}
\log{|t-\gamma|}
+
O(1)
\end{align*}
by using \eqref{eq:005252129}, where the implied constant depends only on $\epsilon$. 
On the other hand, for each $\gamma$ such that $|t-\gamma|>2$, we obtain 
\begin{align*}
0
\geq
\int_{-1}^{1}
f_{L}(y+\gamma-t) 
\log|y|
\,dy
&
\geq
f_{L}(|t-\gamma|-1)
\int_{-1}^{1}
\log{|y|}
\,dy
\\
&
\geq
-8(t-\gamma)^{-2}
\end{align*}
by using \eqref{eq:005252130}. 
By these estimates, $C_{L}(t;\vec{\kappa})$ is evaluated as
\begin{align}\label{eq:06040533}
C_{L}(t;\vec{\kappa})
&
\leq
\exp \bigg(
\frac{1+\epsilon}{1-\epsilon}
\RE(\kappa_{1}+\kappa_{2})
\sum_{|t-\gamma| \leq 2}
\log{|t-\gamma|}
\bigg)
\\
\nonumber
&
\quad
\times
\exp \bigg(
K_{\epsilon} 
(|\kappa_{1}|+|\kappa_{2}|)
\Big\{
\sum_{|t-\gamma| \leq 2} 1
+
\sum_{|t-\gamma|>2}
(t-\gamma)^{-2}
\Big\}
\bigg)
\\
\nonumber
&
=
C(t;\vec{\kappa})
E(t;\vec{\kappa}),
\end{align}
say, where $K_{\epsilon}$ is a positive constant depending only on $\epsilon$. 
We have 
\begin{align*}
&
\sum_{|t-\gamma| \leq 2} 1
\ll
\log(|t|+2)
\leq
\log(T+2), 
\\
&
\sum_{|t-\gamma|>2}
(t-\gamma)^{-2}
\leq
\sum_{|\gamma|>2T+2}
\frac{4}{|\gamma|^2}
+
\sum_{|\gamma| \leq 2T+2}
\frac{1}{4}
\ll
T \log(T+2) 
\end{align*}
for any $t \in [0,T]$ with the absolute implied constants. 
Hence $E(t;\vec{\kappa})$ is evaluated as
\begin{align*}
E(t;\vec{\kappa})
\leq
E(\vec{\kappa},T)
:=
\exp \bigg(
K_{\epsilon} M
(|\kappa_{1}|+|\kappa_{2}|)
T \log(T+2) 
\bigg),
\end{align*}
where $M$ is an absolute positive constant. 
By \eqref{eq:06040532} and \eqref{eq:06040533}, we obtain \eqref{eq:06040515} for any $t \in [0,T]$ with $t \neq \gamma$ for all $\gamma$, where $M(t;\vec{\kappa})$ is taken as
\begin{align*}
M(t;\vec{\kappa})
=
C(t;\vec{\kappa})D(\vec{\kappa},T)E(\vec{\kappa},T).
\end{align*}
Finally, we confirm that $C(t;\vec{\kappa})$, and thus $M(t;\vec{\kappa})$, is integrable over $[0,T]$. 
Recalling that $-1<\RE(\kappa_{1}+\kappa_{2}) <0$, we choose the parameter $0<\epsilon<1/2$ so that
\begin{align*}
\mu
:=
\frac{1+\epsilon}{1-\epsilon}
\RE(\kappa_{1}+\kappa_{2})
>
-1
\end{align*}
is satisfied. 
Then we have
\begin{align*}
\int_{0}^{T}
C(t;\vec{\kappa})
\,dt
=
\int_{0}^{T}
\prod_{|t-\gamma| \leq 2} |t-\gamma|^{\mu}
\,dt
\leq
\sum_{0<\gamma_{n} \leq T}
\int_{\frac{\gamma_{n-1}+\gamma_{n}}{2}}^{\frac{\gamma_{n}+\gamma_{n+1}}{2}}
\prod_{|t-\gamma| \leq 2} |t-\gamma|^{\mu}
\,dt
\end{align*}
where $\gamma_{n}$ denotes the $n$th positive ordinate of the nontrivial zeros of $\zeta(s)$, and we put $\gamma_{0}=-\gamma_{1}$ for convenience. 
Since we assume that all nontrivial zeros are simple, $\gamma_{n}<\gamma_{n+1}$ for all $n \geq1$. 
Then, for any $t \in [\frac{\gamma_{n-1}+\gamma_{n}}{2}, \frac{\gamma_{n}+\gamma_{n+1}}{2}]$ and $\gamma \neq \gamma_{n}$, we obtain the inequality $|t-\gamma|\geq\delta_{T}$ with 
\begin{align*}
\delta_{T}
=
\min_{0<\gamma_{n} \leq T} \frac{\gamma_{n+1}-\gamma_{n}}{2}. 
\end{align*}
Therefore, we conclude that
\begin{align*}
\int_{0}^{T}
C(t;\vec{\kappa})
\,dt
\ll
\sum_{0<\gamma_{n} \leq T}
\int_{\frac{\gamma_{n-1}+\gamma_{n}}{2}}^{\frac{\gamma_{n}+\gamma_{n+1}}{2}}
|t-\gamma_{n}|^{\mu}
\,dt
<
\infty
\end{align*}
due to $\mu>-1$. 
This yields that the function $M(t;\vec{\kappa})$ is also integrable over $[0,T]$, and we complete the proof. 
\end{proof}

\begin{proposition}\label{prop:4.9}
For any $\vec{\kappa}=(\kappa_{1},\kappa_{2}) \in \mathbb{C}^{2}$ with $\RE(\kappa_{1}+\kappa_{2})>-1$, 
\begin{align*}
&
\lim_{L \to \infty}
\int_{\mathcal{U}(N)}
\exp \Big(
\vec{\kappa} \circ
\int_{-\infty}^{\infty} 
f_{L}(x) \log{Z}(e^{-ix};U)
\,dx
\Big)
\,dU
\\
&
=
\int_{\mathcal{U}(N)}
\exp \big(
\vec{\kappa} \circ \log{Z}(1;U)
\big)
\,dU, 
\end{align*}
where $N$ is any positive integer, and $f_{L} \in \mathfrak{P}^{*}$ is defined as \eqref{eq:05170402}. 
\end{proposition}

\begin{proof}
The proof is similar to that of Proposition \ref{prop:4.8}, and it suffices to find a Borel measurable function $M(U;\vec{\kappa})$ which is integrable over $\mathcal{U}(N)$ and satisfies
\begin{align}\label{eq:06041331}
\sup_{L \geq1}
\bigg|
\exp \Big(
\vec{\kappa} \circ
\int_{-\infty}^{\infty} 
f_{L}(x) \log{Z}(e^{-ix};U)
\,dx
\Big)
\bigg|
\leq
M(U;\vec{\kappa})
\end{align}
for $U \in \mathcal{U}(N)$ almost everywhere. 
Let $x \in \mathbb{R}$ with $x \neq \theta_{n}+k \pi$ for all $n=1,\ldots,N$ and $k \in \mathbb{Z}$. 
By \eqref{eq:05091441}, we obtain
\begin{align*}
\log{Z}(e^{-ix};U)
&
=
\sum_{n=1}^{N}
\log \big|
1- e^{i(\theta_{n}-x)}
\big|
+
O(N)
\\
&
=
\sum_{n=1}^{N}
\log \Big|
\sin \frac{\theta_{n}-x}{2}
\Big|
+
O(N),
\end{align*}
where the implied constants are absolute. 
Thus we apply \eqref{eq:06041340} to derive
\begin{align*}
&
\bigg|
\exp \Big(
\vec{\kappa} \circ
\int_{-\infty}^{\infty} 
f_{L}(x) \log{Z}(e^{-ix};U)
\,dx
\Big)
\bigg|
\\
&
\leq
\exp \bigg(
\RE(\kappa_{1}+\kappa_{2})
\int_{-\infty}^{\infty} 
f_{L}(x)
\sum_{n=1}^{N}
\log \Big|
\sin \frac{\theta_{n}-x}{2}
\Big|
\,dx
\bigg)
\exp \bigg(
KN (|\kappa_{1}|+|\kappa_{2}|)
\bigg),
\end{align*}
where $K>0$ is an absolute constant. 
Noting that
\begin{align*}
\int_{-\infty}^{\infty} 
f_{L}(x)
\sum_{n=1}^{N}
\log \Big|
\sin \frac{\theta_{n}-x}{2}
\Big|
\,dx
\leq
0,
\end{align*}
we obtain \eqref{eq:06041331} for any $U \in \mathcal{U}(N)$ in the case where $\RE(\kappa_{1}+\kappa_{2}) \geq0$. 
Hence we consider the case where $-1<\RE(\kappa_{1}+\kappa_{2})<0$ in what follows. 
Assume $-\pi<\theta_{n} \leq \pi$ with $\theta_{n} \neq 0$ for all $n=1,\ldots,N$. 
Then we have 
\begin{align*}
\int_{-\infty}^{\infty} 
f_{L}(x)
\sum_{n=1}^{N}
\log \Big|
\sin \frac{\theta_{n}-x}{2}
\Big|
\,dx
=
\sum_{n=1}^{N}
\sum_{k \in \mathbb{Z}}
\int_{-\pi/2}^{\pi/2} 
f_{L}(2y+\theta_{n}+2k \pi)
\log |\sin y|
\,dy. 
\end{align*}
Let $0<\epsilon<1/2$ be a parameter chosen later. 
For $|k| \leq 1$, we obtain
\begin{align*}
0
&
\geq
\int_{-\pi/2}^{\pi/2} 
f_{L}(2y+\theta_{n}+2k \pi)
\log |\sin y|
\,dy
\\
&
=
\bigg(
\int_{-\pi/2}^{-\epsilon|\theta_{n}|}
+
\int_{-\epsilon|\theta_{n}|}^{\epsilon|\theta_{n}|}
+
\int_{\epsilon|\theta_{n}|}^{\pi/2}
\bigg)
f_{L}(2y+\theta_{n}+2k \pi) 
\log |\sin y|
\,dy
\\
&
\geq
\log \big|\sin (\epsilon \theta_{n})\big|
\int_{-\infty}^{\infty}
f_{L}(2y+\theta_{n}+2k \pi) 
\,dy
+
f_{L} \big((1-2\epsilon)|\theta_{n}|\big)
\int_{-\epsilon|\theta_{n}|}^{\epsilon|\theta_{n}|}
\log |\sin y|
\,dy
\\
&
\geq
\frac{1}{1-2\epsilon} 
\log{|\theta_{n}|}
+
O(1)
\end{align*}
by using \eqref{eq:005252129}, where the implied constant depends only on $\epsilon$. 
On the other hand, for $|k| \geq 2$, we have
\begin{align*}
0
\geq
\int_{-\pi/2}^{\pi/2}
f_{L}(2y+\theta_{n}+2k \pi) 
\log |\sin y|
\,dy
&
\geq
f_{L} \big( (2|k|-2) \pi \big) 
\int_{-\pi/2}^{\pi/2}
\log |\sin y|
\,dy
\\
&
\geq
-C k^{-2}
\end{align*}
by using \eqref{eq:005252130}, where $C>0$ is an absolute constant. 
Combining these estimate, we derive the upper bound
\begin{align*}
&
\exp \bigg(
\RE(\kappa_{1}+\kappa_{2})
\int_{-\infty}^{\infty} 
f_{L}(x)
\sum_{n=1}^{N}
\log \Big|
\sin \frac{\theta_{n}-x}{2}
\Big|
\,dx
\bigg)
\\
&
\leq
\exp \bigg(
\frac{1}{1-2\epsilon} 
\RE(\kappa_{1}+\kappa_{2})
\sum_{n=1}^{N}
\log{|\theta_{n}|}
\bigg)
\exp \bigg(
C_{\epsilon}N
(|\kappa_{1}|+|\kappa_{2}|)
\bigg), 
\end{align*}
where $C_{\epsilon}$ is a positive constant depending only on $\epsilon$. 
As a result, we obtain \eqref{eq:06041331} for any $U \in \mathcal{U}(N)$ with $\theta_{n} \neq 0$ for all $n=1,\ldots,N$, where $M(U;\vec{\kappa})$ is taken as 
\begin{align*}
M(U;\vec{\kappa})
=
\exp \bigg(
\frac{1}{1-2\epsilon} 
\RE(\kappa_{1}+\kappa_{2})
\sum_{n=1}^{N}
\log{|\theta_{n}|}
\bigg)
\exp \bigg(
(K+C_{\epsilon}) N
(|\kappa_{1}|+|\kappa_{2}|)
\bigg).
\end{align*}
Now we choose the parameter $0<\epsilon<1/2$ so that
\begin{align*}
\nu
:=
\frac{1}{1-2\epsilon}
\RE(\kappa_{1}+\kappa_{2})
>
-1
\end{align*}
is satisfied. 
Then we have
\begin{align*}
&
\int_{\mathcal{U}(N)}
\exp \bigg(
\frac{1}{1-2\epsilon} 
\RE(\kappa_{1}+\kappa_{2})
\sum_{n=1}^{N}
\log{|\theta_{n}|}
\bigg)
\,dU
\\
&
=
\int_{\mathcal{U}(N)}
|\theta_{1} \cdots \theta_{N}|^{\nu}
\,dU
\\
&
=
\frac{1}{(2\pi)^{N} N!}
\int_{-\pi}^{\pi}
\cdots
\int_{-\pi}^{\pi}
|\theta_{1} \cdots \theta_{N}|^{\nu}
\prod_{1 \leq j<m \leq N}
\big|e^{i \theta_{j}}-e^{i \theta_{m}}\big|
\,d \theta_{1} \cdots d \theta_{N}
<
\infty
\end{align*}
due to $\nu>-1$, where we used the joint density of $\theta_{1},\ldots,\theta_{N}$ as in \cite{KeatingSnaith2000a}. 
Therefore, the function $M(U;\vec{\kappa})$ is integrable over $\mathcal{U}(N)$, and we complete the proof. 
\end{proof}

From the above, we finally obtain the main result of this paper. 

\begin{proof}[Proof of Theorem \ref{thm:1.2}]
We begin with showing the identities
\begin{align}
\label{eq:06041440}
\big\langle
-f_{L}
\mid
\log{\zeta}(1/2+0+ix+it)
\big\rangle
&
=
\int_{-\infty}^{\infty} 
f_{L}(x) \log^{*}{\zeta}(1/2+ix+it)
\,dx, 
\\
\label{eq:06041441}
\big\langle
-f_{L}
\mid
\log{Z}(e^{-ix-0};U)
\big\rangle
&
=
\int_{-\infty}^{\infty} 
f_{L}(x) \log{Z}(e^{-ix};U)
\,dx
\end{align}
for $t \in \mathbb{R}$ and $U \in \mathcal{U}(N)$. 
Let $(c_{m})$ be any sequence of positive real numbers such that $0<c_{m} \leq 1/4$ and $c_{m} \to 0$ as $m \to \infty$. 
We have 
\begin{align*}
&
\big\langle
-f_{L}
\mid
\log{\zeta}(1/2+0+ix+it)
\big\rangle
\\
&
=
\int_{-\infty}^{\infty} 
f_{L}(x-ic_{m}) \log^{*}{\zeta}(1/2+c_{m}+ix+it)
\,dx
\end{align*}
for all $m \geq1$. 
Then we deduce from \eqref{eq:06011532} that
\begin{align}\label{eq:05230109}
&
|\log^{*}{\zeta}(1/2+c_{m}+ix+it)|
\\
\nonumber
&
\leq
\sum_{ |\gamma-x-t| \leq 1}
\big| \log|c_{m}+i(x+t-\gamma)| \big|
+
O \big(\log(|x+t|+2)\big),
\end{align}
where the implied constant is absolute. 
Assume $0<|\gamma-x-t| \leq 1$. 
Then we derive
\begin{align*}
\big| \log|c_{m}+i(x+t-\gamma)| \big|
\leq
\big	| \log|\gamma-x-t| \big|
+
O(1)
\end{align*}
for all $m \geq1$, where the implied constant is absolute. 
Furthermore, by the definition of $f_{L}(z)$, we have $f_{L}(x-ic_{m}) \ll f_{L}(x)$ for all $m \geq1$ with the implied constant depending only on $L$.
Hence we obtain 
\begin{align*}
&
f_{L}(x-ic_{m})\log^{*}{\zeta}(1/2+c_{m}+ix+it)
\\
&
\ll
f_{L}(x)
\bigg\{
\sum_{ |\gamma-x-t| \leq 1}
\big	| \log|\gamma-x-t| \big|
+
\log(|x+t|+2)
\bigg\}
\end{align*}
for all $m \geq1$, where the implied constant depends only on $L$.  
As seen in the proof of Lemma \ref{lem:4.7}, the function of the right-hand side is integrable over $\mathbb{R}$. 
As a result, the dominated convergence theorem yields that $f_{L}(x) \log^{*}{\zeta}(1/2+ix+it)$ is also integrable over $\mathbb{R}$, and furthermore, 
\begin{align*}
&
\big\langle
-f_{L}
\mid
\log{\zeta}(1/2+0+ix+it)
\big\rangle
\\
&
=
\lim_{m \to \infty}
\int_{-\infty}^{\infty} 
f_{L}(x-ic_{m}) \log^{*}{\zeta}(1/2+c_{m}+ix+it)
\,dx
\\
&
=
\int_{-\infty}^{\infty} 
f_{L}(x) \log^{*}{\zeta}(1/2+ix+it)
\,dx. 
\end{align*}
We omit the proof of \eqref{eq:06041441} since it can be shown in a quite similar way. 
Now we are ready to complete the proof. 
By \eqref{eq:06041440} and Corollary \ref{cor:3.9}, we obtain
\begin{align*}
&
\lim_{T \to \infty} 
\frac{1}{T}
\int_{0}^{T} 
\exp \Big(
\vec{\kappa} \circ
\int_{-\infty}^{\infty} 
f_{L}(x) \log^{*}{\zeta}(1/2+ix+it)
\,dx
\Big)
\,dt
\\
&
=
\mathbb{E} \Big[
\exp \Big(
\vec{\kappa} \circ
\big\langle
-f_{L}
\mid
\log{\zeta}^{\mathsf{rand}}(1/2+0+ix)
\big\rangle
\Big)
\Big]. 
\end{align*}
On the other hand, by \eqref{eq:06041441} and Corollary \ref{cor:3.11}, we also obtain
\begin{align*}
&
\lim_{T \to \infty}
\int_{\mathcal{U}(N_{T})}
\exp \Big(
\vec{\kappa} \circ
\int_{-\infty}^{\infty} 
f_{L}(x) \log{Z}(e^{-ix};U)
\,dx
\Big)
\,dU
\\
&
=
\mathbb{E} \Big[
\exp \Big(
\vec{\kappa} \circ
\big\langle
-f_{L}
\mid
\log{Z}^{\mathsf{Gauss}}(e^{-ix-0})
\big\rangle
\Big)
\Big], 
\end{align*}
where $N_{T}$ is any positive integer with $N_{T} \to \infty$ as $T \to \infty$. 
Hence \eqref{eq:03222030} follows from Theorem \ref{thm:4.1}. 
The remaining work it to prove the latter part of the theorem. 
Let $N_{T}/\log{T} \to 1$ as $T \to \infty$. 
Now we assume that the order of the limits in \eqref{eq:03222030} is commutative, that is, 
\begin{align}\label{eq:05240208}
\lim_{T \to \infty}
\lim_{L \to \infty}
\frac{
\displaystyle{
\frac{1}{T}
\int_{0}^{T} 
\exp \Big(
\vec{\kappa} \circ
\int_{-\infty}^{\infty} 
f_{L}(x) \log^{*}{\zeta}(1/2+ix+it)
\,dx
\Big)
\,dt
}}{
\displaystyle{
\int_{\mathcal{U}(N_{T})}
\exp \Big(
\vec{\kappa} \circ
\int_{-\infty}^{\infty} 
f_{L}(x) \log{Z}(e^{-ix};U)
\,dx
\Big)
\,dU
}}
=
a(\vec{\kappa})
\end{align}
is valid. 
By Propositions \ref{prop:4.8} and \ref{prop:4.9}, we deduce from \eqref{eq:05240208} that \eqref{eq:05240210} holds with $\theta=0$, which is equivalent to \eqref{eq:03211910} as remarked in Section \ref{sec:1}. 
Then, we assume contrarily that \eqref{eq:03211910} is valid. 
It yields \eqref{eq:05240210} with $\theta=0$. 
Hence we obtain \eqref{eq:05240208} by Propositions \ref{prop:4.8} and \ref{prop:4.9}, which completes the proof. 
\end{proof}

\renewcommand{\thesection}{\Alph{section}}
\setcounter{section}{1}
\setcounter{theorem}{0}
\setcounter{equation}{0}

\section*{Appendix. Fourier hyperfunctions of one variable}
Let $D_{\delta}$ be the strip $|\IM(z)|<\delta$ on the complex plane for $\delta>0$. 
We define $\mathfrak{P}_{\delta}$ as the linear space of holomorphic functions $f$ on $D_{\delta}$ such that
\begin{align*}
\sup_{|\IM(z)| \leq \delta_{1}}
|f(z)| e^{-\epsilon |z|}
<
\infty
\end{align*}
for any $\epsilon>0$ and $0<\delta_{1}<\delta$. 
We also define $\widetilde{\mathfrak{P}}_{\delta}$ as the linear space of holomorphic functions $\varphi$ on $D_{\delta} \setminus \mathbb{R}$ such that
\begin{align*}
\sup_{\delta_{1} \leq |\IM(z)| \leq \delta_{2}}
|\varphi(z)| e^{-\epsilon |z|}
<
\infty
\end{align*}
for any $\epsilon>0$ and $0<\delta_{1}<\delta_{2}<\delta$. 
Denote by $\sigma_{\delta,\delta'}: \widetilde{\mathfrak{P}}_{\delta} \to \widetilde{\mathfrak{P}}_{\delta'}$ the restriction map $\varphi \mapsto \varphi_{D_{\delta'} \setminus \mathbb{R}}$. 
Furthermore, we regard $\mathfrak{P}_{\delta}$ as a subspace of $\widetilde{\mathfrak{P}}_{\delta}$ by the injective linear map $\mathfrak{P}_{\delta} \to \widetilde{\mathfrak{P}}_{\delta}$ defined by the restriction $f \mapsto f \mid_{D_{\delta}\setminus \mathbb{R}}$. 
Define $\mathfrak{Q}_{\delta}=\widetilde{\mathfrak{P}}_{\delta}/\mathfrak{P}_{\delta}$ for $\delta>0$. 
For any $0<\delta'<\delta$, there exists a well-defined linear map $\rho_{\delta,\delta'}: \mathfrak{Q}_{\delta} \to \mathfrak{Q}_{\delta'}$ with the following commutative diagram 
\begin{align*}
\xymatrix{
\widetilde{\mathfrak{P}}_{\delta} \ar[r]^{\sigma_{\delta,\delta'}} \ar[d] & \widetilde{\mathfrak{P}}_{\delta'} \ar[d] \\
\mathfrak{Q}_{\delta} \ar[r]^{\rho_{\delta,\delta'}} & \mathfrak{Q}_{\delta'}
}
\end{align*}
where the vertical maps denote the natural surjections. 

\begin{lemma}\label{lem:A.1}
The linear map $\rho_{\delta,\delta'}: \mathfrak{Q}_{\delta} \to \mathfrak{Q}_{\delta'}$ is bijective for any $0<\delta'<\delta$. 
\end{lemma}

\begin{proof}
To see the injectivity, we take an element $\varphi \in \widetilde{\mathfrak{P}}_{\delta}$ such that $\varphi(z)=f(z)$ for any $z \in  D_{\delta'}\setminus \mathbb{R}$ with some $f \in \mathfrak{P}_{\delta'}$. 
Then the function 
\begin{align*}
f_{1}(z)
=
\begin{cases}
f(z)
&
\text{for $|\IM(z)|<\delta'$}, 
\\
\varphi(z)
&
\text{for $\delta' \leq |\IM(z)|<\delta$}
\end{cases}
\end{align*}
belongs to $\mathfrak{P}_{\delta}$. 
Furthermore, we obtain $\varphi(z)=f_{1}(z)$ for any $z \in D_{\delta}\setminus \mathbb{R}$. 
This means that $\rho_{\delta,\delta'}(\varphi \bmod \mathfrak{P}_{\delta})=0$ implies $\varphi \bmod \mathfrak{P}_{\delta}=0$. 
Then, the surjectivity of the map $\rho_{\delta,\delta'}$ is equivalent that for any $\varphi \in \widetilde{\mathfrak{P}}_{\delta'}$, there exist $\varphi_{1} \in \widetilde{\mathfrak{P}}_{\delta}$ and $f \in \mathfrak{P}_{\delta'}$ such that $\varphi(z)+f(z)=\varphi_{1}(z)$ for any $z \in D_{\delta'}\setminus \mathbb{R}$. 
This can be shown by constructing $\varphi_{1}$ and $f$ as follows. 
For any $z \in \mathbb{C}$ with $0<\IM(z)<\delta'$, we take real numbers $c$ and $d$ so that $0<c<\IM(z)<d<\delta'$. 
By shifting the contour, we obtain
\begin{align*}
\varphi(z)
&
=
\frac{1}{2\pi i}
\int_{-\infty+ic}^{\infty+ic}
\varphi(w)
e^{-(w-z)^{2}}
\,\frac{dw}{w-z}
-
\frac{1}{2\pi i}
\int_{-\infty+id}^{\infty+id}
\varphi(w)
e^{-(w-z)^{2}}
\,\frac{dw}{w-z}
\\
&
=
I_{c}(z)
-
J_{d}(z), 
\end{align*}
say. 
Note that $I_{c}(z)$ is holomorphic for $\IM(z)>0$, and that $J_{d}(z)$ is holomorphic for $\IM(z)<\delta'$. 
In a similar way, for any $z \in \mathbb{C}$ with $-\delta<\IM(z)<0$, we take real numbers $c$ and $d$ so that $0<-d<\IM(z)<-c<\delta$. 
Then we obtain
\begin{align*}
\varphi(z)
=
I_{-d}(z)
-
J_{-c}(z), 
\end{align*}
where $I_{-d}(z)$ is holomorphic for $\IM(z)>-\delta'$, and $J_{-c}(z)$ is also holomorphic for $\IM(z)<0$. 
Then we define
\begin{align*}
\varphi_{1}(z)
=
\begin{cases}
I_{c}(z)
&
\text{for $0<\IM(z)<\delta$}, 
\\
-J_{-c}(z)
&
\text{for $-\delta<\IM(z)<0$},
\end{cases}
\end{align*}
and $f(z)=-J_{d}(z)=I_{-d}(z)$ for $-\delta'<\IM(z)<\delta'$. 
By their definitions, we can confirm that $\varphi(z)+f(z)=\varphi_{1}(z)$ for any $z \in D_{\delta'}\setminus \mathbb{R}$, and that $\varphi_{1} \in \widetilde{\mathfrak{P}}_{\delta}$ and $f \in \mathfrak{P}_{\delta'}$. 
As a result, the map $\rho_{\delta,\delta'}$ is surjective. 
\end{proof}

By Lemma \ref{lem:A.1}, we see that
\begin{align*}
\mathfrak{Q}
:=
\varinjlim
\mathfrak{Q}_{\delta} 
\simeq
\mathfrak{Q}_{\delta} 
\end{align*}
for any $\delta>0$, where the inductive limit is taken with respect to the above linear map $\rho_{\delta,\delta'}: \mathfrak{Q}_{\delta} \to \mathfrak{Q}_{\delta'}$ for $0<\delta'<\delta$. 
We refer to the elements of the space $\mathfrak{Q}$ as \textit{Fourier hyperfunctions} and write 
\begin{align*}
g
=
[\varphi] 
\end{align*}
to indicate that a Fourier hyperfunction $g$ is determined from $\varphi \in \widetilde{\mathfrak{P}}_{\delta}$ with some $\delta>0$. 
Then, we define the topological linear space $\mathfrak{P}^{*}$ as the inductive limit
\begin{align*}
\mathfrak{P}^{*}
=
\varinjlim
\mathfrak{B}_{m}^{*}
\end{align*}
with respect to the restriction map $\mathfrak{B}_{m}^{*} \to \mathfrak{B}_{m+1}^{*}$ for $m \geq1$, where $\mathfrak{B}_{m}^{*}$ is the Banach space of continuous functions $f$ on the strip $|\IM(z)| \leq 2^{-m}$ which are holomorphic in $|\IM(z)|<2^{-m}$ and satisfy
\begin{align*}
\| f \|_{m}
:=
\sup_{|\IM(z)| \leq 2^{-m}}
|f(z)| e^{|z|/m}
<
\infty. 
\end{align*}
Let $f \in \mathfrak{P}^{*}$ and $g \in \mathfrak{Q}$. 
Then we can take real numbers $\delta,\delta'>0$ such that $f$ is holomorphic on the strip $D_{\delta}$ and satisfies
\begin{align*}
\sup_{|\IM(z)| \leq \delta_{1}}
|f(z)| e^{\epsilon |z|}
<
\infty
\end{align*}
for any $0<\delta_{1}<\delta$ with some $\epsilon=\epsilon(\delta_{1})>0$, and that $g=[\varphi]$ with $\varphi \in \widetilde{\mathfrak{P}}_{\delta'}$. 
From the above, the pairing of $f$ and $g$ such that
\begin{align*}
\langle f \mid g \rangle
=
\int_{-\infty+ic}^{\infty+ic}
f(z) \varphi(z)
\,dz
-
\int_{-\infty-ic}^{\infty-ic}
f(z) \varphi(z)
\,dz 
\end{align*}
is well-defined, where $c$ is any constant with $0<c<\min(\delta_{1},\delta_{2})$. 
Denote by $(\mathfrak{P}^{*})'$ the dual space of $\mathfrak{P}^{*}$, that is, the linear space of all continuous linear functionals on $\mathfrak{P}^{*}$. 
In particular, if we define the map $T_{g}: \mathfrak{P}^{*} \to \mathbb{C}$ as $T_{g}(f)= \langle f \mid g \rangle$, then we obtain that $T_{g} \in (\mathfrak{P}^{*})'$ for any fixed $g \in \mathfrak{Q}$. 
The goal of this appendix is the following proposition which yields the duality $\mathfrak{Q} \simeq (\mathfrak{P}^{*})'$ as linear spaces. 

\begin{proposition}\label{prop:A.2}
The linear map $j: \mathfrak{Q} \to (\mathfrak{P}^{*})'$ defined by $g \mapsto T_{g}$ is bijective. 
\end{proposition}

In what follows, we construct the inverse map $k: (\mathfrak{P}^{*})' \to \mathfrak{Q}$ according to the method of Kawai \cite[Definition 3.2.8]{Kawai1970} using the notion of Fourier transformations in the space $\mathfrak{Q}$. 
Recall that the Fourier transform of $f \in \mathfrak{P}^{*}$ and its inversion are defined as
\begin{align*}
\mathcal{F} f(\xi)
=
\int_{-\infty}^{\infty}
f(x) e^{ix \xi}
\,dx 
\quad\text{and}\quad
\overline{\mathcal{F}} f(\xi)
=
\int_{-\infty}^{\infty}
f(x) e^{-ix \xi}
\,dx, 
\end{align*}
respectively, which are elements of the space $\mathfrak{P}^{*}$. 
From now on, we write as usual $(f,T)=T(f)$ with $f \in X$ and $T \in X'$ for a topological linear space $X$ and its dual space $X'$. 

\begin{definition}\label{def:A.3}
For $T \in (\mathfrak{P}^{*})'$, we define $\overline{\mathcal{F}}\, T \in (\mathfrak{P}^{*})'$ by $(\overline{\mathcal{F}}\, T, f)=(T, \overline{\mathcal{F}} f)$. 
\end{definition}

Then, we define the linear spaces $\mathfrak{P}_{+}^{*}$ and $\mathfrak{P}_{-}^{*}$ as the inductive limits
\begin{align*}
\mathfrak{P}_{+}^{*}
=
\varinjlim
\mathfrak{B}_{m,+}^{*}
\quad\text{and}\quad
\mathfrak{P}_{-}^{*}
=
\varinjlim
\mathfrak{B}_{m,-}^{*}
\end{align*}
with respect to the restriction maps $\mathfrak{B}_{m,+}^{*} \to \mathfrak{B}_{m+1,+}^{*}$ and $\mathfrak{B}_{m,-}^{*} \to \mathfrak{B}_{m+1,-}^{*}$ for $m \geq1$, where $\mathfrak{B}_{m,+}^{*}$ and $\mathfrak{B}_{m,-}^{*}$ are the following Banach spaces. 
Let $K_{m,+}$ and $K_{m,-}$ denote the subsets of $\mathbb{C}$ defined as
\begin{align*}
K_{m,+}
&
=
\{z \in \mathbb{C} \mid \text{$\RE(z) \geq -2^{-m}$ and $|\IM(z)| \leq 2^{-m}$} \}, 
\\
K_{m,-}
&
=
\{z \in \mathbb{C} \mid \text{$\RE(z) \leq 2^{-m}$ and $|\IM(z)| \leq 2^{-m}$} \} 
\end{align*}
for $m \geq1$. 
Then $\mathfrak{B}_{m,+}^{*}$ is the linear space consisting of all continuous functions $f$ on $K_{m,+}$ which are holomorphic in its interior and satisfy
\begin{align*}
\| f \|_{m,+}
:=
\sup_{z \in K_{m,+}}
|f(z)| e^{|z|/m}
<
\infty, 
\end{align*}
and $\mathfrak{B}_{m,-}^{*}$ is defined by replacing $+$ with $-$. 
For example, the function $\xi \mapsto e^{iz \xi}$ is an element of $\mathfrak{P}_{+}^{*}$ for $\IM(z)>0$ and is an element of $\mathfrak{P}_{-}^{*}$ for $\IM(z)<0$. 
Note that we have a natural injection
\begin{align*}
\mathfrak{P}^{*}
\to 
\mathfrak{P}_{+}^{*}
\oplus
\mathfrak{P}_{-}^{*}
\end{align*}
defined by the restrictions, and it can be shown that the image of this map is closed. 
Therefore, it induces the surjection
\begin{align*}
(\mathfrak{P}_{+}^{*})'
\oplus
(\mathfrak{P}_{-}^{*})'
\to 
(\mathfrak{P}^{*})'. 
\end{align*}
For $T \in (\mathfrak{P}^{*})'$, we take $T_{+} \in (\mathfrak{P}_{+}^{*})'$ and $T_{-} \in (\mathfrak{P}_{-}^{*})'$ according to this surjection so that we obtain the decomposition $T=T_{+}+T_{-}$. 

\begin{definition}\label{def:A.4}
For $T \in (\mathfrak{P}^{*})'$, we define $\mathcal{G}T \in \mathfrak{Q}$ as $\mathcal{G}\,T=[\varphi_{T_{+},T_{-}}]$, where
\begin{align*}
\varphi_{T_{+},T_{-}}(z)
=
\begin{cases}
(e^{iz \xi}, T_{+})
&
\text{for $0<\IM(z)<\delta$}, 
\\
-(e^{iz \xi}, T_{-})
&
\text{for $-\delta<\IM(z)<0$}
\end{cases}
\end{align*}
is an element of the space $\widetilde{\mathfrak{P}}_{\delta}$ for any $\delta>0$. 
\end{definition}

It should be confirmed that $\mathcal{G}\, T$ is independent of the way of the decomposition of $T$. 
Regard $\mathfrak{P}^{*}$ as a subspace of $\mathfrak{P}_{+}^{*}$ and $\mathfrak{P}_{-}^{*}$ by the restriction maps $\mathfrak{P}^{*} \to \mathfrak{P}_{+}^{*}$ and $\mathfrak{P}^{*} \to \mathfrak{P}_{-}^{*}$. 
Then it is known that $\mathfrak{P}^{*}$ is dense in $\mathfrak{P}_{+}^{*}$ and in $\mathfrak{P}_{-}^{*}$ as an analogue of Runge's theorem; see \cite[Theorem 2.2.1]{Kawai1970} for a proof. 
Hence the linear maps $(\mathfrak{P}_{+}^{*})' \to (\mathfrak{P}^{*})'$ and $(\mathfrak{P}_{-}^{*})' \to (\mathfrak{P}^{*})'$ defined by $T \mapsto T \mid_{\mathfrak{P}^{*}}$ are injective. 
Then we regard $(\mathfrak{P}_{+}^{*})'$ and $(\mathfrak{P}_{-}^{*})'$ as subspaces of $(\mathfrak{P}^{*})'$. 
Assume that we have another decomposition $T=S_{+}+S_{-}$ with $S_{+} \in (\mathfrak{P}_{+}^{*})'$ and $S_{-} \in (\mathfrak{P}_{-}^{*})'$. 
Then we obtain
\begin{align*}
T_{+}-S_{+}
=
S_{-}-T_{-}
\in 
(\mathfrak{P}_{+}^{*})' \cap (\mathfrak{P}_{-}^{*})'. 
\end{align*}
Hence we see that
\begin{align*}
\varphi_{T_{+},T_{-}}(z)-\varphi_{S_{+},S_{-}}(z)
&
=
\begin{cases}
(e^{iz \xi}, T_{+}-S_{+})
&
\text{for $0<\IM(z)<\delta$}, 
\\
(e^{iz \xi}, S_{-}-T_{-})
&
\text{for $-\delta<\IM(z)<0$}, 
\end{cases}
\\
&
=
\begin{cases}
(e^{iz \xi}, T_{+}-S_{+})
&
\text{for $0<\IM(z)<\delta$}, 
\\
(e^{iz \xi}, T_{+}-S_{+})
&
\text{for $-\delta<\IM(z)<0$}, 
\end{cases}
\end{align*}
which is continued to an element of $\mathfrak{P}_{\delta}$. 
Therefore, $\varphi_{T_{+},T_{-}}(z)$ and $\varphi_{S_{+},S_{-}}(z)$ define the same Fourier hyperfunction. 

\begin{definition}\label{def:A.5}
We define the map $k:(\mathfrak{P}^{*})' \to \mathfrak{Q}$ as $T \mapsto \mathcal{G}\, \overline{\mathcal{F}}\,T$. 
\end{definition}

\begin{proof}[Proof of Proposition \ref{prop:A.2}]
It suffices to show that $j \circ k =\mathrm{id}$ and that $j$ is injective. 
Let $T \in (\mathfrak{P}^{*})'$. 
For any $f \in \mathfrak{P}^{*}$, we have 
\begin{align*}
(f, (j \circ k) T) 
&
=
\langle f \mid \mathcal{G}\,(\overline{\mathcal{F}}\, T) \rangle
\\
&
=
\int_{-\infty+ic}^{\infty+ic}
f(z) (e^{iz \xi}, (\overline{\mathcal{F}}\, T)_{+})
\,dz
+
\int_{-\infty-ic}^{\infty-ic}
f(z) (e^{iz \xi}, (\overline{\mathcal{F}}\, T)_{-})
\,dz 
\\
&
=
\bigg(
\int_{-\infty+ic}^{\infty+ic}
f(z) e^{iz \xi}
\,dz, 
(\overline{\mathcal{F}}\, T)_{+}
\bigg)
+
\bigg(
\int_{-\infty-ic}^{\infty-ic}
f(z) e^{iz \xi}
\,dz, 
(\overline{\mathcal{F}}\, T)_{-}
\bigg)
\\
&
=
(\mathcal{F} f, (\overline{\mathcal{F}}\, T)_{+})
+
(\mathcal{F} f, (\overline{\mathcal{F}}\, T)_{-})
\\
&
=
(\mathcal{F} f, \overline{\mathcal{F}}\, T)
\\
&
=
(f, \mathcal{F} \overline{\mathcal{F}}\, T)
\\
&
=
(f,T).
\end{align*}
Hence $j \circ k =\mathrm{id}$ follows. 
In order to see the injectivity of the map $j: \mathfrak{Q} \to (\mathfrak{P}^{*})'$, we assume $T_{g}(f)= \langle f \mid g \rangle=0$ for all $f \in \mathfrak{P}^{*}$. 
Let $g=[\varphi]$ with $\varphi \in \widetilde{\mathfrak{P}}_{\delta}$ for some $\delta>0$. 
For any $z \in \mathbb{C}$ with $0<\IM(z)<\delta$, we take real numbers $c$ and $d$ so that $0<c<\IM(z)<d<\delta$. 
By shifting the contour, we obtain
\begin{align*}
\varphi(z)
=
\frac{1}{2\pi i}
\int_{-\infty+ic}^{\infty+ic}
\varphi(w)
e^{-(w-z)^{2}}
\,\frac{dw}{w-z}
-
\frac{1}{2\pi i}
\int_{-\infty+id}^{\infty+id}
\varphi(w)
e^{-(w-z)^{2}}
\,\frac{dw}{w-z}. 
\end{align*}
Furthermore, we have
\begin{align*}
\frac{1}{2\pi i}
\int_{-\infty+ic}^{\infty+ic}
\varphi(w)
e^{-(w-z)^{2}}
\,\frac{dw}{w-z}
-
\frac{1}{2\pi i}
\int_{-\infty-ic}^{\infty-ic}
\varphi(w)
e^{-(w-z)^{2}}
\,\frac{dw}{w-z}
=
0
\end{align*}
due to the assumption. 
Hence we arrive at
\begin{align*}
\varphi(z)
=
\frac{1}{2\pi i}
\int_{-\infty-ic}^{\infty-ic}
\varphi(w)
e^{-(w-z)^{2}}
\,\frac{dw}{w-z}
-
\frac{1}{2\pi i}
\int_{-\infty+id}^{\infty+id}
\varphi(w)
e^{-(w-z)^{2}}
\,\frac{dw}{w-z}. 
\end{align*}
The first integral of the right-hand side is holomorphic for $\IM(z)>-\delta$, and the second integral of the right-hand side is holomorphic for $\IM(z)<\delta$. 
Hence $\varphi(z)$ is continued to a holomorphic function on the strip $D_{\delta}$. 
Furthermore, we have
\begin{align*}
\frac{1}{2\pi i}
\int_{-\infty-ic}^{\infty-ic}
\varphi(w)
e^{-(w-z)^{2}}
\,\frac{dw}{w-z}
&
=
\frac{1}{2\pi i}
\int_{-\infty-ic-i \IM(z)}^{\infty-ic-i \IM(z)}
\varphi(w+z)
e^{-w^{2}}
\,\frac{dw}{w}
\\
&
\ll
e^{\epsilon |z|}
\int_{-\infty}^{\infty}
e^{\epsilon |v|} e^{-v^{2}}
\,\frac{dv}{v}
\\
&
\ll
e^{\epsilon |z|}
\end{align*}
on the strip $|\IM(z)| \leq \delta_{1}$ for any $0<\delta_{1}<\delta$ and $\epsilon>0$, where the implied constant depends only on $\delta_{1}$. 
We similarly obtain
\begin{align*}
\frac{1}{2\pi i}
\int_{-\infty+id}^{\infty+id}
\varphi(w)
e^{-(w-z)^{2}}
\,\frac{dw}{w-z} 
\ll
e^{\epsilon |z|}
\end{align*}
on the strip $|\IM(z)| \leq \delta_{1}$ for any $0<\delta_{1}<\delta$ and $\epsilon>0$, where the implied constant depends only on $\delta_{1}$. 
Hence the function $\varphi$ belongs to $\mathfrak{P}_{\delta}$, which shows $g=0$ in the space $\mathfrak{Q}$. 
Therefore $j: \mathfrak{Q} \to (\mathfrak{P}^{*})'$ is injective. 
This completes the proof. 
\end{proof}

\subsection*{Acknowledgments}
The author is grateful to Nguyen Le Dang Thi for informing him of the paper \cite{CaoTanigawaZhai2021}. 
The work of this paper was supported by JSPS Grant-in-Aid for Early-Career Scientists (Grant Number JP24K16906). 


\providecommand{\bysame}{\leavevmode\hbox to3em{\hrulefill}\thinspace}
\providecommand{\MR}{\relax\ifhmode\unskip\space\fi MR }
\providecommand{\MRhref}[2]{%
  \href{http://www.ams.org/mathscinet-getitem?mr=#1}{#2}
}
\providecommand{\href}[2]{#2}

\end{document}